\documentclass[reqno, a4paper]{amsart}

\title{An Affine invariant Minkowski problem}

\author{Antoine Ablondi}
\address{Antoine Ablondi: IMAG, Universit\'e de Montpellier, CNRS, Montpellier, France.} 
\email{antoine.ablondi@umontpellier.fr}
\date{}

\usepackage[margin=2.5cm]{geometry}
\usepackage{setspace}

\usepackage{enumitem}

\usepackage{amsmath,amssymb,amscd,amsthm,mathtools}
\usepackage[T1]{fontenc}
\usepackage[english]{babel}

\numberwithin{equation}{section}

\usepackage{enumitem}

\usepackage{graphicx}
\graphicspath{ {. /Figures/} }
\usepackage[justification=centering]{caption}
\usepackage[labelsep=period]{caption}
\numberwithin{figure}{section}

\usepackage{xcolor}
\definecolor{newred}{RGB}{200, 30, 45}
\definecolor{newyellow}{RGB}{255, 222, 23}
\definecolor{newblue}{RGB}{25, 48, 126}
\definecolor{newgrey}{RGB}{115, 110, 95}
\definecolor{newbrown}{RGB}{160, 62, 45}

\usepackage{hyperref}
\hypersetup{colorlinks, linkcolor={newred}, citecolor={newbrown}, 
 pdftitle={An Affine invariant Minkowski problem},
	pdfauthor={Antoine Ablondi}}
\usepackage{pdfpages}

\newtheorem{maintheorem}{Theorem}

\newtheorem{theorem}{Theorem}[section]
\newtheorem{corollary}[theorem]{Corollary}
\newtheorem{lemma}[theorem]{Lemma}
\newtheorem{proposition}[theorem]{Proposition}

\theoremstyle{definition}
\newtheorem{definition}[theorem]{Definition}
\newtheorem{example}[theorem]{Example}

\theoremstyle{remark}
\newtheorem{remark}[theorem]{Remark}

\newtheorem*{remark*}{Remark}

\newcommand{\R}{\mathbb{R}} 
\newcommand{\N}{\mathbb{N}} 
\newcommand{\V}{\mathbb{V}} 
\newcommand{\A}{\mathbb{A}} 
\renewcommand{\P}{\mathbb{P}} 

\newcommand{\E}{\mathbb{E}} 
\newcommand{\Sph}{\mathbb{S}} 
\newcommand{\M}{\mathbb{M}} 
\renewcommand{\H}{\mathbb{H}} 

\newcommand{\SL}{\mathrm{SL}} 
\newcommand{\SO}{\mathrm{SO}} 
\newcommand{\SA}{\mathrm{SA}} 

\newcommand{\Id}{\mathrm{Id}} 

\newcommand{\dif}{\mathrm{d}} 
\DeclareMathOperator{\grad}{grad} 
\DeclareMathOperator{\Hess}{Hess} 
\DeclareMathOperator{\Jac}{Jac} 

\DeclareMathOperator{\graph}{graph} 
\DeclareMathOperator{\epi}{epigraph} 

\newcommand{\C}{\mathcal{C}} 
\renewcommand{\S}{\Sigma_\C} 
\newcommand{\G}{\mathcal{G}} 

\newcommand{\DL}{\mathcal{D}} 

\newcommand{\linvol}{\mathit{vol}} 
\newcommand{\vol}{\mathrm{vol}} 
\newcommand{\dvol}{\dif \mathrm{vol}} 
\newcommand{\covol}{\mathrm{covol}} 

\newcommand{\Area}{\mathrm{A}} 
\newcommand{\MA}{\mathrm{MA}} 

\newcommand{\st}{\; \middle\vert \;} 

\newcommand{\namelessfunction}[4]{ \left\lbrace \begin{matrix}
{#1} &\longrightarrow& {#2}\\
{#3} & \longmapsto & {#4}
\end{matrix} \right. }
\newcommand{\function}[5]{{#1}: \namelessfunction{#2}{#3}{#4}{#5}}

\newcommand{\vp}[1]{{\left( {#1} \right)}} 
\newcommand{\bp}[1]{{\bigl( {#1} \bigr)}} 
\newcommand{\Bp}[1]{{\Bigl({#1}\Bigr)}} 

\begin{document}

\begin{abstract}

In Euclidean space, the generalised Minkowski problem asks, for a given finite Radon measure $\mu$ on the unit sphere $\Sph^d$, to find a compact convex set $K$ with area measure $\mu$. For convex sets in the Minkowski space invariant under an affine deformation of a uniform lattice of $\SO_0(d,1)$, the analogous Minkowski problem was considered and solved by Barbot--Béguin--Zeghib (partially) and Bonsante--Fillastre.
By a theorem of Mess--Barbot--Bonsante, that also solves the Minkowski problem in flat Lorentzian spacetimes with compact hyperbolic Cauchy surface. 

We consider convex domains of  the oriented real affine space $\R^{d+1}$ which are invariant under a subgroup of affine transformations obtained by adding translation parts to a discrete subgroup of $\SL (\R^{d+1})$ dividing a convex cone. 
We prove that those convex domains satisfy a local Steiner Formula, allowing to introduce natural area measures and define an affine invariant Minkowski problem. We then solve that Minkowski problem through a variational method using the convexity of a covolume functional.
We also give  an interpretation of those results in some ``affine spacetimes'', which were introduced by the author in a preceding work.

\end{abstract}
\maketitle

{\hypersetup{linkcolor=black}
\hypersetup{bookmarksdepth=2} \setcounter{tocdepth}{1} \tableofcontents}

\begin{spacing}{1.225}

\section*{Introduction}

\subsection*{The Minkowski problem in Euclidean space}

Let $d \geq 1$, let $\E^{d+1}$ denote the $(d+1)$-dimensional \emph{Euclidean affine space} and let $\Sph^d$ be the \emph{unit sphere} in the Euclidean vector space $(\V^{d+1}, \Vert\cdot\Vert_{euc})$ of translations of $\E^{d+1}$. Given a bounded convex domain $K$ of $\E^{d+1}$ and  a Borel subset $b$ of $\Sph^d$, the \emph{parallel $\varepsilon$-neighbourhood} $\mathcal{N}_\varepsilon(K)(b)$ consists of all the points $P \in \E^{d+1}$ at distance less then $\varepsilon$ from $\partial K$ having orthogonal projection $P'$ onto $\partial K$ such that $\overrightarrow{P'P}$ is collinear to a vector $\vec{n} \in b$. Notice that we have 
\begin{equation*}
\mathcal{N}_\varepsilon(K)(\Sph^{d}) = K +\varepsilon \Sph^{d} \, .
\end{equation*}
Fenchel and Jessen have proved \cite{Fenchel_Jessen_38} that the volume of such neighbourhoods satisfies a \emph{local Steiner Formula}: it is uniquely decomposed as
\begin{equation*}
\vol_{euc}\bp{\mathcal{N}_\varepsilon(K)(b)} = \frac{1}{d+1}\sum_{i=0}^{d}\varepsilon^{d+1-i} \binom{d+1}{i} S_i(K) (b)\, , 
\end{equation*}
where each $S_i(K)$ is a finite Radon measure on $\Sph^d$, called the \emph{area measure of order $i$} of $K$. The measure $S_0(K)$ is always $\vol_{\Sph^d}$ the volume on the Riemannian unit sphere $\Sph^d \subset (\V^{d+1}, \Vert\cdot\Vert_{euc})$.

We shall say that the $d$-th area measure $S_d(K)$ is \emph{‘‘the'' area measure} of $K$ and denote it by $\Area(K)$. Notice that it can equivalently be seen as some sort of derivative of the volume of parallel neighbourhoods of $K$:
\begin{equation*}
    \Area(K) = S_d(K) = \lim_{\varepsilon \to 0} \frac{\vol_{euc}\bp{\mathcal{N}_\varepsilon(K)}}{\varepsilon} \, .
\end{equation*}
The area measure generalises the notion of \emph{Gaussian curvature}: if $K$ is a bounded convex domain of $\E^{d+1}$ with a locally uniformly convex $C^2$ boundary, its area measure $\Area(K)$ is the measure on $\Sph^d$ of the form 
\begin{equation*}
 \Area(K) = \frac{1}{\phi} \, \vol_{\Sph^d} \, ,
\end{equation*}
where $\phi(\vec{n})$ is the Gaussian curvature of $\partial K$ at the point $P(\vec{n}) \in \partial K$ with exterior normal vector $\vec{n} \in \Sph^d$. 

The \emph{generalised Euclidean Minkowski problem} is the question of the existence of a bounded convex domain $K$ with prescribed area measure $\Area(K)$, generalising the \emph{classical Euclidean Minkowski problem} \cite{Minkowski_03} asking for the existence of a $C^2$ bounded of $\E^{d+1}$ having specified \emph{Gaussian curvature}.

It has been solved by Alexandrov \cite{Alexandrov_38}, who has shown that for any finite Borel measure $\mu$ on $\Sph^d$ with support not included in any hemisphere and satisfying an ‘‘equilibrium'' condition 
\begin{equation}\label{equilibirum}\tag{Eq}
    \int_{\Sph^d}\vec{n} \, \dif\mu(\vec{n})=0 \, ,
\end{equation}
there exists a unique (up to translation) bounded convex domains $K$ inside the Euclidean affine space with area measure $\Area(K)=\mu$. 

Alexandrov's proof of the existence of a solution relies on polyhedral approximation and the fact that Minkowski had already solved the problem for area measures coming from bounded polyhedra via a variational method \cite{Minkowski_97}. A completely variational proof of that fact has later been found by Carlier \cite{Carlier_03}. That proof is based on the study of \emph{volume functional} on bounded convex domains of $\E^{d+1}$ through the use of their \emph{support functions} on the unit sphere $\Sph^{d}$.

The uniqueness, up to translation, of the solution of the generalised Minkowski problem is then an independent convex geometry fact, which is a consequence of the Brunn--Minkowski inequality \cite[Theorem~8.1.1]{Schneider_93}.

\subsection*{The invariant Minkowski Problem in Minkowski space}

Let $d \geq 1$ and let $\M^{d,1}$ denote the $(d + 1)$-dimensional \emph{time-oriented Minkowski affine space}, that is the oriented real affine space endowed with a complete flat Lorentzian metric and a time orientation. Let $\H^d$ denote the one-sheeted hyperboloid consisting of all unit future timelike vectors of the Minkowski vector space $\R^{d,1} =(\V^{d+1}, \Vert\cdot\Vert_{d,1})$ of translations of the Minkowski affine space $\M^{d,1}$. The metric $\Vert\cdot\Vert_{d,1}$ induces a Riemannian metric on $\H^{d}$, giving the classical hyperboloid model of the $d$-dimensional \emph{hyperbolic space}.

An \emph{$F$-convex domains} of $\M^{d,1}$\cite{Fillastre_Veronelli_16,Bonsante_Fillastre_17} is a convex domain $K$ equal to its own future (denoted by $J^+(K)$) and such that every spacelike linear hyperplane of $\R^{d,1}$ directs a supporting affine hyperplane of $K$. Notice that the condition $K = J^+(K)$ imposes that any supporting hyperplane $H$ of $K$ is either \emph{spacelike}, i.e. the restriction of $\Vert\cdot\Vert_{d,1}$ onto $H$ has signature $(d,0,0)$, or \emph{null}, i.e. the restriction of $\Vert\cdot\Vert_{d,1}$ onto $H$ has signature $(d-1,0,1)$.

Bonsante has shown \cite[Proposition~4.3]{Bonsante_05} that if  $K$ is an $F$-convex domain of $\M^{d,1}$, there is an orthogonal (for the ambient Lorentzian metric) projection from $K$ to its boundary $\partial K$, closely related to the Lorentzian cosmological time on $(K,\Vert\cdot\Vert_{d,1})$ \cite{Andersson_Galloway_Howard_98}. Hence, like in the Euclidean case, one can define parallel neighbourhoods. Consider $K$ an $F$-convex domain of $\M^{d,1}$ and $b$ a Borel subset of $\mathbb{H}^d$. The \emph{parallel $\varepsilon$-neighbourhood} $\mathcal{N}_\varepsilon(K)(b)$ consists of all the points $P \in K$ at Lorentzian distance at most $\varepsilon$ from $\partial K$ having orthogonal projection $P'$ onto $\partial K$ such that $\overrightarrow{P'P}$ is collinear to a vector $\vec{n} \in b$. Those parallel neighbourhoods now satisfy
\begin{equation*}
    \mathcal{N}_\varepsilon(K)(\H^d) = K \setminus(K+ \varepsilon \H^d) \,.
\end{equation*}
The volume of such neighbourhoods also satisfies a \emph{local Steiner Formula} \cite{Fillastre_Veronelli_16}: it uniquely decomposes as
\begin{equation*}
\vol_{d,1}\bp{\mathcal{N}_\varepsilon(K)(b)} = \frac{1}{d+1}\sum_{i=0}^{d}\varepsilon^{d+1-i} \binom{d+1}{i} S_i(K)(b) \, , 
\end{equation*}
where each $S_i(K)$ is a Radon measure on the hyperbolic space $\H^d$ and $S_0(K)$ is always $\vol_{\H^d}$, the hyperbolic volume on the hyperboloid. Once again the area measure $\Area(K) = S_d(K)$ generalises Gaussian curvature: if $K$ is an $F$-convex domain of $\M^{d,1}$ with a locally uniformly convex $C^2$ boundary, its area measure $A(K)$ is a Radon measure on $\H^d$ satisfying
\begin{equation*}
 \Area(K) = \frac{1}{\phi} \, \vol_{\H^d} \, ,
\end{equation*}
where $\phi(\vec{n})$ is the \emph{extrinsic Gaussian curvature} of the hypersurface $\partial K$ inside  $\M^{d,1}$ at the point $P(\vec{n}) \in \partial K$ with exterior normal vector $\vec{n} \in \H^d$. 

In that setting, Bonsante and Fillastre \cite{Bonsante_Fillastre_17} have solved and \emph{invariant generalised Minkowski problem} (concerning prior results, see the paragraph preceding the statement of Theorem~\ref{theo intro C-curv AGC}). They have considered $F$-convex domains of $\M^{d,1}$ invariant under a subgroup $\Gamma_\tau < \SO_0(d,1) \ltimes \R^{d,1}$ of affine transformations obtained by adding translations parts to a uniform lattice $\Gamma < \SO_0(d,1)$. The study of such groups was motivated by the work of Mess \cite{Mess_07,Mess_notes}, Bonsante \cite{Bonsante_05}, and Barbot \cite{Barbot_05}, proving that they describe all the 
$(d+1)$-dimensional flat Lorentzian spacetimes with compact hyperbolic Cauchy surface.

 Then, the area measure $A(K)$ of a $\Gamma_\tau$-invariant $F$-convex domain is a $\Gamma$-invariant Radon measure on $\H^d$, thus descending to a finite measure on the closed hyperbolic manifold $\H^d/\Gamma$. Bonsante and Fillastre have proved that for any finite measure $\mu$ on $\H^d/\Gamma$, there exists a unique $\Gamma_\tau$-invariant $F$-convex domain in $\M^{d,1}$ with area measure descending to $\mu$. Note that unlike the Euclidean case, the measure $\mu$ does not have to satisfy any ‘‘equilibrium'' condition \eqref{equilibirum}.

Their proof is variational and similar to Carlier's proof of the Euclidean case:  it relies on the use of the \emph{covolume functional} on invariant $F$-convex domains and of their \emph{hyperbolic support functions} on $\H^d$. 

\subsection*{Affine Steiner Formula}

Let $d \geq 2$, let $\A^{d+1}$ denote the $(d+1)$-dimensional \emph{oriented real affine space} endowed with a parallel volume form $\dvol$. Let $\V^{d+1}$ denote the vector space of translations of $\A^{d+1}$, which is naturally endowed with the linear volume form $\linvol$ inducing $\dvol$ on $\A^{d+1}$. Let $\SL(\V^{d+1})$ denote the space of volume and orientation preserving linear transformations of $(\V^{d+1},\linvol)$.

An open convex cone $\C \subset \V^{d+1}$ is said to be \emph{proper} if it is non-empty and its closure does not contain any entire line. 
Given a proper open convex cone $\C \subset \V^{d+1}$, we shall consider the so-called \emph{$\C$-convex domains} of the affine space $\A^{d+1}$ (see Figure~\ref{fig C-convex}). A $\C$-convex domain $K \subset \A^{d+1}$ is a convex domain having \emph{asymptotic cone} (or \emph{recession cone}) equal to $\overline{\C}$, the closure of $\C$, i.e. it is a convex domain $K \subset \A^{d+1}$ such that
\begin{equation*}
 \left\lbrace \vec{v} \in \V^{d+1} \st \forall P \in K, \ P+ \R_+ \vec{v} \subseteq K \right\rbrace = \overline{\C} \, .
\end{equation*}
Note that the set of linear hyperplanes of $\V^{d+1}$ that do not intersect $\C \subset \V^{d+1}$ is given by $\overline{\P(\C^*)}$, the closure of the projectivisation of the proper open convex cone $\C^*$ dual to $\C$:
\begin{equation*}
\C^* \coloneqq \mathrm{int} \left\lbrace \varphi \in (\V^{d+1})^* \st \varphi(\vec{v}) \leq 0, \ \forall \, \vec{v} \in \C\right\rbrace \, ,
\end{equation*}
where $\mathrm{int}$ denotes the interior of a set. Thus, any supporting hyperplane of a $\C$-convex domain must be directed by an element in $\overline{\P(\C^*)}$. An affine hyperplane directed by an element of $\P(\C^*)$ will be called \emph{$\C$-spacelike}.

A central idea in the study of convex subsets of $\A^{d+1}$ is convex duality: one should see convex subsets as intersections of half-spaces rather than sets of points. 
In the Euclidean and Minkowski cases described above, the ``spheres'' \(\Sph^{d}\) and \(\H^d\) allow to define a normal direction from a supporting hyperplane of a convex set. In our affine setting, 
this ‘‘sphere'' will be $\S$, the \emph{Cheng--Yau unit affine sphere} associated with $\C \subset (\V^{d+1}, \linvol)$ \cite{Cheng_Yau_77,Loftin_10}. It is a uniquely defined smooth convex hypersurface asymptotic to $\C$, endowed with a canonical volume form $\dvol_{\S}$, see subsection~\ref{subsec Affine differential geometry}. 

\begin{remark}
Note that when $\C$ is the cone of future timelike vectors in the Minkowski vector space $\R^{d,1}$,   the $\C$-convex domains are exactly the $F$-convex domains from \cite{Fillastre_Veronelli_16,Bonsante_Fillastre_17}, and the affine sphere $\S$ is $\H^d$. Then, we shall recover all the tools from convex geometry in the Minkowski space.
\end{remark}

Thanks to that normal direction,  we have proved \cite[Section 6]{Ablondi_25} that, similarly to the Minkowski setting, every $\C$-convex domain $K$ admits a well-defined projection from itself to its boundary $\partial K$,  that gives a foliation of $K$ by the hypersurfaces $(\partial(K+t\S))_{t>0}$. That allows to define the parallel $\varepsilon$-neighbourhoods for $\C$-convex domains similarly to the Minkowski and Euclidean cases (see subsection~\ref{subsec parallel neighbourhoods}). Using duality, that is the correspondence between a $\C$-spacelike hyperplane and its normal vector given by $\S$, they are defined in the following way. Consider $K$ an $\C$-convex domain of $\A^{d+1}$ and $b$ a Borel subset of $\P(\C^*)$, $\mathcal{N}_\varepsilon(K)(b)$ consists of all the points $P \in K\setminus(K+\varepsilon\S)$ such that at $P'$, the normal projection of $P$ onto $\partial K$, $K$ has a supporting hyperplane directed by an element in $b$. 

Then, the volume $V_\varepsilon(K)$ of the parallel $\varepsilon$-neighbourhoods will be a Radon measure on $\P(\C^*)$ (Proposition~\ref{prop Veps is Radon}). Similarly, using duality, the support function of a $\C$-convex domain, which indicates the position of its supporting hyperplanes, will be a function on the hypersurface $\S^*$ dual to $\S$, which happens to be the Cheng--Yau affine sphere asymptotic to $\C^*$ \cite{Gigena_81}, and which is radially identified with $\P(\C^*)$.

\begin{remark}
That might look different from the Euclidean and Minkowski cases where no duality seems to be involved. Actually, it is not. Using its bilinear form, the Euclidean vector space $(\V^{d+1}, \Vert \cdot \Vert_{euc})$ (respectively the Minkowski vector space $\R^{d,1}$) is usually implicitly identified with its dual space, then the unit sphere (resp. the unit hyperboloid) is an self-dual hypersurface:  $\Sph^{d} = (\Sph^{d})^*$ (resp. $\H^{d} = (\H^{d})^*$). Hence, support functions and area measures are often seen as supported by $\Sph^{d}$ (resp. $\H^{d}$) rather than $(\Sph^{d})^*$ (resp. $(\H^{d})^*$).
\end{remark} 

We then prove that those measures satisfy a local Steiner Formula. 

\begin{maintheorem}[Local Steiner Formula]
\label{theo intro Steiner}
Let $\C \subset \V^{d+1}$ be a proper open convex cone and $K \subset \A^{d+1}$ be a $\C$-convex domain. Then, there exist a unique family of Radon measures $S_0(K)$, $\dots$, $S_d(K)$ on $\P(\C^*)$, such that for all $\varepsilon>0$ and $b$ a Borel set of $\P(\C^*)$,
\begin{equation*}
V_\varepsilon(K)(b) = \vol\bp{\mathcal{N}_\varepsilon(K)(b)} = \frac{1}{d+1}\sum_{i=0}^{d}\varepsilon^{d+1-i} \binom{d+1}{i} S_i(K)(b) \, . 
\end{equation*}
Moreover, one always has $S_0(K) = \vol_{\S^*}$, where $\dvol_{\S^*}$ is the canonical volume form on the affine sphere $\S^*$ radially identified with $\P(\C^*)$.
\end{maintheorem}

\begin{remark}
In general, the affine sphere has no reason to be self-dual, that is, there might not be an diffeomorphism $\S \to \S^*$ sending the \emph{affine invariants} of $\S$ onto those of $\S^*$. Nevertheless, there is always a canonical smooth diffeomorphic identification $\S \simeq \S^*$ given by the \emph{Blaschke conormal} (see \cite{Loftin_10} for more details). This diffeomorphism actually sends $\dvol_{\S}$ onto $\dvol_{\S^*}$, so from now we shall abusively denote by $\vol_{\S}$ the canonical volume on $\S \simeq \S^* \simeq\P(\C^*)$ given by the affine sphere $\S$.
\end{remark}

We thus define the \emph{area measure} of a $\C$-convex domain $K$ as the Radon measure $\Area(K) \coloneqq S_d(K)$ given by Theorem~\ref{theo intro Steiner}.

In the case where $K$ is a $\C$-convex domain with a locally uniformly convex $C^2$ boundary, the situation is similar to the Euclidean and Minkowski cases: we prove that its area measure $A(K)$ is of the form 
\begin{equation*}
 \Area(K) = \frac{1}{\phi} \, \vol_{\S} \, ,
\end{equation*}
where $\dvol_{\S}$ is the canonical volume on $\P(\C^*)$ given by the affine sphere $\S$ (see the remark above), and $\phi$ is (the push-forward by the Gauss map of) the \emph{Gaussian curvature of $\partial K$ relative to the affine sphere $\S$}, or more precisely \emph{relative to the pair $(\S,\S^*)$}, which we shall call \emph{$\C$-curvature}. That curvature is \emph{relative} in the sense of \emph{affine differential geometry} \cite{LSZH_15}, and is obtained in the following way. First, endow $\partial K$ with the following transverse vector field $N$, which we shall call the \emph{$\C$-normal vector field on $\partial K$}: for every point $P \in \partial K \subset \A^{d+1}$, let $N(P)$ be the vector $\vec{n} \in \S \subset \V^d$ such that the tangent hyperplane of $\S$ at $\vec{n}$ has the same direction as the tangent hyperplane of $\partial K $ at $P$ (see Figure~\ref{fig normal}). Then, as for the usual extrinsic Gaussian curvature in a Riemannian setting, $\C$-curvature is obtained by considering the infinitesimal variation of the $\C$-normal field on the hypersurface $\partial K$. 

\begin{remark}
\label{rem intro curvature}
That construction of $\C$-curvature actually generalises Euclidean and Minkowski geometry: 
\begin{itemize}
 \item In the Euclidean affine space $\E^{d+1}$, the Gauss structural equations of an immersed hypersurface come from its affine differential geometry relative to the unit sphere $\Sph^d$ (which is an affine sphere); the normal vector to any affine hyperplane $H$ is the vector of $\vec{n} \in \Sph^d$ having tangent hyperplane with the same direction as $H$. 
 \item In the Minkowski affine space $\mathbb{M}^{d,1}$, the Gauss structural equations of an immersed spacelike hypersurface come from its affine differential geometry relative to the hyperboloid $\H^d$ (which is an affine sphere); the normal vector to any spacelike affine hyperplane $H$ is the vector of $\vec{n} \in \H^d$ having tangent hyperplane with the same direction as $H$. 
\end{itemize}  
\end{remark}

\subsection*{An affine invariant Minkowski problem}

We now consider the case when the cone $\C$ is \emph{divisible} by a discrete torsion-free subgroup $\Gamma < \SL(\V^{d+1})$, that is the projective action of $\Gamma$ on $\P(\C)$ is free, properly discontinuous and cocompact. Let $\Gamma_\tau < \SA(\A^{d+1}) $ be an \emph{affine deformation} of $\Gamma$, that is a discrete subgroup of orientation and volume preserving affine transformations obtained by adding translation parts to elements of $\Gamma$. Then, one checks that the area measures of $\Gamma_\tau$-invariant $\C$-convex domains descend to Radon measures on the projective manifold $\P(\C^*)/\Gamma$, given by the quotient of the proper convex projective domain $\P(\C^*)$ by the dual projective action of $\Gamma$, which is also free, properly discontinuous and cocompact \cite[Lemma~2.8]{Benoist_04}. 

In that affine invariant setting, we solve the following Minkowski problem. 

\begin{maintheorem}
\label{theo intro Minkowski solution}
Let $\Gamma_\tau < \SA(\A^{d+1})$ be an affine deformation of a subgroup $\Gamma < \SL(\V^{d+1})$ dividing a proper open convex cone $\C \subset \V^{d+1}$. For every finite Radon measure $\mu$ on the closed convex projective manifold $\P(\C^*)/\Gamma$, there exists a unique $\Gamma_\tau$-invariant $\C$-convex domain $K \subset \A^{d+1}$ with area measure $\Area(K)=\mu$. 
\end{maintheorem}

Even though  the $\Gamma_\tau$-invariant $\C$-convex domains are unbounded, the divisibility condition, which is a cocompactness property, and the volume preserving nature of the transformations we study, allow to introduce their \emph{covolume}. Inspired by Carlier's variational approach and the work of Bonsante and Fillastre, we show that, in some sense area measures are derivatives of the covolume functional.

\begin{maintheorem}
\label{theo intro derivative covol}
The Gateau derivative of the covolume functional at a $\Gamma_\tau$-invariant $\C$-convex domain $K$ is the area measure $\Area(K)$. 
\end{maintheorem}

For a more rigorous and precise statement of Theorem~\ref{theo intro derivative covol}, we refer to Theorem~\ref{theo area measure as gateau derivative} and Remark~\ref{rem Gateaux derivative}. 

Using Theorem~\ref{theo intro derivative covol}, we proceed and get a variational proof that the affine invariant Minkowski problem always admits a solution.

It is worth noticing that there is a Monge--Amp\`ere equation underlying that affine Minkowski problem. Indeed, by arbitrarily choosing an affine chart to identify $\P(\C^*)$ with a bounded open convex domain $\Omega^* \subset \R^d$, support functions of $\C$-convex domains can be seen as convex functions on $\Omega^*$. We show (Proposition~\ref{prop A MA C-convex}) that the area measure of a $\C$-convex domain $K$ satisfies 
\begin{equation*}
    \Area(K) = (-\omega_\C) \MA(s_K) \, ,
\end{equation*}
where $\omega_\C : \Omega^* \to \R$ is the support function of the affine sphere $\S$ and $\MA(s_K)$ is the Monge--Amp\`ere measure of $s_K : \Omega^* \to \R$, the support function of $K$. Hence, prescribing the area measure $\mu$ is equivalent to finding a solution of a Monge--Amp\`ere equation:
\begin{equation*}
    \MA(s) = \frac{1}{(-\omega_\C)} \mu\, .
\end{equation*}
Then, the uniqueness of the solution in Theorem~\ref{theo intro Minkowski solution} is a consequence from standard arguments of Monge--Amp\`ere equation theory.

Theorem~\ref{theo intro Minkowski solution} extends a long series of results:
\begin{itemize}
 \item Oliker and Simon have proved it in the case where $\Gamma_\tau = \Gamma < \SO(d,1)$ is a uniform lattice, $\S = \H^d$, and $\mu$ of the form $\phi \, \vol_{\H^d}$, with $\phi$ a smooth positive function on $\H^d/\Gamma$ \cite{Oliker_Simon_83},
 \item Barbot, B\'eguin and Zeghib have proved it for any affine deformation $\Gamma_\tau$ of a Fuchsian subgroup $\Gamma < \SO(2,1)$, $\S = \H^2$, and $\mu$ of the form $\phi \, \vol_{\H^2}$, with $\phi$ a positive smooth function on $\H^2/\Gamma$ \cite{BBZ_11},
 \item Bonsante and Fillastre have proved it for any affine deformation $\Gamma_\tau$ of a uniform lattice $\Gamma < \SO(d,1)$, $\S = \H^d$, and $\mu$ any finite Radon measure on $\H^d/\Gamma$ \cite{Bonsante_Fillastre_17},
 \item Labourie \cite[Theorem~8.2.1]{Labourie_07} and later Nie and Seppi \cite{Nie_Seppi_23} have proved it, with different approaches, for any affine deformation $\Gamma_\tau$ of a group $\Gamma < \SL(\V^3)$ dividing a convex cone $\C \subset \V^3$, and $\mu = k \, \vol_{\S^*}$, where $k>0$ is constant \cite{Nie_Seppi_23}. 
\end{itemize}
In fact, Labourie and Nie and Seppi were not considering $\C$-curvature, but the distinct \emph{affine Gaussian curvature}. Nevertheless, if the boundary of a $\C$-convex domain has constant $\C$-curvature, we prove that it also has a constant (with a different value) affine Gaussian curvature.

\begin{maintheorem}
\label{theo intro C-curv AGC}
Let $K$ be convex domain in $\A^{d+1}$ with $C^2$ boundary. If $\partial K$ has constant $\C$-curvature $\kappa>0$, then it has constant  affine Gaussian curvature $\kappa^{2(d+1)/(d+2)}$. 
\end{maintheorem}

Note that affine Gaussian curvature is only defined for $C^2$ hypersurfaces, while the notion of $\C$-curvature can naturally be generalised to the boundary of any $\C$-convex domain via area measures.

\subsection*{Relations to affine spacetimes}

In \cite{Ablondi_25}, we have introduced a notion of \emph{affine spacetime} structure, with which the quotients of the invariant domains considered above are naturally endowed. If $\Gamma_\tau < \SA(\A^{d+1}) $ is an affine deformation of a discrete subgroup $\Gamma < \SL(\V^{d+1})$ dividing a proper open convex cone $\C \subset \V^{d+1}$, by the work of Nie and Seppi \cite{Nie_Seppi_23}, and Choi \cite{Choi_25}, there exist a maximal $\Gamma_\tau$-invariant $\C$-convex domain $D_\tau \subset \A^{d+1}$. Then, we have shown that $(D_\tau/\Gamma_\tau,\C)$  is a \emph{Maximal Globally Hyperbolic affine spacetimes admitting a $C^2$ locally uniformly Convex and Compact Cauchy surface}, denoted as \emph{MGHCC affine spacetimes} and that, conversely, any MGHCC affine spacetime can be built in that way \cite[Theorems I and II]{Ablondi_25}. 

When the cone $\C$ is quadratic, that is a theorem on \emph{flat Lorentzian spacetimes} due to Mess \cite{Mess_07,Mess_notes} for $d=2$, and Bonsante \cite{Bonsante_05} and Barbot \cite{Barbot_05} for $d \geq 2$. 

Hence, Theorem~\ref{theo intro Minkowski solution} can produce \emph{Cauchy surfaces} inside MGHCC affine spacetimes with prescribed $\C$-curvature. But one has to be careful, as the boundary $S= \partial K$ of a $\Gamma_\tau$-invariant $\C$-convex domain $K$ might not completely be included inside the open convex set $D_\tau$: one could have $\partial K \cap \partial D_\tau \neq \emptyset$. Still, like in the flat Lorentzian case \cite[Theorem~1.4]{Bonsante_Fillastre_17}, if the \emph{total mass} of the area measure $\Area(K)$ is big enough, that cannot happen. 

\begin{maintheorem}
\label{theo intro Minkowski pb in MGHCC measure}
Let $M = (D_\tau/\Gamma_\tau,\C)$ be an MGHCC affine spacetime. Then, there exists a constant $m(M) \geq 0$ such that if $\mu$ is a Radon measure on $\P(\C^*)/\Gamma$ with total mass $\mu(\P(\C^*)/\Gamma) > m(M)$, there is a unique convex Cauchy surface $S$ in $M$ with area measure equal to $\mu$. 
\end{maintheorem}

When that total mass condition is satisfied, one can also use regularity results on the \emph{Monge--Amp\`ere equation} underlying the Minkowski problem to get information on the regularity of the solution.

\begin{maintheorem}
\label{theo intro Minkowski pb in MGHCC reg}
Let $M = (D_\tau/\Gamma_\tau,\C)$ be an MGHCC affine spacetime. Then, there exists a constant $\varepsilon(M) \in (0,+\infty]$ such that if $\phi$ is a positive $C^{k+1}$ function on $\P(\C^*)/\Gamma$ with $k\geq 2$, and such that $\phi<\varepsilon(M)$, then there is a unique $C^{k+2}$ Cauchy surface $S$ in $M$ with $\C$-curvature equal to $\phi$. 

If $d=2$ or $D_\tau$ is simple (see Definition~\ref{def simple}), the constant $\varepsilon(M)$ can be taken to be $+\infty$. 
\end{maintheorem}

Note that the total mass condition cannot be removed from the statement of Theorem~\ref{theo intro Minkowski pb in MGHCC measure}. Indeed, for all $d \geq 3$, Bonsante and Fillastre \cite[subsection 3.7]{Bonsante_Fillastre_17} have provided examples, based on the work of Pogorelov \cite{Pogorelov_78}, of an affine deformation $\Gamma_\tau$ of a uniform lattice $\Gamma$ of $\SO_0(d,1)$, and a positive measure $\mu$ on $\H^d/\Gamma$ (that is such that $\mu \geq m \, \vol_{\H^d}$ with $m>0$) such that the future-convex domain with area measure $\mu$ given by Theorem~\ref{theo intro Minkowski solution} meets the boundary $\partial D_\tau$. 

\subsection*{\texorpdfstring{Dimension $2+1$ and higher Teichm\"uller theory}{Dimension 2+1 and higher Teichm\"uller theory}}

Finally, let us write a few words about the special case when $d=2$, and the relation between the work of the present article and (higher) Teichm\"uller theory.

First, let us start by considering the case when the cone $\C$ is quadratic, i.e. convex geometry inside the Minkowski space $\M^{2,1}$. That setting is closely related to \emph{classical Teichm\"uller theory}. Indeed, if $S_g$ is a closed oriented and connected surface of genus $g \geq 2$, its \emph{Teichm\"uller space} $ \mathcal{T}(S_g)$ can be seen as a connected component of the character variety
\begin{equation*}
\chi \bp{ \pi_1(S_g) , \SO_0(2,1)} \coloneqq \text{Hom}\bp{\pi_1(S_g),\SO_0(2,1)} \big/ \SO_0(2,1) \, , 
\end{equation*}
where the quotient is by conjugation. More precisely, $ \mathcal{T}(S_g)$ is diffeomorphic to $\chi^{fd}( \pi_1(S_g) , \SO_0(2,1))$ one of the two connected components of the character variety consisting of conjugacy classes of faithful and discrete representations \cite{Goldman_thesis}, i.e. representations $\rho$ such that $\rho(\pi_1(S_g))$ is a uniform lattice of $\SO_0(2,1)$.

If $[\rho] \in \mathcal{T}(S_g)$, through the special identification between the Lie algebra $\mathfrak{so}(2,1)$ and the Minkowski space $\R^{2,1}$, an affine deformations of $\rho(\pi_1(S_g))$ can actually be seen as vector in the tangent space $\mathrm{T}_{[\rho]} \mathcal{T}(S_g)$. By the work of Mess~\cite{Mess_07,Mess_notes}, an affine deformation $\Gamma_\tau$ of $\Gamma=\rho(\pi_1(S_g))$ can also be identified with a \emph{measured geodesic lamination} on the hyperbolic surface $\H^{2}/\rho(\pi_1(S_g))$, which is another model for tangent vectors of $\mathrm{T}_{[\rho]} \mathcal{T}(S_g)$. 

The work of Barbot and Fillastre \cite{Barbot_Fillastre_20} has highlighted how that fact can be recovered through convex geometry in the Minkowski space. The first area measure $S_1(D_\tau)$ of the maximal $\Gamma_\tau$-invariant $F$-convex domain $D_\tau \subset \M^{2,1}$ (given by the Steiner Formula, Theorem~\ref{theo intro Steiner}) induces a finite measure on $\H^d/\Gamma$ supported by a geodesic lamination, from which one can recover the measured lamination.

Another tool from Teichm\"uller theory which can be recovered through convex geometry in the affine Minkowski space $\M^{2,1}$ is \emph{grafting}. Indeed, let $\Gamma_\tau$ be an affine deformation of a Fuchsian subgroup. Then, the surface $\partial(D_\tau+\H^2) \subset \M^{2,1}$ is a $C^{1,1}$ surface endowed with a $C^{0,1}$ Riemannian metric induced by the ambient Lorentzian metric of $\M^{2,1}$. By the work of Benedetti and Bonsante \cite{Benedetti_Bonsante_09}, it actually is the Riemannian surface obtained by the grafting of the hyperbolic surface $\H^{2}/\Gamma$ along the measured geodesic lamination associated with the affine deformation $\Gamma_\tau$ described above.

Now, let us consider the more general case of a divisible cone $\C$. That setting is closely related to \emph{higher Teichm\"uller theory}. Indeed, if $S_g$ is a closed oriented and connected surface of genus $g \geq 2$,  a theorem of Choi and Goldman \cite{Choi_Goldman_93} states that the \emph{Hitchin component} \cite{Hitchin_92} of the representation space of $\pi_1(S_g)$ into $\SL(\V^3)$ is also the \emph{moduli space of convex projective structures on $S_g$}, and it can be described as 
\begin{equation*}
\mathrm{Hit}(S_g) = \mathrm{Hom}_{\mathrm{div}} \bp{\pi_1(S_g), \SL(\V^3)} \Big/\SL(\V^3) \, ,
\end{equation*}
where the quotient is by conjugation and $\mathrm{Hom}_{\mathrm{div}} (\pi_1(S_g), \SL(\V^3))$ is the set of faithful and discrete group representations $\rho \in \mathrm{Hom}(\pi_1(S_g) , \SL(\V^3))$ such that there is a proper open convex cone $\C_\rho \subset \V^3$ divisible by $\rho(\pi_1(S_g))$. 

In that setting, the area measure $S_1(D_\tau)$ and the surface $\partial(D_\tau+\S) \subset \A^{2+1}$ are still-well defined. Hence, we can hope for a natural generalisation of measured geodesic lamination and a generalised grafting process for convex projective surfaces. 

Note that we have based our study of $\C$-convex domains on affine spheres, which were used independently by Loftin \cite{Loftin_01} and Labourie \cite{Labourie_07} in order to get a holomorphic parametrisation of the $\SL(\V^3)$-Hitchin component. Thus, we could hope that a generalised grafting process obtained in that way would have a nice behaviour in the Labourie--Loftin parametrisation of $\mathrm{Hit}(S_g)$.

It would also be interesting to see if we can relate our work with the one of Bobb and Farre \cite{Bobb_Farre_24} concerning the correspondence between affine deformation of dividing groups and affine laminations \cite{Hatcher_Oertel_92}.

\subsection*{Organisation of the article} 

In Section~\ref{sec (P)}, we describe the setting of this article and introduce a parametrisation $(\mathbf{P})$ of the affine space which will be useful for computations. In Section~\ref{sec C-convex domains}, we introduce and study $\C$-convex domains and the Cheng--Yau affine sphere $\S$ asymptotic to $\C$. Section~\ref{sec area meas}
is dedicated to the proof, via an approximation argument, of the local Steiner Formula (first part of Theorem~\ref{theo intro Steiner}), allowing to define area measures for $\C$-convex domains. Then, in Section~\ref{sec relation area measure curvature and MA}, we use affine differential geometry relative to the Cheng--Yau affine sphere $\S$ in order to introduce a notion of $\C$-curvature for $\C$-convex domains with enough regularity. We then proceed to exhibit the relation between area measures and $\C$-curvature, and prove the second part of Theorem~\ref{theo intro Steiner} and a stronger version  of Theorem~\ref{theo intro C-curv AGC}. We also highlight the relation between the area measure of a $\C$-convex domain and the Monge--Amp\`ere measure of its support function. Section~\ref{sec invariant convex domains} focuses on the case of $\C$-convex domains invariant under an affine deformation $\Gamma_\tau$ of a discrete subgroup $\Gamma$ dividing $\C$. In Section~\ref{sec covolume}, we introduce the covolume functional on such invariant domains, we prove its strict convexity and that, in some sense, its derivative is the area measure (Theorem~\ref{theo intro derivative covol}). Finally, in Section~\ref{sec inv Mink pb}, we solve the Minkowski problem for such invariant domains (Theorem~\ref{theo intro Minkowski solution}), and explain how classical facts from Monge--Amp\`ere equation theory provide some regularity results implying Theorems~\ref{theo intro Minkowski pb in MGHCC measure} and \ref{theo intro Minkowski pb in MGHCC reg}. 

In addition, Appendix~\ref{app DL domain} presents Dirichlet--Lee fundamental domains for the actions of affine deformation of dividing subgroups. Those convex fundamental domains appear in the proof of the strict convexity of the covolume. Appendix~\ref{app C2+-approximation} provides, for the convenience of the reader, the proof of an approximation lemma for $\C$-convex domain due to Choi, which is used in the proof of Theorem~\ref{theo intro derivative covol}. 

\subsection*{Acknowledgments} I am deeply indebted and grateful to my advisors Fran\c{c}ois Fillastre and Andrea Seppi for introducing me to this subject as well as for their regular help and guidance. I would like to thank Andreas Bernig for various useful explanations and remarks concerning area measures. I also thank Suhyoung Choi for introducing me to his work, which contains the approximation lemma presented in Appendix~\ref{app C2+-approximation}. 

\section{Setting and parametrisation of the affine space}
\label{sec (P)}

\subsection{Setting of the article}

Let $d \geq 2$ be an integer. Let $\V^{d+1}$ be the $(d+1)$-dimensional real vector space endowed with an orientation and a linear volume form $\linvol$. Let $\A^{d+1}$ be the $(d+1)$-dimensional real affine space with $\V^{d+1}$ as its vector space of translations. It is endowed with an orientation, a flat affine connection $D$ and a parallel differential volume form $\dvol$, induced by $\linvol$ on $\V^{d+1}$. We shall denote by $\vol$ the measure on $\A^{d+1}$ obtained by integrating the form $\dvol$.

The dual vector space $(\V^{d+1})^*$ is the space of real linear forms on $\V^{d+1}$. It is naturally endowed with a linear volume form $\linvol^*$ defined in the following way: the volume of a basis $\varphi_1, \dots,\varphi_{d+1}$ of $(\V^{d+1})^*$ is 
\begin{equation*}
 \linvol^*(\varphi_1, \dots,\varphi_{d+1}) = \frac{1}{\linvol(\vec{e}_{d+1},\dots,\vec{e}_1)} \, ,
\end{equation*}
where $\vec{e}_1,\dots,\vec{e}_{d+1}$ is the basis of $\V^{d+1}$ pre-dual to $\varphi_1, \dots,\varphi_{d+1}$.

\begin{definition}[Proper convex cone]
A \emph{cone} in the vector space $\V^{d+1}$ is a subset $\C \subseteq \V^{d+1}$ such that for all $\vec{v} \in \C$ and $\lambda>0$, one has $\lambda\vec{v} \in \C$.

A cone $\C \subset \V^{d+1}$ is \emph{convex} if for all $\vec{v},\vec{w} \in \C$ and $t \in [0,1]$ , one has $ (1-t)\vec{v}+t\vec{w} \in \C$. A convex cone $\C \subset \V^{d+1}$ is said to be \emph{proper} if it is non-empty and its closure does not contain any entire line.   
\end{definition} 

Let $\C \subset \V^{d+1}$ be a proper open convex cone and $\C^* \subset (\V^{d+1})^*$ be the open \emph{polar cone} to $\C$. That is, the proper open convex cone in $(\V^{d+1})^*$ defined as 
\begin{equation}
\label{eq C*}
\C^* \coloneqq \mathrm{int} \left\lbrace \varphi \in (\V^{d+1})^* \st \varphi(\vec{v}) \leq 0, \ \forall \, \vec{v} \in \C\right\rbrace,
\end{equation}
where $\mathrm{int}$ denotes the interior of a set. Using the identity~\eqref{eq C*} and identifying each elements $[\varphi] \in \P((\V^{d+1})^*)$ with the linear hyperplane $\ker \varphi \subset \V^{d+1}$, the set $\overline{\P(\C^*)} $ can be seen as the set of \emph{supporting hyperplanes} of the convex cone $\C$, that is of linear hyperplanes $H \subset \V^{d+1}$ such that $H\cap \C = \emptyset $ and $H\cap \partial\C \neq \emptyset$. Among those hyperplanes, $\P(\C^*)$ is the set of linear hyperplane $H \subset \V^{d+1}$ such that $H\cap \overline{\C} = \{0\} $. As a set of linear hyperplanes in $\V^{d+1}$, $\P(\C^*)$ can also be seen as a set of directions of affine hyperplanes in $\A^{d+1}$. Let us then introduce the following definition.

\begin{definition}[$\C$-spacelike hyperplane]
\label{def C-spacelike hyperplane}
An affine hyperplane of $\A^{d+1}$ is said to be \emph{$\C$-spacelike} if it is directed by an element in $\P(\C^*)$.
\end{definition} 

\begin{remark}
The adjective ‘‘spacelike'' comes from the special case where the cone $\C$ is quadratic. In that case, the vector space $\V^{d+1}$ can be endowed with a quadratic form of signature $(d,1)$ with isotropic cone $\partial \C \cup (- \partial \C)$. Then, $\A^{d+1}$ is endowed with a flat Lorentzian metric and $\C$-spacelike hyperplanes are spacelike in the Lorentzian sense.
\end{remark}

Let $\SL(\V^{d+1})$ denote the group of volume preserving linear automorphisms of $(\V^{d+1}, \linvol)$ and $\SA(\A^{d+1})$ denote the group of volume preserving affine automorphisms of $(\A^{d+1}, \dvol)$. We shall denote by $\star$ the \emph{dual action} of $\SL(\V^{d+1})$ on $(\V^{d+1})^*$; for all $\varphi \in (\V^{d+1})^*$ and $\gamma \in \SL(\V^{d+1})$,
\begin{equation*}
\gamma \star \varphi \coloneqq \varphi \circ \gamma^{-1} \, ,
\end{equation*}
and by $*$ the \emph{projective dual action} of $\SL(\V^{d+1})$ on $\P(\V^{d+1})^*$; for all $\varphi \in \P(\V^{d+1})^*$ and $\gamma \in \SL(\V^{d+1})$,
\begin{equation*}
\gamma * [\varphi] \coloneqq [\gamma \star \varphi]= [\varphi \circ \gamma^{-1}] \, .
\end{equation*}
Note that for all $\vec{v} \in \V^{d+1}$, $\varphi \in (\V^{d+1})^*$ and $\gamma \in \SL(\V^{d+1})$,
\begin{equation}
\label{eq action and dual action}
(\gamma \star \varphi)(\gamma \cdot \vec{v})= \varphi(\vec{v}) \, .
\end{equation}

\subsection{A parametrisation of the affine space}

Most of the results presented in this article are purely affine and we shall state them as such, without using any parametrisation of the affine space. Nevertheless, their proofs sometime require to use a parametrisation of the affine space $\A^{d+1}$ as $\R^{d+1}$ in order to do computations. 

Let us arbitrarily choose a base point $O \in \A^{d+1}$, and a direct unit-volume basis $(\vec{e}_1, \dots, \vec{e}_{d+1})$ of $(\V^{d+1},\linvol)$ such that $\vec{e}_{d+1} \in \C$. That gives us a parametrisation of $\A^{d+1}$, and thus $\V^{d+1}$, as $\R^{d+1}$ so that 
\begin{enumerate}[label=($\mathbf{P}$\arabic*)]
\item the cone $\C \subset \V^{d+1}\simeq \R^{d+1}$ gets parametrised as
 \begin{equation}
 \label{eq expression Omega}
 \C = \left\lbrace t \, (x,1) \st x \in \Omega, \ t>0 \right\rbrace ,
 \end{equation}
where $\Omega$ is a bounded open convex domain of $\R^d$ containing $0$,
\item the linear volume form $\linvol$ on $\V^{d+1} \simeq \R^{d+1}$ is the determinant $\det$, which integrates to the Lebesgue measure $\mathcal{L}$ on $\R^{d+1}$.
\end{enumerate}
We shall denote this parametrisation by $(\mathbf{P})$ and use the following notation convention:
\begin{itemize}
\item an element of $\R^d$ will often be denoted by $x$ and an element of $\R^{d+1}$ by $X$,
\item using the usual scalar products $\cdot$ on $\R^d$ and $\R^{d+1}$ for duality identifications, dual elements will often be denoted by $y \in \R^d = (\R^d)^*$ and $Y \in \R^{d+1} = (\R^{d+1})^*$. 
\end{itemize} 
Then, in the parametrisation $(\mathbf{P})$, the polar cone $\C^*$ is expressed as 
\begin{equation*}
\C^*  = \mathrm{int} \left\lbrace Y \in (\R^{d+1})^* \st X\cdot Y \leq 0, \ \forall X \in \C\right\rbrace = \left\lbrace Y \in (\R^{d+1})^* \st X\cdot Y<0, \ \forall X \in \overline{\C} \setminus \{0\} \right\rbrace 
\end{equation*}
so that, using \eqref{eq expression Omega}
\begin{equation}
 \label{eq expression Omega*}
\C^* = \left\lbrace t \,(y,-1) \st y \in \Omega^*, \ t>0 \right\rbrace ,
\end{equation}
where $\Omega^*$ is the bounded open convex domain of $\R^d$ containing $0$ defined by
\begin{equation*}
\Omega^* = \left\lbrace y \in \R^d \st x \cdot y < 1 , \ \forall x \in \overline{\Omega} \right\rbrace. 
\end{equation*}
Using \eqref{eq expression Omega} and \eqref{eq expression Omega*}, we shall radially identify the bounded open convex domains $\Omega$ and $\Omega^*$ with respectively $\P(\C)$ and $\P(\C^*)$. The following table summarises the identifications in the parametrisation $(\mathbf{P})$.

\medskip

\begin{center}
\begin{tabular}{|p{5.5cm} | p{5.5cm}|}
\hline
\multicolumn{2}{|c|}{Identifications}\\
\hline \hline
Unparametrised setting & Parametrisation $(\mathbf{P})$ \\
\hline
$(\V^{d+1}, \linvol)$ & $(\R^{d+1},\det)$ \\
$(\A^{d+1},\dvol)$ & $(\R^{d+1},\dif \mathcal{L})$ \\
$\bp{(\V^{d+1})^*,\linvol^*}$ & $(\R^{d+1},\det)$\\
$\P(\C)$ & $\Omega$ \\
$\P(\C^*)$ & $\Omega^*$ \\
$\SL(\V^{d+1})$ & $\SL(\R^{d+1})$ \\
$\SA(\A^{d+1})$ & $\SL(\R^{d+1})\ltimes \R^{d+1}$\\
\hline
\end{tabular}
\end{center} 

\medskip

In $(\mathbf{P})$, the dual actions of $\SL(\R^{d+1})$ are described as follows. For all $Y \in \R^{d+1}$ and $\gamma \in \SL(\R^{d})$,
\begin{equation*}
 \gamma \star Y = \gamma^{-\top} Y,
\end{equation*}
where $\gamma^{-\top} \in \SL(\R^{d})$ is the inverse of the transpose of $\gamma$. For all $y \in \Omega^*$ and $\gamma \in \SL(\R^{d})$ preserving $\C$ (and thus $\C^*)$,
\begin{equation*}
\gamma * y \coloneqq R\bp{\gamma \star (y,-1)}  = R\bp{\gamma^{-\top}(y,-1)}\, 
\end{equation*}
where $R : \C^*\to\Omega^*$ is the radial projection of the cone $\C^*$ onto $\Omega^*$. 

\section{\texorpdfstring{$\C$-convex domain}{C-convex domains}}
\label{sec C-convex domains}

\subsection{\texorpdfstring{$\C$-convex domains}{C-convex domains}}
\label{subsec C-convex domain}

A convex domain of the affine space $\A^{d+1}$ is a non-empty open set $K \subseteq \A^{d+1}$ such that for all $P,Q \in K$, the segment $[PQ]$ is included in $K$; i.e. for all $t \in [0,1]$,
\begin{equation*}
 P+ t \overrightarrow{PQ} \in K \, .
\end{equation*}
A \emph{supporting hyperplane} of a convex domain $K \subseteq \A^{d+1}$ is an affine hyperplane $H \subseteq \A^{d+1}$ such that $H \cap K = \emptyset$ and $H \cap \partial K \neq \emptyset$. For such a hyperplane, we shall write that \emph{‘‘$p \in \partial K$ has $H$ as a supporting hyperplane''} if $p \in \partial K \cap H$.

\begin{definition}[$\C$-convex domain]
\label{def C-convex}
A \emph{$\C$-convex} domain is a convex domain $K \subset \A^{d+1}$ such that:
\begin{itemize}
 \item $K+\C = K$,
 \item the set $\mathcal{H}(K) \subset\P((\V^{d+1})^*)$ of directions of supporting hyperplanes of $K$ satisfies 
 \begin{equation*}
 \P(\C^*) \subseteq \mathcal{H}(K) \subseteq\overline{\P(\C^*)} \, .
 \end{equation*}
\end{itemize} 

The \emph{$\C$-spacelike boundary} of $\C$-convex domain $K$, denoted by $\partial_{sp}K$, is the set of points $p \in \partial K$ having a supporting hyperplane with direction in $\P(\C^*)$. 

A $\C$-convex domain $K$ is said to be \emph{$\C$-spacelike} if $\partial K = \partial_{sp}K$.
\end{definition}

\begin{remark}
\label{rem asymptotic cone}
Using a classical fact from convex analysis \cite[Theorem~14.2 and Corollary~14.2.1]{Rockafellar_97}, $\C$-convex domains can be equivalently be characterised in the following way. A $\C$-convex domain is a convex domain $K \subset \A^{d+1}$ having \emph{asymptotic cone} (or \emph{recession cone}) equal to $\overline{\C}$, the closure of $\C$, i.e. such that
\begin{equation*}
 \left\lbrace \vec{v} \in \V^{d+1} \st \forall P \in K, \ P+ \R_+ \vec{v} \subseteq K \right\rbrace = \overline{\C} \, .
\end{equation*}
\end{remark}

\begin{example}
\label{example cone}
A trivial example of a $\C$-convex domain is any translation $P+\C$ of the cone $\C$, with $P \in \A^{d+1}$.
\end{example}

\begin{figure}[ht]
 \centering
 \includegraphics{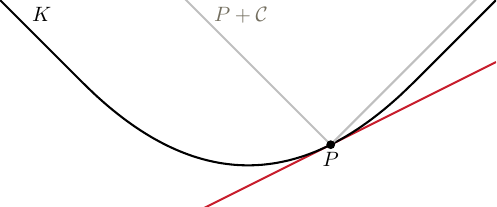}
 \caption{A $\C$-convex domain $K$. }
 \label{fig C-convex}
\end{figure}

\begin{definition}[Gauss map]
\label{def Gauss map}
The ($\C$-spacelike) \emph{Gauss map} of a $\C$-convex domain $K \subset \R^{d+1}$ is the set-valued map
\begin{equation*}
 \G_K : \partial_{sp}K \longrightarrow \mathcal{P}\bp{\P(\C^*)}
\end{equation*}
mapping any point of $\partial_{sp}K$ to the set of directions of its supporting hyperplanes, seen as an element of the power set $ \mathcal{P}(\P(\C^*))$.
\end{definition}

In order to translate the definition of a $\C$-convex domain in the parametrisation $(\mathbf{P})$ of $\A^{d+1}$ described in Section~\ref{sec (P)}, let us recall the definition of the subdifferential of a convex function.

\begin{definition}[Subgradient and subdifferential]
\label{def subgrad subdif}
Let $U \subseteq \R^d$ be a convex domain, let $s: U \to \R$ be a convex function and let $x_0 \in U$. An element $y \in \R^d$ is a \emph{subgradient} of $s$ at $x_0$ if for all $x\in \R^d$,
\begin{equation*}
s(x_0) + (x-x_0) \cdot y \leq s(x). 
\end{equation*}
The set of all subgradients of $s$ at a point $x_0$ is a non-empty convex set of $\R^d$ called the \emph{subdifferential} of $s$ at $x_0$ and is denoted by $\partial s(x_0)$. If $A$ is a subset of $U$, we shall denote 
\begin{equation*}
\partial s(A) \coloneqq \bigcup_{x \in A}\partial s(x) \, . 
\end{equation*}
\end{definition}

Then, it is straightforward to notice that in the parametrisation $(\mathbf{P})$, $\C$-convex domains can be characterised in the following way.

\begin{proposition}[Epigraph characterisation of $\C$-convex domains]
\label{prop char C-convex}
In the parametrisation $(\mathbf{P})$, a $\C$-convex domain $K \subset \R^{d+1}$ is an open set which can be expressed as the epigraph
\begin{equation*}
 K= \epi(f) \coloneqq \left\lbrace(x,\lambda) \in \R^{d+1} \st \lambda>f(x)\right\rbrace
\end{equation*}
of a convex function $f: \R^d\to\R$ such that $\Omega^* \subseteq \partial f (\R^d) \subseteq \overline{\Omega^*}$.

The $\C$-spacelike boundary of a $\C$-convex domain $K = \epi(f)$ then is the set of points $X=(x,f(x)) \in \partial K$ such that $\partial f(x) \cap \Omega^* \neq \emptyset$. 
\end{proposition}

In $(\mathbf{P})$, $\C$-convex domains can be studied through their \emph{support functions}, defined in the following way. The \emph{total support function} of a $\C$-convex domain $K \subset \R^{d+1}$ is the function
\begin{equation*}
\function{\tilde{s}_K}{\C^*}{\R}{Y}{\sup \left\lbrace X\cdot Y \st X \in K\right\rbrace} \, .
\end{equation*}
One checks that it is well-defined (i.e. the supremum is always finite) by definition of a $\C$-convex domain. It is also clearly sublinear and $1$-homogeneous, hence convex. 

The \emph{$\Omega^*$-support function} $s_K:\Omega^* \to \R$ of a $\C$-convex domain $K$ is defined by $s_K (y)=\tilde{s}_K(y,-1)$. By $1$-homogeneity of the total support function, we have that for all $Y = (y,-\mu)$ in $\C^*$,
\begin{equation*}
\tilde{s}_K(Y)=\tilde{s}_K(y,-\mu)=\mu \, s_K\vp{\frac{y}{\mu}} \, ,
\end{equation*}
so that the datum of an $\Omega^*$-support function is equivalent to the one of a total support function. From now on, when there is no possible confusion, we shall simply write that $s_K$ is the \emph{support function of $K$}. 

These support functions are related to the $\C$-convex domains of $\R^{d+1}$ through the \emph{Legendre--Fenchel transform}, which is a classical tool in \emph{convex analysis} \cite{Rockafellar_97} and \emph{convex geometry} \cite{Schneider_93}. 

\begin{definition}[Legendre--Fenchel transform]
\label{def Legendre--Fenchel transform}
Let $U \subseteq \R^d$ be a convex domain. The \emph{Legendre--Fenchel transform} or \emph{Legendre--Fenchel conjugate} of a lower semi-continuous function $s: U \to \R \cup \{+\infty\}$ is the lower semi-continuous convex function $s^*:(\R^d)^* \to \R \cup \{+\infty\}$ defined by 
\begin{equation*}
s^*(y) = \sup_{x \in U}\bp{x \cdot y - s(x)} \, ,
\end{equation*}
using the usual scalar product $\cdot$ on $\R^d$ to identify $(\R^d)^*$ with $\R^d$. 
\end{definition}

\begin{proposition}[{\cite[Proposition~4.9]{Ablondi_25}}]
\label{prop K is epigraph}
In the parametrisation $(\mathbf{P})$, let $K$ be a $\C$-convex domain of $\R^{d+1}$ with support function $s$. Then, it is the epigraph of $f=s^*$, the Legendre--Fenchel transform of $s$. 

Conversely, any convex function $s:\Omega^* \to \R$ gives a $\C$-convex domain, defined as the epigraph of its Legendre--Fenchel transform $f=s^*$, which is finite over the whole $\R^d$. 
\end{proposition}

Here is a classical convex geometry fact \cite[Theorem~1.7.5]{Schneider_93} concerning sum of support functions. 

\begin{proposition}[{\cite[Proposition 4.10]{Ablondi_25}}]
\label{prop support function and sum}
In the parametrisation $(\mathbf{P})$, let $K$ and $K'$ be two $\C$-convex domains of $\R^{d+1}$, with respective support function $s$ and $s'$. Then, their \emph{Minkowski sum}
\begin{equation*}
 K+K'\coloneqq \left\lbrace X+X' \st (X,X')\in K\times K' \right\rbrace
\end{equation*}
is a $\C$-convex domain with support function $s+s'$.
\end{proposition}

\begin{remark}
\label{rem convex combination}
Note that the Minkowski sum between two $\C$-convex domains is only defined in a parametrisation of $\A^{d+1}$, here $(\mathbf{P})$. Otherwise one cannot sum points of an affine space. Nevertheless, convex combinations of two convex domains $K,K' \subset \A^{d+1}$ are always well-defined: for all $t\in [0,1]$,
\begin{equation*}
 (1-t)K+tK'\coloneqq \left\lbrace P+t \, \overrightarrow{PP'} \st (P,P')\in K\times K' \right\rbrace \subset \A^{d+1} \, 
\end{equation*}
is a $\C$-convex domain.
\end{remark}

In the parametrisation $(\mathbf{P})$, the Gauss map $\G_K$ of a $\C$-convex domain $K$ with support function $s$ can thus be expressed as the following set-valued map
\begin{equation*}
\label{eq gauss map def}
\function{\G_{K}}{\partial_{sp} K}{\mathcal{B}(\Omega^*)}{X}{\partial s^*\bp{\pi(X)}} \, ,
\end{equation*}
where $\mathcal{B}(\Omega^*)$ is the Borel algebra of $\Omega^*$, and $\pi$ is the projection on the first $d$ coordinates
\begin{equation*}
 \function{\pi}{\R^{d+1}}{\R^{d}}{X=(x,\lambda)}{x} \, . 
\end{equation*}

\subsection{\texorpdfstring{$C^2_+$-convex $\C$-convex domains}{C2+ C-convex domains}}

Let us now describe a subfamily of $\C$-convex domains with more regularity.

\begin{definition}[$C^2_+$ $\C$-convex domain]
\label{def C2+}
A $C^2_+$ $\C$-convex domain $K \subset \A^{d+1}$ is a $\C$-convex domain with $C^2$ locally uniformly convex boundary, i.e. such that $\partial K$ is locally the graph of a convex function $f$ with non-degenerate Hessian.
\end{definition}

Again, one notices that in the parametrisation $(\mathbf{P})$, those can be characterised in the following way.

\begin{proposition}[Epigraph characterisation of $C^2_+$ $\C$-convex]
\label{prop C2+}
In the parametrisation $(\mathbf{P})$, a $C^2_+$ $\C$-convex domain $K \subset \R^{d+1}$ is an open set which can be expressed as the epigraph of a $C^2$ convex function $f: \R^d\to\R$ such that 
\begin{itemize}
 \item \label{point1 def C2+} $\Omega^* \subseteq \grad f (\R^d) \subseteq \overline{\Omega^*}$,
 \item \label{point2 def C2+} for all $x \in \R^d$, one has $\Hess f(x) >0$. 
\end{itemize}

The $\C$-spacelike boundary of a $C^2_+$ $\C$-convex domain $K = \epi(f)$ then is the set of points $X=(x,f(x)) \in \partial K$ such that $\grad f(x) \in \Omega^*$.
\end{proposition}

\begin{remark}
Note that the definition of a $C^2_+$ $\C$-convex domain given here is slightly stronger than the one given in \cite{Ablondi_25}, where only the $\C$-spacelike boundary is required to be $C^2$ and locally uniformly convex. This is only for the sake of presentation and convenience of the reader, all the following would still hold with the less restrictive definition. 
\end{remark}

\begin{proposition}[{\cite[Propositions 4.11]{Ablondi_25}}]
\label{prop K is epigraph C2+}
In the parametrisation $(\mathbf{P})$, let $K$ be a $C^2_+$ $\C$-convex domain of $\R^{d+1}$. Then, its support function $s$ is $C^2$ with everywhere positive definite Hessian.
\end{proposition}

The Gauss map of a $C^2_+$ $\C$-convex domain $K$ is then a genuine function mapping any point of $\partial_{sp}K$ to the direction of its tangent hyperplane. It is actually a $C^1$ diffeomorphism. Indeed, in the parametrisation $(\mathbf{P})$, if $K$ has support function $s$, its Gauss map is the $C^1$ map
\begin{equation*}
\label{eq expression of the Gauss map}
\function{\G_{K}}{\partial_{sp} K}{\Omega^*}{X}{\grad s^*\bp{\pi(X)}} \, ,
\end{equation*}
where $\pi$ is the projection on the first $d$ coordinates, with inverse the $C^1$ map
\begin{equation}
\label{eq expression of the inverse of the Gauss map}
\function{\G_{K}^{-1}}{\Omega^*}{\partial_{sp} K}{y}{\bp{\grad s(y) , \grad s(y) \cdot y -s (y)}} 
\end{equation}
Note that $\G_{K}^{-1}$ gives a parametrisation of $\partial_{sp}K$ by $\Omega^*$. 

\subsection{Cheng--Yau affine sphere}
\label{subsec Cheng-Yau affine sphere}

In the Euclidean affine space $\E^{d+1}$, the unit sphere $\Sph^{d}$ of $(\V^{d+1},\Vert\cdot\Vert_{euc})$ plays a central role in the study of compact domains convex domains, as it allows to define the normal vector to a hyperplane \cite{Schneider_93}. In the Minkowski affine space $\M^{d,1}$, it is the one-sheeted hyperboloid $\H^{d}$ of $(\V^{d+1},\Vert\cdot\Vert_{d,1})$ that is used for defining normal vectors \cite{Fillastre_Veronelli_16}. In this subsection, we shall thus present the hypersurface of $(\V^{d+1},\linvol)$ which shall fill that role for $\C$-convex domain of $\A^{d+1}$: the \emph{Cheng--Yau} affine sphere $\S$ associated with $\C$.

For definitions and explanations concerning Cheng--Yau affine spheres in the setting of this article, we refer the reader to \cite{Loftin_10} and \cite{Ablondi_25}. Here, we shall only explicit some of the properties satisfied by $\S$:
\begin{itemize}
 \item $\S$ is a smooth locally uniformly convex hypersurface in $(\V^{d+1},\linvol)$ uniquely associated with $\C$.
 \item $\S$ is asymptotic to $\partial \C$.
 \item $\S$ is canonically endowed with a volume form $\dvol_{\S}$.
 \item $\S$ is preserved by $\mathrm{Aut}_{\SL}(\C)$, the subgroup of elements of $\SL(\V^{d+1})$ preserving the cone $\C$. 
 \item $\S$ ‘‘behaves well'' with duality. The hypersurface $\S^* \subset (\V^{d+1})^*$ \emph{polar} to $\S$ is defined as 
 \begin{equation*}
 \S^* = \left\lbrace \varphi \in (\V^{d+1})^* \st \exists \, \vec{v} \in \S \ \text{such that} \ \varphi(\vec{v}) = -1 \ \text{and} \ \ker(\varphi)=\mathrm{T}_{\vec{v}}\S\right\rbrace . 
 \end{equation*}
 By the work of Gigena \cite[Proposition~1]{Gigena_81}, it actually happens to be $\Sigma_{(\C^*)}$, the Cheng--Yau affine sphere associated with $\C^*$ in $((\V^{d+1})^*,\linvol^*)$. We shall call $\S^* = \Sigma_{(\C^*)}$ the \emph{polar affine sphere} of $\C$. Moreover, there is a natural smooth diffeomorphism between $\S$ and $\S^*$, called the \emph{Blaschke conormal field} and this map sends the volume form $\dvol_{\S}$ onto $\dvol_{\S^*}$.
\end{itemize}

\begin{remark}
\label{rem K+S}
Note that for every point $P \in \A^{d+1}$ and $t\geq 0$, $P + t\, \S$ is the boundary of the $\C$-spacelike $C^2_+$ $\C$-convex domain $P+ (t\, \S + \C) \subset \A^{d+1}$. As $K$ is a $\C$-convex domain, $K = K +\C$ and thus for all $t \geq 0$,
\begin{equation*}
 K+t\, \S = K+ \C + t\, \S = K+ (t\, \S + \C)
\end{equation*}
is a $\C$-convex domain (those sums are well-defined because $\C$ and $\S$ are subsets of the vector space $\V^{d+1})$.   
\end{remark}

In the parametrisation $(\mathbf{P})$, $\S \subset \R^{d+1}$ is the graph of $\omega_\C^*$, the Legendre--Fenchel transform of the smooth convex negative function $\omega_\C$ on $\Omega^*$ satisfying the Monge--Amp\`ere equation 
\begin{equation}
\label{eq MA aff sphere}
\begin{cases}
\det \Hess \omega= (-\omega)^{-d-2} & \text{in} \ \Omega^* \, ,\\
\omega_{\vert \partial\Omega^*} =0 \, .
\end {cases}
\end{equation}
The polar affine sphere $\S^* \subset \R^{d+1}$ is then the radial graph of $-\omega_\C$:
 \begin{equation*}
 \S^* = \left\lbrace -\frac{1}{\omega_\C(y)} (y,-1) \st y \in \Omega^*\right\rbrace \subset \, (\R^{d+1})^*. 
 \end{equation*}

In the Euclidean (respectively Minkowski) case, one can completely understand orthogonality for the ambient metric through the unit sphere $\Sph^d$ (resp. the hyperboloid $\H^{d}$): at $\vec{n} \in \Sph^d$ (resp. $ \vec{n} \in \H^{d}$) the orthogonal to $\vec{n}$ is the linear hyperplane directing $\mathrm{T}_{\vec{n}}\Sph^d$ (resp. $\mathrm{T}_{\vec{n}}\H^d$). We shall extend this notion in the following way.

\begin{definition}[$\C$-normal field]
\label{def C-normal vf}
The \emph{$\C$-normal} field on $\P(\C^*)$ is the field $N_{\S}$ associating to each $\C$-spacelike hyperplane direction $[\varphi] \in \P(\C^*)$ the unique vector $\vec{n} \in \S \subset\V^{d+1}$ such that $\mathrm{T}_{\vec{n}}\S$ is directed by $[\varphi]$. 
\end{definition} 

In the parametrisation $(\mathbf{P})$, using the support function $\omega_\C$ of the affine sphere, the $\C$-normal field can explicitly be expressed as the smooth map
\begin{equation}
\label{eq expression of the C-normal field}
\function{N_{\S}}{\Omega^*}{\S \subset\R^{d+1}}{y}{\bp{\grad \omega_\C(y) , \grad \omega_\C(y) \cdot y -\omega_\C (y)}} \,. 
\end{equation}

\section{Steiner Formula and area measures}

\label{sec area meas}

\subsection{Parallel neighbourhood volume measures}
\label{subsec parallel neighbourhoods}

In this subsection, we shall introduce the \emph{parallel $\varepsilon$-neighbourhood volume measures} of $\C$-convex domains which are \emph{Radon measures} on $\P(\C^*)$. So here is a reminder concerning measure theory in $\R^d$. On an open subset $U \subseteq \R^d$, a Radon measure is just a Borel measure which is finite on compact sets. Such a measure $\mu$ is: 
\begin{itemize}
 \item \emph{inner regular}, that is for any Borel set $b\subseteq U$,
 \begin{equation*}
 \mu(b) = \sup \left\lbrace \mu(A) \st A \subseteq b \ \text{and $A$ is compact} \right\rbrace ,
 \end{equation*}
 \item \emph{outer regular}, that is for any Borel set $b\subseteq U$,
 \begin{equation*}
 \mu(b) = \inf \left\lbrace \mu(O) \st b \subseteq O \ \text{and $O$ is open} \right\rbrace . 
 \end{equation*}
\end{itemize}

Given a $\C$-convex domain $K \subset \A^{d+1}$ and $\varepsilon>0$, we can introduce the \emph{parallel $\varepsilon$-neighbourhood volume measure} of $K$ as the measure $V_\varepsilon(K)$ on $\P(\C^*)$ defined in the following way. For any Borel subset $b \subseteq \P(\C^*)$, $V_\varepsilon(K)(b)$ is defined as the volume of the \emph{parallel $\varepsilon$-neighbourhood} $\mathcal{N}_\varepsilon(K)(b) \subset \A^{d+1}$, which is the set of points of the form $P + t \, \vec{n} $ where $P \in \partial_{sp} K \subset \A^{d+1}$ and $\vec{n}\in \S \subset\V^{d+1}$ both have a common supporting hyperplane direction in $b \subset \P(\C^*)$, and $t \in (0,\varepsilon)$.

Using the Gauss map $\G_K$ (Definition~\ref{def Gauss map}) and the $\C$-normal field $N_{\S} $on $\P(\C^*)$ (Definition~\ref{def C-normal vf}), it can equivalently be described in the following way. For any Borel subset $b \subseteq \P(\C^*)$,
\begin{equation*}
V_\varepsilon(K)(b) \coloneqq \vol \bp{\mathcal{N}_\varepsilon(K)(b)}  
\end{equation*}
is the volume of the \emph{parallel $\varepsilon$-neighbourhood} 
\begin{equation*}
\mathcal{N}_\varepsilon(K)(b) \coloneqq \left\lbrace P + t \, N_{\S} [\varphi]\st \ [\varphi]\in b, P \in \G_K^{-1}[\varphi], \ t \in (0,\varepsilon) \right\rbrace \subset K \, ,
\end{equation*}
where $P \in \G_K^{-1}[\varphi]$ means that $P \in \partial_{sp}K$ is such that $[\varphi] \in \G_K(P)$.
In the case where $K$ is a $C^2_+$ $\C$-convex domain, and thus $\G_K$ is a genuine bijective map, that becomes 
\begin{equation}
\label{eq eps neigh C2+}
\mathcal{N}_\varepsilon(K)(b) = \left\lbrace \G_{K}^{-1}[\varphi]+t \, N_{\S}[\varphi]\st [\varphi] \in b , \ t \in (0,\varepsilon )\right\rbrace \, . 
\end{equation}

\begin{remark}
Using the causal distance on $K$ given by the affine sphere $\Sigma_\C$ \cite[Section 6]{Ablondi_25}, $\mathcal{N}_\varepsilon(K)(b)$ can be interpreted as set of points $P$ of $K$ which are at distance at most $\varepsilon$ from $\partial_{sp}K$ and having $\C$-normal projection $P'$ onto $\partial_{sp}K$ such that $\overrightarrow{P'P}$ is collinear to an element of $b$. 
\end{remark}

\begin{proposition}
\label{prop Veps is Radon}
Let $K$ be a $\C$-convex domain and $\varepsilon>0$. Then, $V_\varepsilon(K)$ is a Radon measure on $\P(\C^*)$. 
\end{proposition}
\begin{proof}
Let us work in the parametrisation $(\mathbf{P})$ described in Section~\ref{sec (P)}, where $V_\varepsilon(K)$ is a measure on the bounded convex domain $\Omega^* \subset\R^d$. Let $s$ be the support function of $K$.

We shall prove that for all $A\subset \Omega^*$ compact, $\mathcal{N}_\varepsilon(K)(A) \subset \R^{d+1}$ is bounded, and thus that its Lebesgue volume $V_\varepsilon(K)(A)$ is finite. Let $A$ be a compact subset of $\Omega^*$. As $N_{\S}$ is continuous on $\Omega^*$, the image $N_{\S}(A)$ is compact. Hence, as $\mathcal{N}_\varepsilon (K)(A)\subseteq \G_K^{-1}(A)+[0,\varepsilon] \, N_{\S}(A)$, we just have to show that $\G_K^{-1}(A)$ is bounded. 

By contradiction, assume that there exists an unbounded sequence $(X_n)_n = (x_n,s^*(x_n))_n \in \G_K^{-1}(A)^\N$. That means that for all $n$, $\partial s^*(x_n) \cap A \neq \emptyset$. Then, by the Fenchel Inequality Theorem \cite[Theorem~23.5]{Rockafellar_97}, each $x_n$ is in $\partial s(A)$. 

For all $n\in \N$, we have $X_n \in \graph s^*$. Hence, as the sequence $(X_n)_n$ is unbounded and $s^*$ is continuous (it is convex and finite over the whole $\R^d$), the sequence of $x_n \in \partial s(A) \subseteq \R^{d}$ is necessarily unbounded. But, by \cite[Lemma~A.22]{Figalli_17}, the compactness of $A$ imply that $\partial s(A)$ is compact, giving a contradiction. 
\end{proof}

\subsection{\texorpdfstring{Steiner Formula and area measures of $\C$-convex domains}{Steiner Formula and area measures of C-convex domains}}

The goal of this subsection is to prove the following theorem.

\begin{theorem}[{First part of Theorem~\ref{theo intro Steiner}}, Local Steiner Formula for $\C$-convex domains]
\label{theo Steiner}
Let $K$ be a $\C$-convex domain. Then, there exists a unique family $S_0(K)$, $\dots$, $S_d(K)$ of Radon measures on $\P(\C^*)$ such that for all $\varepsilon>0$,
\begin{equation*}
V_\varepsilon(K) = \frac{1}{d+1}\sum_{i=0}^{d}\varepsilon^{d+1-i} \binom{d+1}{i} S_i(K) \, . 
\end{equation*}
\end{theorem}

Theorem~\ref{theo Steiner} then motivates to introduce the following definition.

\begin{definition}[Area measures of a $\C$-convex domain]
Let $K$ be a $\C$-convex domain. For all $0 \leq i \leq d$, the Radon measure $S_i(K)$ on $\P(\C^*)$ given by Theorem~\ref{theo Steiner} is called the \emph{area measure of $K$ of order $i$}. The measure $S_d(K)$ is called \emph{‘‘the'' area measure} of $K$. We shall denote it by $\Area(K)$.
\end{definition}

\begin{proposition}
For all $0\leq i \leq d$, the function mapping a $\C$-convex domain to its $i$-th area measure $S_i$ is invariant by translation and $i$-homogeneous. In the parametrisation $(\mathbf{P})$ that is, for all $\lambda > 0$, $\tau \in \R^{d+1}$ and $K$ a $\C$-convex domain,
\begin{equation*}
S_i(\lambda K + \tau)= \lambda^{i} S_i(K) \, . 
\end{equation*}
\end{proposition}
\begin{proof}
In the parametrisation $(\mathbf{P})$, let $K$ be $\C$-convex domain, $\lambda > 0$, and $\tau \in \R^{d+1}$. For all $\varepsilon>0$, using the $(d+1)$-homogeneity of the Lebesgue measure $\mathcal{L}$ on $\R^{d+1}$, and the definition of parallel neighbourhoods, one easily checks that
\begin{equation*}
    V_{\lambda\varepsilon}(\lambda K + \tau ) =  \lambda^{d+1} V_\varepsilon(K) \, .
\end{equation*}
Then, using the local Steiner Formula defining the area measures (Theorem~\ref{theo Steiner}) and the uniqueness of the coefficients of a polynomial function on $\R$, we get the wanted result.
\end{proof}

We shall prove Theorem~\ref{theo Steiner} by approximation. We will first prove it for $C^2_+$ $\C$-convex domains. Then, we will introduce the right notion of convergence for $\C$-convex domains and Radon measures on $\P(\C^*)$, and show that every $\C$-convex domain can be approximated by a sequence of $C^2_+$ $\C$-convex domains. Finally, we will show that it allows to extend the local Steiner Formula to the general case.

\subsubsection{\texorpdfstring{The $C^2_+$ case}{The C2+ case}} Let us start by proving the local Steiner Formula for $C^2_+$ $\C$-convex domains.

\begin{lemma}[Local Steiner Formula for $C^2_+$ $\C$-convex domains]
\label{lem Steiner formula for C2+}
Let $K$ be a $C^2_+$ $\C$-convex domain. Then, there exists a unique family $S_0(K)$, $\dots$, $S_d(K)$ of Radon measures on $\P(\C^*)$ such that for all $\varepsilon>0$,
\begin{equation*}
V_\varepsilon(K) = \frac{1}{d+1}\sum_{i=0}^{d}\varepsilon^{d+1-i} \binom{d+1}{i} S_i(K) \, .
\end{equation*}
\end{lemma}

We need a couple of intermediate results in order to prove Lemma~\ref{lem Steiner formula for C2+}. 

\begin{lemma}
\label{Foliation of neighbourhood}
Let $K$ be a $C^2_+$ $\C$-convex domain. Then, for any Borel subset $b$ of $\P(\C^*)$, we have
\begin{equation*}
\mathcal{N}_\varepsilon(K)(b) = \bigcup_{t \in (0,\varepsilon)} \G_{K+t\S}^{-1}(b) \, . 
\end{equation*}
\end{lemma}
\begin{proof}
In the parametrisation $(\mathbf{P})$, using the expression of $\mathcal{N}_\varepsilon(K)(b)$ given by \eqref{eq eps neigh C2+}, the expression of the Gauss map of a $C^2_+$ $\C$-convex domain given by \eqref{eq expression of the inverse of the Gauss map}, the expression of the $\C$-normal field given by \eqref{eq expression of the C-normal field}, and Proposition~\ref{prop support function and sum}, we get that for $b$ a Borel subset of $\Omega^*$,
\begin{equation*}
\mathcal{N}_\varepsilon(K)(b) = \left\lbrace \G_{K}^{-1}(y)+t \, N_{\S}(y)\st y\in b , \ t \in (0,\varepsilon )\right\rbrace= \left\lbrace \G_{K+t\S}^{-1}(y)\st y\in b , \ t \in (0,\varepsilon) \right\rbrace . \qedhere
\end{equation*}
\end{proof}
\begin{lemma}
\label{lem Jacobian computation}
In the parametrisation $(\mathbf{P})$, let $K$ be a $C^2_+$ $\C$-convex domain with support function $s$. Then, the map
\begin{equation*}
\function{\Phi}{\Omega^* \times \R^*_+}{K}{(y,t)}{\G_{K+t\S}^{-1}(y)} \, ,
\end{equation*}
induces a diffeomorphism between $U\times (0,\varepsilon)$ and $\mathcal{N}_\varepsilon(K)(U)$, for any open subset $U$ of $\Omega^*$ and $\varepsilon > 0$. Moreover, for all $y\in \Omega^*$ and $t>0$, we have
\begin{equation*}
 \left\vert \Jac_\Phi(y,t) \right\vert = -\omega_\C(y) \det\bp{\Hess (s + t \omega_\C)(y)} \, . 
\end{equation*}
\end{lemma}
\begin{proof}
For all $t\in \R_+$, let us denote $K_t = K+t\S$. It is a $\C$-convex domain domain (see Remark~\ref{rem K+S}) and by Proposition~\ref{prop support function and sum}, it is actually $C^2_+$ with support function $s_t \coloneqq s + t \omega_\C$. Notice that, using \eqref{eq expression of the inverse of the Gauss map}, the diffeomorphism $\G_{K_t}^{-1}$ can be expressed as
\begin{equation*}
\function{\G_{K_t}^{-1}}{\Omega^*}{\partial K_t \subset \R^{d+1}}{y}{\left(\begin{matrix} \grad s(y)+ t\grad \omega_\C(y) \\ \grad s(y)\cdot y + t \grad \omega_\C (y) \cdot y -(s+t\omega_\C)(y)
\end{matrix}\right)} \, . 
\end{equation*}
By Lemma~\ref{Foliation of neighbourhood}, for any open subset $U$ of $\Omega^*$ and $\varepsilon > 0$, we have $\Phi (U \times (0,\varepsilon))= \mathcal{N}_\varepsilon(K)(U)$. In order to conclude, we only have to compute the determinant of the Jacobian matrix of $\Phi$ at any point, and notice that it never vanishes. 

Let $y\in \Omega^*$ and $t>0$. On one side, 
\begin{equation*}
\partial_t \Phi(y,t) = \partial_t \G_{K_t}^{-1}(y) = \left(\begin{matrix} \grad \omega_\C(y) \\ \grad \omega_\C (y) \cdot y -\omega_\C(y)
\end{matrix}\right) \, . 
\end{equation*}
On the other side, for all $i \in \{1,\dots ,d\}$,
\begin{equation*}
\partial_{y_i} \Phi(y,t) = \left(\begin{matrix} \partial_{y_i}\grad s_t(y) \\ \partial_{y_i} \grad s_t (y) \cdot y + \grad s_t (y) \cdot e_i - \partial_{y_i} s_t(y)
\end{matrix}\right) = \left(\begin{matrix} \partial_{y_i} \grad s_t(y) \\ \partial_{y_i} \grad s_t (y) \cdot y
\end{matrix}\right)\, . 
\end{equation*}
Thus, 
\begin{equation*}
\det \Jac_\Phi(y,t) = \det\left(\begin{array}{c|c}
\Hess s_t(y) & \grad \omega_\C(y)\\ 
\hline \partial_{y_1} \grad s_t (y) \cdot y \ \vline \ \cdots \ \vline \ \partial_{y_d} \grad s_t (y) \cdot y & \grad \omega_\C (y) \cdot y -\omega_\C(y) 
\end{array}\right) \, . 
\end{equation*}
By $d$-linearity of the determinant,
\begin{multline*}
\det \Jac_\Phi(y,t) = \det\left(\begin{array}{c|c|c|c}
\partial_{y_1} \grad s_t (y) & \cdots & \partial_{y_d} \grad s_t (y) & \grad \omega_\C(y)\\
\hline \partial_{y_1} \grad s_t (y) \cdot y & \cdots & \partial_{y_d} \grad s_t (y) \cdot y & \grad \omega_\C (y) \cdot y
\end{array}\right) \\
 - \det\left(\begin{array}{c|c}
\Hess s_t(y) & 0\\
\hline \partial_{y_1} \grad s_t (y) \cdot y \ \vline \ \cdots \ \vline \ \partial_{y_d} \grad s_t (y) \cdot y & \omega_\C(y)
\end{array}\right) \, . 
\end{multline*}
Because its lines satisfy $L_{d+1} = y_1 L_1 + \dots+ y_n L_d$, 
\begin{equation*}
\det\left(\begin{array}{c|c|c|c}
\partial_{y_1} \grad s_t (y) & \cdots & \partial_{y_d} \grad s_t (y) & \grad \omega_\C(y)\\
\hline \partial_{y_1} \grad s_t (y) \cdot y & \cdots & \partial_{y_d} \grad s_t (y) \cdot y & \grad \omega_\C (y) \cdot y
\end{array}\right)=0 \, . 
\end{equation*}
Moreover, we have
\begin{equation*}
\det\left(\begin{array}{c|c}
\Hess s_t(y) & 0\\
\hline \partial_{y_1} \grad s_t (y) \cdot y \ \vline \ \cdots \ \vline \ \partial_{y_d} \grad s_t (y) \cdot y & \omega_\C(y)
\end{array}\right) = \omega_\C(y) \det\bp{\Hess s_t(y)}< 0 \, . 
\end{equation*}
So we finally get 
\begin{equation*}
\left\vert \Jac_\Phi(y,t) \right\vert = -\omega_\C(y) \det\bp{\Hess s_t(y)} = -\omega_\C(y) \det\bp{\Hess (s + t \omega_\C)(y)} > 0 \, . \qedhere
\end{equation*}
\end{proof}

\begin{proof}[Proof of Lemma~\ref{lem Steiner formula for C2+}]
Let us work in the parametrisation $(\mathbf{P})$. By Lemma~\ref{lem Jacobian computation}, for any open subset $U \subseteq \Omega^*$,
\begin{equation*}
V_\varepsilon(K)(U) = \int_{\mathcal{N}_\varepsilon(K)(U)} \dif x = \int_U \left\vert \int_0^\varepsilon \Jac_\Phi(y,t) \right\vert \dif t \dif y = \int_U \int_0^\varepsilon (-\omega_\C)(y) \det\Bp{\Hess (s + t \omega_\C)(y)} \dif t \dif y \, .
\end{equation*}
Thus, by the uniqueness part of the Carath\'eodory's Extension Theorem on $\R^d$, that the following measures are equal 
\begin{equation}
\label{eq Veps before radii}
V_\varepsilon(K) = \vp{\int_0^\varepsilon (-\omega_\C)\det\bp{\Hess (s + t \omega_\C)} \dif t} \mathcal{L} \, ,
\end{equation}
where $\mathcal{L}$ is the Lebesgue measure on $\Omega^* \subseteq \R^d$. As for all $y \in \Omega^*$, $\Hess(s)(y)$ and $\Hess (\omega_\C)(y)$ are positive definite symmetric matrices, the determinant $\det (\Hess (s+t\omega_\C)(y))$ is a degree $d$ polynomial in $t$ with only non-negative coefficients (continuous in $y$). Hence, for all $y \in \Omega^*$,
\begin{equation*}
 f(\varepsilon) \coloneqq -\int_0^\varepsilon \omega_\C(y) \det\bp{\Hess (s + t \omega_\C)(y)} \dif t
\end{equation*}
is then a degree $d+1$ polynomial in $\varepsilon$ with non-negative coefficients (continuous in $y$) and with zero as constant coefficient. Hence, identity~\eqref{eq1 proof radii} can be rewritten as 
\begin{equation*}
V_\varepsilon(K) = \frac{1}{d+1}\sum_{i=0}^d \binom{d+1}{i}\varepsilon^{d+1-i} f_i(K) \, \mathcal{L} \, ,
\end{equation*}
where the $f_i(K) : \Omega^* \to \R^*_+$ are non-negative continuous functions. Setting $S_i(K) \coloneqq f_i(K) \, \mathcal{L}$, which is Radon by continuity of $f_i$, we get the wanted result. The uniqueness of the measures $S_i$'s is a direct consequence of the uniqueness of the coefficient of a polynomial function on $\R$.
\end{proof}

\subsubsection{\texorpdfstring{Convergence of $\C$-convex domains and Radon measures}{Convergence of C-convex domains and Radon measures}}

We shall prove Theorem~\ref{theo Steiner} by approximation, so let us introduce the right notions of convergence for our setting.

On one hand, we shall consider the \emph{epi-convergence} of $\C$-convex domain. 

\begin{definition}[Epi-convergence of $\C$-convex domains]
\label{def epi-convergence}
A sequence $(K_n)_{n\in \N} $ of $\C$-convex domains of $\A^{d+1}$ is said to \emph{epi-converge} toward a $\C$-convex domain $K$ if in the parametrisation $(\mathbf{P})$ described in Section~\ref{sec (P)} their respective support functions converges uniformly to the support function of $K$ on every compact subset of $\Omega^*$.
\end{definition}

With a few computations, one can prove that that epi-convergence is truly an affine convergence: the epi-convergent nature of a sequence of $\C$-convex domains of $\A^{d+1}$ does not depend of the choice of the parametrisation $(\mathbf{P})$ of $\A^{d+1}$.

\begin{remark}
By the work of Beer, Rockafellar and Wets \cite{Beer_Rock_Wets_92}, the epi-convergence of $\C$-convex domains is actually equivalent to their convergence for the \emph{Kuratowski--Painlev\'e topology} on subsets of $\A^{d+1}$.
\end{remark}

Hence, in the parametrisation $(\mathbf{P})$, epi-convergence is the \emph{local uniform convergence} of support functions. That is the right notion of convergence to consider because of the following convex analysis theorem. 

\begin{theorem}[{\cite[Theorems 7.17 and 11.34]{Rockafellar_Wets_98}}]
\label{theo cv iff cv*}
Let $(s_n)_{n\in \N}$ be a family of support functions on $\Omega^*$. Then, the sequence $(s_n)_{n\in \N}$ converges locally uniformly towards a support function $s$ on $\Omega^*$ if and only if the sequence of their conjugates $(s_n^*)_{n\in \N}$ converges locally uniformly on $\R^d$ to $s^*$. 
\end{theorem}

Then, a key fact to prove Theorem~\ref{theo Steiner} will be the following $C^2_+$ approximation lemma. 

\begin{lemma}
\label{lem approx}
Let $K$ be a $\C$-convex domain. Then, there exists a sequence of $C^2_+$ $\C$-convex domain epi-converging to $K$.
\end{lemma}
\begin{proof}
Let us work in $(\mathbf{P})$ and let $s$ be a support function of a $\C$-convex. We want to prove that there exists a sequence of $C^2_+$ support functions converging locally uniformly to $s$ on $\Omega^*$. This is {\cite[Lemma~2.16]{Bonsante_Fillastre_17}}.
\end{proof}

\begin{remark}
In fact, the proof of {\cite[Lemma~2.16]{Bonsante_Fillastre_17}} by Bonsante and Fillastre is written in the case where the cone $\C$ is quadratic, but it directly translates to our more general setting. Throughout the rest of the article, every time we refer to \cite{Bonsante_Fillastre_17} for a fact, it means that their proof was not relying on any hyperbolic geometry (i.e. on the fact that the cones $\C$ and $\C^*$ are quadratic, so that $\S = \S^*=\H^d$) and thus that it can immediately be adapted to our more general case. 
\end{remark}

On the other hand, we shall consider the \emph{vague convergence}\footnote{Vague convergence of Radon measures on $\P(\C^*)$ is a form of weak*-convergence using the duality between Radon measures on $\P(\C^*)$, and $C^0_c(\P(\C^*))$, the space of compactly supported continuous function on $\P(\C^*)$. It is sometime called weak convergence, but that can cause confusion as, often, weak convergence designates the (stronger) weak*-convergence obtained through the duality between Radon measures on $\P(\C^*)$ and $C^0_b(\P(\C^*))$, the space of bounded continuous function on $\P(\C^*)$. } 
of Radon measures on $\P(\C^*)$.

\begin{definition}[Vague convergence of Radon measures]
A sequence $(\mu_n)_{n\in \N}$ of Radon measures on $\P(\C^*)$ converges vaguely towards another Radon measure $\mu$ on $\P(\C^*)$, if for all $f \in C^0_c(\P(\C^*))$, a compactly supported continuous function on $\P(\C^*)$, 
\begin{equation*}
\int_{\P(\C^*)} f \,\dif \mu_n \xrightarrow[n\to +\infty]{} \int_{\P(\C^*)} f \,\dif \mu \, . 
\end{equation*}
\end{definition}

The epi-convergence of $\C$-convex domains and the vague convergence of their volume measures are related through the following fact. 

\begin{lemma}
\label{lem vague cv of Veps}
Let $(K_n)_{n\in\N}$ be a sequence of $\C$-convex domains epi-converging to a $\C$-convex domain $K$. Then, for all $\varepsilon>0$, the sequence of measures $(V_\varepsilon(K_n))_{n\in\N}$ converges vaguely to $V_\varepsilon(K)$. 
\end{lemma}
\begin{proof}
Let us work in the parametrisation $(\mathbf{P})$ described in Section~\ref{sec (P)}. The hypothesis becomes: the respective support functions $(s_n)_{n\in\N}$ of the domains $(K_n)_{n\in\N}$ converge locally uniformly on $\Omega^*$ to $s$, the support function of $K$. We want to show that the sequence of Radon measures $(V_\varepsilon(K_n))_{n\in\N}$ on $\Omega^*$ converges vaguely to $V_\varepsilon(K)$. 

Let $\varepsilon>0$. By a classical measure theory fact \cite[Theorem~1.40]{Evans_Gariepy_15}, we just have to prove that:
\begin{enumerate}[label=(\arabic*)]
 \item \label{pt limsup} for any compact set $A\subseteq \Omega^*$, we have $\limsup_{n} V_\varepsilon(K_n)(A) \leq V_\varepsilon(K)(A)$,
 \item \label{pt liminf} for any open set $U\subseteq \Omega^*$, we have $V_\varepsilon(K)(U) \leq \liminf_{n} V_\varepsilon(K_n)(U)$. 
\end{enumerate}

For point~\ref{pt limsup}, let $A\subseteq \Omega^*$ be a compact set. We claim that 
\begin{equation}
\label{claim1}
 \limsup_{n \to +\infty} \mathcal{N}_\varepsilon(K_n)(A) \subseteq \overline{\mathcal{N}_\varepsilon(K)(A)} \, . 
\end{equation}
Thus, we have,
\begin{equation*}
\mathcal{L}\vp{\limsup_{n\to+\infty} \mathcal{N}_\varepsilon(K_n)(A)} \leq \mathcal{L}\bp{\overline{\mathcal{N}_\varepsilon(K)(A)}} = \mathcal{L}\bp{\mathcal{N}_\varepsilon(K)(A)} + \mathcal{L}\bp{\G^{-1}_K(A)}+ \mathcal{L}\bp{\G^{-1}_{K + \varepsilon \S}(A)} \, . 
\end{equation*}
Note that $\G^{-1}_K(A) \subseteq \partial K$ and $\G^{-1}_{K + \varepsilon \S}(A) \subseteq \partial (K+ \varepsilon \S)$. Thus, as the boundary of a convex set has zero Lebesgue measure \cite[Theorem~1]{Lang_86}, we have 
\begin{equation*}
\mathcal{L}\bp{\overline{\mathcal{N}_\varepsilon(K)(A)}} = \mathcal{L}\bp{\mathcal{N}_\varepsilon(K)(A)} = V_\varepsilon(K)(A) \, .
\end{equation*}
That concludes the proof of point~\ref{pt limsup}, as then
\begin{equation*}
 \limsup_{n\to+\infty} V_\varepsilon(K_n)(A) \leq \mathcal{L}\vp{\limsup_{n\to+\infty} \mathcal{N}_\varepsilon(K_n)(A)} \leq \mathcal{L}\bp{\overline{\mathcal{N}_\varepsilon(K)(A)}} = V_\varepsilon(K)(A) \, . 
\end{equation*}

Now, let us prove our claim that the inclusion~\eqref{claim1} is true. Let $X \in \limsup_n \mathcal{N}_\varepsilon(K_n)(A)$. Then, there exist an increasing sequence $(n_k)_{k \in \N}$ of integers such that for all $k\in \N$,
\begin{equation}
\label{eq dec cosm nk}
 X = P_{n_k} + t_{n_k} \, N_{\S}(y_{n_k}) \, ,
\end{equation}
where $t_{n_k}\in (0,\varepsilon)$, $y_{n_k} \in A$, and $P_{n_k} = \bp{p_{n_k},s_{n_k}^*(p_{n_k})}$ with $y_{n_k} \in \partial s^*_{n_k}(p_{n_k})$. By the Fenchel Inequality Theorem \cite[Theorem~23.5]{Rockafellar_97}, each $p_{n_k}$ is in $\partial s_{n_k}(y_{n_k})$. As $A$ and $[0,\varepsilon]$ are compact, up to extracting another subsequence, we can assume that we have $y \in A$ and $t \in [0,\varepsilon]$ such that $y_{n_k} \xrightarrow[]{} y$ and $ t_{n_k} \to t$. By \cite[Theorem~24.5]{Rockafellar_97}, for all $\delta>0$, there is a rank $k_\delta \in \N$ such that for all $k\geq k_\delta$, 
\begin{equation*}
 p_{n_k} \in \partial s(y) + B_\delta \, ,
\end{equation*}
where $B_\delta$ is the Euclidean ball of radius $\delta$. Moreover, as $\partial s(y)$ is compact \cite[Theorem~A.22]{Figalli_17}, up to extracting again we can assume that there is $p \in \partial s(y)$ such that $p_{n_k} \to p$. Thus, setting $P = (p,s^*(p))$, we have, taking the limit of equation~\eqref{eq dec cosm nk}, that 
\begin{equation*}
 X = P + t \, N_{\S}(y) \, ,
\end{equation*} 
where $t\in [0,\varepsilon]$, $y \in A$, and $P = (p,s^*(p))$ with $p \in \partial s(y)$. That is $X \in \overline{\mathcal{N}_\varepsilon(K)(A)}$, proving our claim. 

For point~\ref{pt liminf}, let $U\subseteq\Omega^*$ be an open set. 
We claim that
\begin{equation}
\label{claim2}
 \mathcal{N}_\varepsilon(K)(U) \subseteq \liminf_{n\to +\infty} \mathcal{N}_\varepsilon(K_n)(U) \, . 
\end{equation}
Then, that implies that 
\begin{equation*}
 V_\varepsilon(K)(U) = \mathcal{L}\bp{\mathcal{N}_\varepsilon(K)(U)} \leq \mathcal{L}\vp{\liminf_{n\to+\infty} \mathcal{N}_\varepsilon(K_n)(U)} \leq \liminf_{n\to+\infty} V_\varepsilon(K_n)(U) \, ,
\end{equation*}
concluding this point. 

Now, let us prove our claim that inclusion~\eqref{claim2} is true. Let $X \in \mathcal{N}_\varepsilon(K)(U)$, we have to show that there exists $n_{X} \in \N$ such that for all $n \geq n_{X}$ we have $X \in \mathcal{N}_\varepsilon(K_n)(U)$. 

Decompose $X$ as $X = (x,\lambda)\in \R^d\times \R$ and set $t \in (0,\varepsilon)$, $y \in U$, and $P\in \G^{-1}_K(y)$, such that 
\begin{equation*}
 X = P + t \, N_{\S}(y) \, . 
\end{equation*}
Then, $t \in (0,\varepsilon)$ is such that $X \in \graph (s+t\omega_\C)^*$, so $\lambda = (s+t\omega_\C)^*(x) > s^*(x)$. 

By Theorem~\ref{theo cv iff cv*}, $s_n^*$ converges pointwise to $s^*$, so for $n$ big enough, $\lambda > s_n^*(x)$ and $X \in K_n = \epi(s_n^*)$, and we can consider the normal projection \cite[section 6]{Ablondi_25} of $X$ onto $\partial_{sp}K_n$: there are $t_n \in \R$, $y_n \in \Omega^*$, and $P_n\in \G^{-1}_{K_n}(y_n)$, such that
\begin{equation*}
 X = P_n + t_n \, N_{\S}(y_n) \, . 
\end{equation*}

Now let us show that $t_n \to t$. Let $\delta \in (0,t)$. As, by Theorem~\ref{theo cv iff cv*},
\begin{equation*}
(s_n +(t\pm \delta) \, \omega_\C)^*(x) \xrightarrow[n \to +\infty]{} (s +(t\pm \delta) \omega_\C)^*(x) \, , 
\end{equation*}
and 
\begin{equation*}
(s +(t- \delta) \,\omega_\C)^*(x)< \lambda =(s +t\omega_\C)^*(x) <(s +(t + \delta) \,\omega_\C)^*(x) \, ,
\end{equation*}
for $n$ big enough, we have 
\begin{equation*}
(s_n +(t- \delta) \, \omega_\C)^*(x)< \lambda <(s_n +(t + \delta) \omega_\C)^*(x) \, ,
\end{equation*}
implying that $t_n$, which is such that $\lambda = (s_n +t_n\omega_\C)^*(x)$, is in the interval $(t-\delta,t+\delta)$. 

Now knowing that $t_n \to t$, as, by \cite[Proposition~6.7]{Ablondi_25}, $y_n$ satisfies 
\begin{equation*}
 y_n = \grad(s_n +t_n \omega_\C)^*(x) \, ,
\end{equation*}
by \cite[Theorem~24.5]{Rockafellar_97}, we have $y_n \to y$. As $t \in (0,\varepsilon)$ and $y$ is in the open set $U$, there exists $n_X \in \N$ such that for all $n\leq n_X$ we have $t_n\in (0,\varepsilon)$ and $y_n\in U$, i.e. $X \in \mathcal{N}_\varepsilon(K_n)(U)$, proving the claim. 
\end{proof}

\subsubsection{Proof of the Steiner formula}

We now have all the tools needed to prove the local Steiner Formula in the most general setting. 

\begin{proof}[Proof of Theorem~\ref{theo Steiner}]
Let $K$ be a $\C$-convex domain. By Lemma~\ref{lem Steiner formula for C2+}, we can use a Lagrange polynomial interpolation with $d+1$ fixed distinct positive real numbers $\varepsilon_0,\dots, \varepsilon_{d}$, in order to get real numbers $(a_{ik})_{0\leq i,k\leq d}$ (independent of any choice of a $\C$-convex domain), such that for any $C^2_+$ $\C$-convex domain $K'$ we have 
\begin{equation}
\label{eq polarisation Si}
 S_i(K') = \sum_{k=0}^{d} a_{ik} V_{\varepsilon_k}(K') \, . 
\end{equation}
Thus, we set $S_i(K)$ to be the (a priori signed) Radon measure on $\P(\C^*)$ defined by 
\begin{equation}
\label{eq def Si}
 S_i(K) \coloneqq \sum_{k=0}^{d} a_{ik} V_{\varepsilon_k}(K) \, . 
\end{equation}
By Lemma~\ref{lem approx}, we can set $(K_n)_{n\in\N}$ to be a sequence of $C^2_+$ $\C$-convex epi-converging towards $K$. By Lemma~\ref{lem vague cv of Veps}, for any $\varepsilon$, the measures for $V_{\varepsilon}(K_n)$ converges vaguely to $V_{\varepsilon}(K)$. Thus, by \eqref{eq polarisation Si} and \eqref{eq def Si}, the $S_i(K_n)$'s converges vaguely to the $S_i(K)$'s, which are non-negative Radon measures because the $V_{\varepsilon_k}(K)$ are Radon measures by Proposition~\ref{prop Veps is Radon}. At the vague limit of the Steiner Formula for the $\C^2_+$ $\C$-convex domains $(K_n)_n$ we get: 
\begin{equation*}
V_\varepsilon(K) = \frac{1}{d+1}\sum_{i=0}^{d}\varepsilon^{d+1-i} \binom{d+1}{i} S_i(K) \, . \qedhere
\end{equation*}
\end{proof}

\begin{corollary}
\label{cor vague limit Veps/eps}
Let $K$ be a $\C$-convex domain. Then, its area measure $\Area(K) \coloneqq S_d(K)$ is the following vague limit of Radon measures:
\begin{equation*}
 \frac{V_\varepsilon(K)}{\varepsilon} \xrightarrow[\varepsilon \to 0]{v} \Area(K) \, . 
\end{equation*}
\end{corollary}

Lemma~\ref{lem vague cv of Veps} and the proof of Theorem~\ref{theo Steiner} also imply the following. 

\begin{proposition}
\label{prop vague convergence of Si}
Let $(K_n)_{n\in\N}$ be a sequence of $\C$-convex domains epi-converging towards $K$. Then, for all $0 \leq i \leq d$, the sequence of measures $(S_i(K_n))_{n\in\N}$ converges vaguely to $S_i(K)$. 
\end{proposition}

\section{Relation between area measures, curvature and Monge--Amp\`ere measures}
\label{sec relation area measure curvature and MA}

In this section we shall exhibit the relation between the area measures introduced in Section~\ref{sec area meas} with a well-suited notion of curvature, called \emph{$\C$-curvature}, as well as their relation with the Monge--Amp\`ere measures of support functions. While doing so, we will prove the second part of Theorem~\ref{theo intro Steiner} (Corollary~\ref{cor S0}) and a stronger version of Theorem~\ref{theo intro C-curv AGC}.

In order to introduce $\C$-curvature for $\C$-spacelike boundaries of $C^2_+$ $\C$-convex domains, we shall use the theory of \emph{affine differential geometry} \cite{Nomizu_Sassaki_95,LSZH_15} which studies immersed hypersurfaces in the affine space through the following method. 

Consider a $C^2$ immersion $f: U \to \A^{d+1}$ of an open set $U \subseteq \R^d$. A smooth map $N: U \to \V^{d+1}$ transversal to $f$ allows one to decompose the tangent space at each point $f(p) \in \Sigma \coloneqq f(U)$ as the direct sum 
\begin{equation*}
\mathrm{T}_{f(p)}\A^{d+1} = \mathrm{T}_{f(p)}\Sigma \oplus \bigl\langle N(p)\bigr\rangle \, .
\end{equation*}
Using that decomposition, one can define the \emph{affine invariants} of $(f,N)$, which are differential forms on $U$, by studying the coefficients, in that decomposition, of the flat affine connection $D$ on $\A^{d+1}$. 

\subsection{\texorpdfstring{$\C$-curvature via affine differential geometry}{C-curvature via affine differential geometry}}
\label{subsec Affine differential geometry}

Let $K$ be a $C^2_+$ $\C$-convex domain in $\A^{d+1}$. The inverse of its Gauss map, gives a parametrisation of $\partial_{sp}K$ by the embedding $\G_K^{-1} : \P(\C^*) \hookrightarrow\A^{d+1}$. Then, the $\C$-normal field $N_{\S} : \P(\C^*) \to \V^{d+1}$ induced by the affine sphere $\S$ (Definition~\ref{def C-normal vf}) provides a natural transverse vector field on the hypersurface $\partial_{sp}K$. 

That transverse vector field on $\partial_{sp}K$ can be described in the following way: it associates to every point $P$ the vector $\vec{n} \in \S$ such that the tangent hyperplane of $\S$ at $\vec{n}$ has the same direction as the one of $\partial K $ at $P$. See Figure~\ref{fig normal}.

\begin{figure}[ht]
 \centering
 \includegraphics{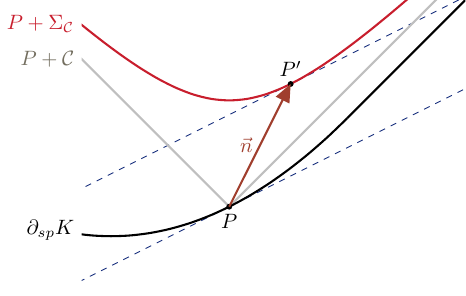}
 \caption{At point $P \in \partial_{sp} K$, the transverse $\C$-normal vector $\vec{n}$ is given by the vector $\overrightarrow{PP'}$, where $P'$ is the point of $P + \S$ with tangent hyperplane parallel to $\mathrm{T}_P\partial_{sp}K$. }
 \label{fig normal}
\end{figure}

Given a $C^2_+$ $\C$-convex domain $K$, let us then consider the affine invariants of the pair $(f^K \coloneqq \G_K^{-1} , N \coloneqq N_{\S})$. Using the terminology of \cite[Sections 1.14 and 1.15]{LSZH_15}, we are actually considering the affine differential geometry of $\G_K^{-1} : \P(\C^*) \hookrightarrow\A^{d+1}$ relative to the normalisation $(N_{\S}, N_{\S}^*)$ given by the affine sphere $\S$ (and its polar affine sphere $\S^*$). See Remark~\ref{rem intro curvature} to understand how this construction relates with the Euclidean and Minkowski cases.

\begin{definition}[$\C$-affine differential invariants]
\label{def C-invariants}
Let $K$ be a $C^2_+$ $\C$-convex domain. The \emph{$\C$-affine invariants} of $K$ are the affine invariants of the pair $(f^K \coloneqq \G_K^{-1} , N \coloneqq N_{\S})$. That is, the following differential tensors on $\P(\C^*)$:
\begin{itemize}
\item the \emph{induced $\C$-connection} $\nabla^K$ and \emph{induced $\C$-fundamental form} $h_K \in \Gamma(\mathrm{T}^*\P(\C^*) \otimes \mathrm{T}^*\P(\C^*))$, determined by
\begin{equation*}
D_{(f^K_*X)} (f^K_*Y) =\nabla^K_X Y + h_K(X,Y) N \, ,
\end{equation*}
where here and below, $X$, $Y$ and $X_1,\dots, X_n$ are any vector fields on $\P(\C^*)$,
\item the \emph{$\C$-shape operator} $S_K \in \Gamma(\mathrm{T}^*\P(\C^*) \otimes \mathrm{T}\P(\C^*))$ and \emph{$\C$-transversal connection form} $\tau_K \in \Gamma(\mathrm{T}^*\P(\C^*))$, determined by 
\begin{equation*}
D_{(f^K_*X)} N =S_K(X) + \tau_K(X) N \, ,
\end{equation*}
\item the \emph{Gauss-Kronecker $\C$-curvature} $\phi_K \in C^\infty(\P(\C^*))$, defined by
\begin{equation*}
\phi_K = \det S_K \, ,
\end{equation*}
\item the \emph{induced $\C$-volume form $\nu \in \Gamma(\bigwedge^{d+1}\mathrm{T}^*\P(\C^*))$}, defined by 
\begin{equation*}
\nu_K(X_1, \dots, X_d) = \dvol \bp{ \dif f^K(X_1), \dots, \dif f^K(X_1),N} \, . 
\end{equation*}
\item the \emph{volume form $\dvol_{h_K} \in \Gamma(\bigwedge^{d+1}\mathrm{T}^*(\Omega^*))$ induced by} $h_K$. 
\end{itemize} 
\end{definition}

\begin{remark}
Note that all those operators are not directly defined on the hypersurface $\partial_{sp}K$, but on $\P(\C^*)$ by pull-back by the Gauss map $\G_K : \P(\C^*) \to \partial_{sp}K$. 
\end{remark}

Using the parametrisation $(\mathbf{P})$ described in Section~\ref{sec (P)}, the explicit expression of the $\C$-normal field given by \eqref{eq expression of the C-normal field} and the one of the inverse of the Gauss map given by \eqref{eq expression of the inverse of the Gauss map}, computations \cite[proof of Proposition~7.1]{Nie_Seppi_22} give the following. 

\begin{proposition}
\label{prop aff invariants relative aff sphere}
In the parametrisation $(\mathbf{P})$, let $K$ be a $C^2_+$ $\C$-convex domain with support function $s$. Its $\C$-affine invariants are the following tensors on $\Omega^* \subset \R^{d}$:
\begingroup
\allowdisplaybreaks
\begin{align*}
& h_K= -\frac{1}{\omega_\C} \Hess s \, ,\\
& S_K = (\Hess s)^{-1}\Hess\omega_\C \, ,\\
& \tau_K=0 \, ,\\
& \phi_K = \det S_K = (-\omega_\C)^{-d-2} \det(\Hess s)^{-1} \, ,\\
& \nu_K= -\omega_\C \det(\Hess s) \dif \mathcal{L} \, ,\\
& \dvol_{h_K} = (-\omega_\C) ^{-d/2} \det(\Hess s) ^{1/2} \dif \mathcal{L} \, ,
\end{align*}
\endgroup
where $\omega_\C$ is the support function of the affine sphere $\S$, $\mathcal{L}$ is the Lebesgue measure on $\Omega^* \subseteq \R^d$, and, in the first identity we see the Hessian as the bilinear form, i.e. a $(2,0)$-tensor, defined by: for all $v,w \in \mathrm{T}_y \Omega^*= \R^d$,
\begin{equation*}
\Hess f(y) (v,w) = \sum_{1\leq i,j \leq d} \partial_i \partial_j f(y) \, u_i v_i \, ,
\end{equation*}
while for all the others equation we consider it to be a field of $d \times d$ matrices, i.e. a $(1,1)$-tensor, defined by: for all $y \in \Omega^*$,
\begin{equation*}
\bp{\Hess f(y) }_{1\leq i,j \leq d} = \bp{\partial_i \partial_j f(y)}_{1\leq i,j \leq d} \, . 
\end{equation*}
\end{proposition}

As expected, when $\partial K$ is any translation of the affine sphere $\S$, one notices that $S_K = S_{\S} = \Id$ and $\phi_K=\phi_{\S} = 1$. Moreover, because of the Monge--Amp\`ere equation~\eqref{eq MA aff sphere} satisfied by $\omega_\C$, in that case the two volume forms $\dif \Area(\S) =\nu_{\S}$ and $\dvol_{h_{\S}}$ coincide and are equal to $(-\omega_\C) ^{-d-1} \dif \mathcal{L}$, where $\mathcal{L}$ is the Lebesgue measure on $\Omega^* \subseteq \R^d$. Up to the identification $\Omega^* \simeq \S$ (given by the diffeomorphism $N_{\S}$), this volume form is the canonical volume from $\dvol_{\S}$ on $\S$ evoked in subsection~\ref{subsec Cheng-Yau affine sphere}.

\begin{remark}
In fact, when $\partial K$ is any translation of the affine sphere $\S$, the transverse vector field chosen is precisely the \emph{Blaschke affine normal} field on $\partial K$ given by the theory of affine differential geometry \cite{Nomizu_Sassaki_95,Loftin_10,LSZH_15}. 
\end{remark}

\begin{definition}[$\C$-curvature of a $C^2_+$ $\C$-convex domain]
Let $K$ be a $C^2_+$ $\C$-convex domain. We shall refer to the eigenvalues $k_1,\cdots, k_d$ of the $\C$-shape operator $S_K$ given by Proposition~\ref{prop aff invariants relative aff sphere}, as the \emph{principal $\C$-curvatures} of the hypersurface $\partial_{sp} K$, and to their inverses $r_1\coloneqq 1/k_1,\, \dots \, ,r_d\coloneqq 1/k_d$ as the \emph{principal radii of $\C$-curvature} of $\partial_{sp} K$. 
\end{definition} 

The $\C$-curvature and the principal $\C$-curvatures of a $C^2_+$ $\C$-convex domain thus satisfy
\begin{equation*}
 \phi_K = k_1 k_2 \cdots k_d = \frac{1}{r_1 r_2 \cdots r_d} \, . 
\end{equation*}

\begin{proposition}
The $\C$-curvature is invariant by translation and is $(-d)$-homogeneous. In the parametrisation $(\mathbf{P})$, that is, for all $\lambda > 0$, $\tau \in \R^{d+1}$ and $K \subset \R^{d+1}$ a $C^2_+$ $\C$-convex domain,
\begin{equation*}
\phi_{\lambda K + \tau}= \lambda^{-d} \phi_K. 
\end{equation*}
\end{proposition}
\begin{proof}
Let $\tau \in \R^{d+1}$, $\lambda > 0$ and $K$ be a $C^2_+$ $\C$-convex domain with support function $s$. Introducing $a_\tau : \R^d \to \R$ the affine function define by $a_\tau(y) = \tau \cdot (y,-1)$, we have that the support function of $K' \coloneqq \lambda K + \tau$ is $s' = \lambda s+a_\tau$. Thus,
\begin{align*}
\phi_{\lambda K +\tau} & = (-\omega_\C)^{-d-2} \det(\Hess s')^{-1} = (-\omega_\C)^{-d-2} \det\bp{\Hess(\lambda s+a_\tau)}^{-1} \\
& = (-\omega_\C)^{-d-2} \det(\lambda\Hess s)^{-1} = \lambda^{-d} \phi_K \, . \qedhere
\end{align*}
\end{proof}

\begin{theorem}[Relation between area measures and $\C$-curvature]
\label{theo Steiner formula for C2+}
Let $K$ be a $C^2_+$ $\C$-convex domain. Then, in the local Steiner Formula for $K$:
\begin{equation*}
V_\varepsilon(K) = \frac{1}{d+1}\sum_{i=0}^{d}\varepsilon^{d+1-i} \binom{d+1}{i} S_i(K) \, ,
\end{equation*}
the Radon measures $S_0(K)$, $\dots$, $S_d(K)$ on $\P(\C^*)$ are satisfy
\begin{equation*}
S_i(K) = \sigma_i(K) \, \vol_{\S} \, ,
\end{equation*}
where $\sigma_i(K) : \P(\C^*) \to \R^*_+$ is the $i$-th normalised elementary symmetric function of $r_1,\dots,r_d$, the principal radii of $\C$-curvature of $\partial_{sp}K$, i.e. 
\begin{equation}
\label{eq def sym function}
\sigma_i(K) \coloneqq \binom{d}{i}^{-1}\sum_{1\leq j_1<\cdots<j_i\leq d} r_{j_1}\cdots r_{j_d} \, . 
\end{equation}
\end{theorem}

\begin{remark}
That generalises the Euclidean \cite[subsection 2.5 and Section 4]{Schneider_93} and Minkowski \cite[subsection 3.1.2]{Fillastre_Veronelli_16} cases where area measures and radii of curvature are related in a similar way. 
\end{remark}

\begin{proof}
Let us work in the parametrisation $(\mathbf{P})$ and let $\varepsilon>0$. Starting from the identity~\eqref{eq Veps before radii} in the proof of Lemma~\ref{lem Steiner formula for C2+}, we have 
\begin{equation}
\label{eq1 proof radii}
V_\varepsilon(K) = \vp{\int_0^\varepsilon (-\omega_\C)\det\bp{\Hess (s + t \omega_\C)} \dif t} \mathcal{L} \, ,
\end{equation}
where $\mathcal{L}$ is the Lebesgue measure on $\Omega^*$. Note that
\begin{equation*}
\det\bp{\Hess (s + t \omega_\C)} = \det(t \Hess \omega_\C + \Hess s) =\det(\Hess \omega_\C)\det\bp{tI + \Hess s (\Hess \omega_\C)^{-1}} \, . 
\end{equation*}
Using the Monge--Amp\`ere equation~\eqref{eq MA aff sphere} satisfied by $\omega_\C$, the support function of the affine sphere $\S$, and the expression of $S_K$ given by Proposition~\ref{prop aff invariants relative aff sphere}, that becomes
\begin{equation*}
\det\bp{\Hess (s + t \omega_\C)} = (-\omega_\C)^{-d-2} \det(tI + S_K^{-1}) \, ,
\end{equation*}
and thus, using that $\dvol_{\S} = (-\omega_\C)^{-d-1} \dif \mathcal{L}$, we have 
\begin{equation}
\label{eq2 proof radii}
(-\omega_\C)\det\bp{\Hess (s + t \omega_\C)} \dif \mathcal{L} = \det(tI + S_K^{-1}) \, \dvol_{\S} \, . 
\end{equation}
As the eigenvalues of the symmetric matrix $S_K^{-1}$ are $r_1(K),\dots,r_d(K)$ the principal radii of $\C$-curvature of $\partial_{sp}K$, we have 
\begin{equation}
\label{eq3 proof radii}
\det(tI + S_K^{-1}) = \sum_{i=0}^d \binom{d}{i}t^{d-i} \sigma_i(K) \, . 
\end{equation}
where $\sigma_i(K) : \Omega^* \to \R^*_+$ is the $i$-th elementary symmetric function of the principal radii of $\C$-curvature of $\partial_{sp} K$, defined by expression~\eqref{eq def sym function}.
Using equations~\eqref{eq1 proof radii},\eqref{eq2 proof radii} and~\eqref{eq3 proof radii}, we get that for any Borel subset $b \subseteq \Omega^*$ 
\begin{equation*}
V_\varepsilon(K) = \sum_{i=0}^d \binom{d}{i}\vp{\int_0^\varepsilon t^{d-i} \dif t} \sigma_i(K) \, \vol_{\S} = \frac{1}{d+1}\sum_{i=0}^d \binom{d+1}{i}\varepsilon^{d+1-i} \sigma_i(K) \, \vol_{\S} \, . \qedhere
\end{equation*}
\end{proof}

\begin{corollary}
\label{cor link A(K) vol}
Let $K$ be a $C^2_+$ $\C$-convex domain. Then,
\begin{align*}
& S_1(K) = \frac{1}{d}\mathrm{Tr}(S_K^{-1}) \, \vol_{\S} \, , \\
& S_d(K) = \Area(K) = \det(S_K^{-1}) \, \vol_{\S} = \frac{1}{\phi_K} \,\vol_{\S} \, .
\end{align*}
\end{corollary}

\begin{remark}
Using Proposition~\ref{prop aff invariants relative aff sphere}, in $(\mathbf{P})$, we have $\nu_K = -\omega_\C \det(\Hess s) \mathcal{L}$, $ \det(\Hess s) = \phi_K^{-1} (-\omega_\C)^{-d-2}$ and $ \vol_{\S} = (-\omega_\C)^{-d-1} \mathcal{L}$ so that 
\begin{equation*}
    \nu_K = \frac{1}{\phi_K} \,\vol_{\S} \, .
\end{equation*}
Thus, the area measure of a $C^2_+$ $\C$-convex domain $K$ is equal to $\nu_K$, the induced $\C$-volume form given by Definition~\ref{def C-invariants}.
\end{remark}

\begin{corollary}[Second part of Theorem~\ref{theo intro Steiner}]
\label{cor S0}
Let $K$ be a $\C$-convex domain. Then, 
\begin{equation*}
S_0(K) = \vol_{\S} \, ,
\end{equation*}
\end{corollary}
\begin{proof}
For $C^2_+$ $\C$-convex domains, it is a straightforward consequence of Theorem~\ref{theo Steiner formula for C2+}. By Lemma~\ref{lem approx} and Proposition~\ref{prop vague convergence of Si}, it is also true in the general case.
\end{proof}

Finally, let us notice that like Euclidean and Minkowski Gaussian curvature, the $\C$-curvature comes from a relative differential geometry, and thus differs from the classical \emph{affine Gaussian curvature} given by the affine differential geometry theory \cite[Definition~3.4]{Nomizu_Sassaki_95}, which is obtained by using the \emph{affine normal} (also called \emph{Blaschke normal}) to the hypersurface $\partial_{sp} K$ itself in order to get an affine differential structure on it. Nevertheless, the following theorem states that the notions of constant curvature in both models are related, again justifying our choice to consider affine invariants and curvature relative to the affine sphere $\S$, and not relative to any other smooth convex hypersurface asymptotic to the cone $\C$. 

\begin{theorem}[Stronger version of Theorem~\ref{theo intro C-curv AGC}]
\label{prop CAGC and CCcurv coincide}
Let $K$ be convex domain in $\A^{d+1}$ with $C^2$ boundary. Let $\partial_{nd} K$ denote the non-degenerate part of $\partial K$, that is the part of $\partial K$ which can locally be expressed as graph of a convex function $f$ with $\Hess(f)>0$. 

Then, $K$ is a $C^2_+$ $\C$-convex domain such that $\partial_{sp} K$ has constant $\C$-curvature $\kappa>0$ if and only if $\partial_{nd} K$ has constant Gaussian affine curvature $\kappa^{2(d+1)/(d+2)}$ and the affine sphere given by the Blaschke vector field on $\partial_{nd} K$ is asymptotic to $\C \subset \V^{d+1}$. In that case $\partial_{sp} K = \partial_{nd} K$. 
\end{theorem}
\begin{proof}
Let us work in the parametrisation $(\mathbf{P})$.

On one side, $K$ is a $C^2_+$ $\C$-convex domain such that $\partial_{sp} K$ has constant $\C$-curvature $\kappa>0$ if and only if it is the graph of the Legendre--Fenchel transform of a convex function $s$ on $\Omega^*$ satisfying the following Monge--Amp\`ere equation:
\begin{equation*}
\det(\Hess s) = \kappa^{-1}(-\omega_\C)^{-d-2} \, . 
\end{equation*}

On the other side, Nie and Seppi have shown \cite[Corollary~7.5]{Nie_Seppi_22} that $\partial_{nd} K$, the non-degenerate part of $\partial K$, has constant Gaussian affine curvature $k>0$ and the affine sphere given by the Blaschke vector field on $\partial_{nd} K$ is asymptotic to $\C$, if and only if it is the graph of the Legendre--Fenchel transform of a convex function $s$ on $\Omega^*$ satisfying the following Monge--Amp\`ere equation:
\begin{equation*}
\det(\Hess s) = k^{-\frac{(d+2)}{2(d+1)}}(-\omega_\C)^{-d-2} \, . \qedhere
\end{equation*}
\end{proof}

\subsection{Relation between area measures and Monge--Amp\`ere measures}

Let us end this section by highlighting how the area measure of a $\C$-convex domain is also very closely related to the \emph{Monge--Amp\`ere measure} of its support function, which is a classical measure associated with any convex function. 

\begin{definition}[Monge--Amp\`ere measure of a convex function]
Let $U$ be an open convex subset of $\R^d$ and $s : U \to \R$ be a convex function. The \emph{Monge--Amp\`ere measure} of $s$, denoted by $\MA(s)$, is a measure on $U$ defined in the following way. For all Borel subset $b \subseteq U$,
\begin{equation*}
\MA(s)(b) \coloneqq \mathcal{L}\bp{\partial s(b)} \, ,
\end{equation*}
where $\mathcal{L}$ is the Lebesgue measure on $U\subset \R^{d}$ and we recall that
\begin{equation*}
\partial s(b) \coloneqq \bigcup_{x \in b}\partial s(x) \, . 
\end{equation*}
\end{definition}

\begin{proposition}[{\cite[Theorem~A.31]{Figalli_17}}]
\label{prop MA=det Hess}
If $s$ is a $C^2$ convex function on an open convex subset $U \subset \R^d$, then 
\begin{equation*}
\MA(s) = \det( \Hess s ) \, \mathcal{L} \, . 
\end{equation*}
\end{proposition}

Here is the relation between the area measure of a $\C$-convex domain and the Monge--Amp\`ere measure of its support function. 

\begin{proposition}
\label{prop A MA C-convex}
In the parametrisation $(\mathbf{P})$, for every $\C$-convex domain $K$ with support function $s$, we have 
\begin{equation*}
\Area(K) = (-\omega_\C) \MA(s) \, . 
\end{equation*}
\end{proposition}
\begin{proof}
In the $C^2_+$-case, that is a direct consequence from Proposition~\ref{prop aff invariants relative aff sphere}, Corollary~\ref{cor link A(K) vol} and Proposition~\ref{prop MA=det Hess}. In the general case, consider $(s_n)_{n\in\N}$ a sequence of $C^2_+$ support function converging locally uniformly to $s$ given by Lemma~\ref{lem approx}. On one side, by Proposition~\ref{prop vague convergence of Si}, we have that $A(K_n)$ converges vaguely to $A(K)$. On the other side, it is known \cite[Proposition~2.6]{Figalli_17} that the local uniform convergence of $(s_n)_{n\in \N}$ towards $s$ implies the vague convergence of $(\MA(s_n))_{n\in\N}$ towards $\MA(s)$, giving the wanted equality at the vague limit. 
\end{proof}

\section{Invariant domains}
\label{sec invariant convex domains}

In this section we shall consider the case where the convex cone $\C$ is \emph{divisible} in the sense of Benoist \cite{Benoist_04}, and we will focus on $\C$-convex domains satisfying some kind of affine invariance related to this divisibility property.

\begin{definition}[Divisible convex cone]
A proper open convex cone $\C \subset\V^{d+1}$ is said to be divisible by a subgroup $\Gamma < \SL(\V^{d+1})$ if the projective action of $\Gamma$ on $\P(\C)$ is free, properly discontinuous and cocompact.
\end{definition}

\begin{proposition}[{\cite[Lemma~2.8]{Benoist_04}}]
A proper open convex cone $\C \subset\V^{d+1}$ is divisible by a subgroup $\Gamma < \SL(\V^{d+1})$ if and only its polar cone $\C^*$ is divisible by $\Gamma$ (acting with the dual action $\star$ described in Section~\ref{sec (P)}).
\end{proposition}

\begin{definition}[Affine deformation of a subgroup $\Gamma < \SL(\V^{d+1})$]
An \emph{affine deformation of a subgroup} $\Gamma < \SL(\V^{d+1})$ is a subgroup $\Gamma_\tau < \SA(\A^{d+1})$ such that the linear action of $\Gamma_\tau$ on $\V^{d+1}$, the vector space of translation in $\A^{d+1}$, is exactly the one of $\Gamma$, i.e. there is a group isomorphism
\begin{equation*}
 \Phi_\tau : \Gamma \longrightarrow \Gamma_\tau
\end{equation*}
such that for all $\gamma \in \Gamma$, $\gamma_\tau \coloneqq \Phi_\tau(\gamma) \in \SA(\A^{d+1})$ satisfies 
\begin{equation*}
 \dif \gamma_\tau= \gamma \in \SL(\V^{d+1})\, .
\end{equation*}
\end{definition}

In the parametrisation $(\mathbf{P})$ described in Section~\ref{sec (P)}, an affine deformation of a subgroup $\Gamma < \SL(\R^{d+1})$ is a subgroup
\begin{equation}
\label{eq affine deformation}
 \Gamma_\tau = \left\lbrace\bp{\gamma,\tau(\gamma)}\st \gamma\in \Gamma\right\rbrace < \SL (\R^{d+1}) \ltimes \R^{d+1}
\end{equation}
obtained by adding translation parts to elements of $\Gamma$ through a translation function $\tau: \Gamma \to \R^{d+1}$. As $\Gamma_\tau$ is a group, the translation function $\tau$ must satisfy the following \emph{cocycle relation}: for all $\alpha,\beta \in \Gamma$
\begin{equation*}
\label{eq cocycle condition}
\tau({\alpha \beta}) = \alpha \tau(\beta) + \tau(\alpha) \, . 
\end{equation*}

From now on, if we assume that $\Gamma_\tau < \SA(\A^{d+1})$ is an affine deformation of a subgroup $\Gamma< \SL(\V^{d+1})$, we shall use the following notation convention:
\begin{itemize}
 \item for all $\gamma \in \Gamma$, we will denote by $\gamma_\tau$ the unique element of $\Gamma_\tau$ having linear part $\gamma$,
 \item in the parametrisation $(\mathbf{P})$, we will denote by $\tau $ the translation function $\tau: \Gamma \to \R^{d+1}$ defining $\Gamma_\tau$ through \eqref{eq affine deformation}.
\end{itemize}
\subsection{\texorpdfstring{$\tau$-convex domains}{tau-convex domains}}

Throughout the rest of this section, we assume that $\C$ is divisible by a subgroup $\Gamma < \SL(\V^{d+1})$ and let $\Gamma_\tau < \SA(\A^{d+1})$ be an affine deformation of $\Gamma$.

\begin{definition}[$\tau$-convex domain]
\label{def tau-convex domain and tau-support functions}
A \emph{$\tau$-convex domain} of $\A^{d+1}$ is a $\C$-convex domain $K$ invariant under the affine action of $\Gamma_\tau < \SA(\A^{d+1})$.

In the parametrisation $(\mathbf{P})$, a \emph{$\tau$-support function} is the $\Omega^*$-support function of a $\tau$-convex domain. 
\end{definition}

\begin{proposition}[{\cite[Corollary~5.7]{Ablondi_25}}]
\label{prop equiv supp function}
In the parametrisation $(\mathbf{P})$, a $\C$-convex domain $K$ is $\tau$-convex if and only if its total support function $\tilde{s}$ is $\tau$-equivariant, in the sense that for all $Y \in \C^*$ and $\gamma \in \Gamma$, it satisfies 
\begin{equation*}
\label{eq action on C* support function}
\tilde{s}(Y) = \tilde{s}(\gamma^{-1} \star Y) + \tau(\gamma) \cdot Y \, ,
\end{equation*}
where $\star$ denotes the dual action of $\Gamma$ (see Section~\ref{sec (P)}), and $\tau $ is translation function $\tau: \Gamma \to \R^{d+1}$ defining $\Gamma_\tau$ in the parametrisation $(\mathbf{P})$.

Equivalently, a $\C$-convex domain $K$ is $\tau$-convex if and only if its support function $s$ is $\tau$-equivariant, in the sense that for all $y \in \Omega^*$ and $\gamma \in \Gamma$, it satisfies
\begin{equation}
\label{eq action on Omega* support function}
s(y) = \frac{\omega_\C(y)}{\omega_\C(\gamma^{-1}*y)}s(\gamma^{-1} * y)+ \tau(\gamma) \cdot (y,-1) \, ,
\end{equation}
where $\omega_\C$ is the support function of the Cheng--Yau affine sphere $\S$, and $*$ denotes the dual projective action of $\Gamma$ on $\Omega^*$ (see Section~\ref{sec (P)}) and $\tau $ is the translation function $\tau: \Gamma \to \R^{d+1}$ defining $\Gamma_\tau$ in the parametrisation $(\mathbf{P})$.
\end{proposition}

When $\tau =0$, using equation~\eqref{eq action on Omega* support function}, one directly gets the following characterisation of equivariance for functions on $\Omega^*$.

\begin{proposition}
\label{prop equiv and inv for function on Omega*}
A function $f : \Omega^* \to \R$ is $\Gamma$-equivariant, that is $\tau$-equivariant with $\tau =0$, if and only if the function $\bar{f} = f/\omega_\C$ on $\Omega^*$ is $\Gamma$-invariant.
\end{proposition}

The following is a direct consequence from the identity~\eqref{eq action and dual action}.

\begin{proposition}
\label{prop equivariance of the Gauss map}
Let $K$ be a $\tau$-convex domain of $\A^{d+1}$. Its Gauss map is an equivariant set-valued map with respect to the affine action of $\Gamma_\tau$ on $\partial_{sp} K$ and the projective dual action of $\Gamma$ on $\P(\C)^*$. That is, for all $P \in \partial_{sp} K$ and $\gamma_\tau \in \Gamma_\tau$, one has the set equality
\begin{equation*}
\G_K(\gamma_\tau \cdot P) = \gamma * \G_K(X) \coloneqq \left\lbrace \gamma * [\varphi] \st [\varphi] \in \G_K(X)\right\rbrace \, . 
\end{equation*}
Similarly, as the affine sphere is $\Gamma$-invariant, the $\C$-normal vector field is $\Gamma$-equivariant: for all $\varphi \in \P(\C^*)$ and $\gamma \in \Gamma$, 
\begin{equation*}
N_{\S}(\gamma * [\varphi]) = \gamma \cdot N_{\S}[\varphi]\, . 
\end{equation*}
\end{proposition}

From \cite{Nie_Seppi_22}, one can deduce the following fact.

\begin{proposition}[{\cite[Proposition~5.12]{Ablondi_25}}]
\label{prop existence of boundary function gtau}
There is a unique continuous function on the boundary $\partial \Omega^*$, denoted by $g_\tau$, such that every $\tau$-equivariant continuous function on $\Omega^*$ continuously extends on $\overline{\Omega^*}$ to a function coinciding with $g_\tau$ on $\partial \Omega^*$.
\end{proposition}

Introducing $s_\tau$ the \emph{convex envelope} of the function $g_\tau$ given by Proposition~\ref{prop existence of boundary function gtau}, which can be expressed in the following way,
\begin{equation*}
s_\tau (y) = \mathrm{Env}(g_\tau)(y) \coloneqq \sup \left\lbrace a(y) \st a : \overline{\Omega^*} \to \R \ \text{is affine and} \ a\leq g_\tau \ \text{on} \ \partial \Omega^*\right\rbrace \, ,
\end{equation*}
and considering the $\C$-convex domain $D_\tau $ with support function $s_\tau$, we get the following result, proved independently by Choi \cite{Choi_25} and Nie and Seppi \cite{Nie_Seppi_23}. 

\begin{proposition}
\label{prop Dtau}
There is a \emph{maximal $\tau$-convex domain} $D_\tau \subset \A^{d+1}$ for inclusion, in the sense that every $\tau$-convex domain is included in $D_\tau$.
\end{proposition}

In \cite{Ablondi_25}, we have proved the following.

\begin{proposition}[{\cite[Theorem 7.8]{Ablondi_25}}]
\label{prop action Dtau free prop disc}
The action of $\Gamma_\tau$ on the maximal $\tau$-convex domain $D_\tau$ is free and properly discontinuous, and there is a homeomorphism
\begin{equation*}
D_\tau /\Gamma_\tau \simeq \P(\C)/\Gamma \times \R \, .
\end{equation*}  
\end{proposition}

\begin{proposition}
The maximal $\tau$-convex domain $D_\tau$ is the unique $\tau$-convex domain with area measure 
\begin{equation*}
 \Area(D_\tau) = 0 \, . 
\end{equation*}
\end{proposition}
\begin{proof}
Let us use the parametrisation $(\mathbf{P})$. Proposition~\ref{prop A MA C-convex} implies that a $\tau$-convex domain has zero area measure if and only if its $\tau$-support function $s$ satisfies $\MA(s) =0 $. By Proposition~\ref{prop existence of boundary function gtau} and  \cite[1.5.2]{Gutierrez_01}, that is the case if and only if $s = \mathrm{Env}(g_\tau)=s_\tau$. 
\end{proof}

\begin{proposition}
Let $K$ be a $\tau$-convex domain. Then, the Radon measure $V_\varepsilon(K)$ on $\P(\C^*)$ is invariant under the dual projective action of $\Gamma$. 
\end{proposition}
\begin{proof}
Let $b$ be Borel subset of $\P(\C^*)$ and $\gamma_\tau\in \Gamma_\tau$. Then, using the equivariance of the Gauss maps,
\begin{align*}
\mathcal{N}_\varepsilon(K)(\gamma*b) & = \left\lbrace P + t \, N_{\S} [\varphi]\st P \in \partial_{sp} K, \ [\varphi]\in \G_K(P)\cap (\gamma*b), \ t \in (0,\varepsilon) \right\rbrace \\
 & = \left\lbrace \gamma_\tau \cdot P + t \, \gamma N_{\S} [\varphi]\st P \in \partial_{sp} K, \ [\varphi]\in \G_K(P)\cap b, \ t \in (0,\varepsilon) \right\rbrace \\
 & = \left\lbrace \gamma_\tau \cdot( P + t N_{\S} [\varphi])\st P \in \partial_{sp} K, \ [\varphi]\in \G_K(P)\cap b, \ t \in (0,\varepsilon) \right\rbrace \\
 & = \gamma_\tau \cdot \mathcal{N}_\varepsilon(K)(b)\, . 
\end{align*}
So that, as $\gamma_\tau$ preserves the volume,
\begin{equation*}
V_\varepsilon(K)(\gamma*b) = \vol\bp{\mathcal{N}_\varepsilon(K)(\gamma*b)} = \vol \bp{\mathcal{N}_\varepsilon(K)(b)} = V_\varepsilon(K)(b) \, . \qedhere
\end{equation*}
\end{proof}

Hence, given a $\tau$-convex domain $K$ and $\varepsilon>0$, the Radon measure $V_\varepsilon(K)$ descends to a Radon measure on 
the quotient $\P(\C^*)/\Gamma$ which is naturally endowed with a \emph{convex compact projective manifold structure}. 
We shall denote by $V_\varepsilon(K/\Gamma_\tau)$ that Radon measure induced on $\P(\C^*)/\Gamma$. The locally finite nature of the Radon measure $V_\varepsilon(K/\Gamma_\tau)$ and the compactness of $\P(\C^*)/\Gamma$ imply that $V_\varepsilon(K/\Gamma_\tau)$ is a finite measure. 

\begin{remark} 
\label{rem induced volume} 
Let $K$ be an any $\tau$-convex domain $K$. As $K$ is included in $D_\tau$, by Proposition~\ref{prop action Dtau free prop disc}, the quotient $K/\Gamma_\tau$ is a manifold. As elements of $\SA(\A^{d+1})$ preserve the volume form $\dvol$ on $\A^{d+1}$, it is naturally endowed with a volume form induced by $\dvol$, the ambient one on $\A^{d+1}$, that we still denote by $\dvol$ (and by $\dif \mathcal{L}$ in the parametrisation $(\mathbf{P})$). Then, if $K$ is a $\tau$-convex domain, the measure $V_\varepsilon(K/\Gamma_\tau)$ on $\P(\C^*)/\Gamma$ satisfies that for all Borel subset $b \subseteq \P(\C^*)/\Gamma$
\begin{equation*}
 V_\varepsilon(K/\Gamma_\tau)(b) = \vol\bp{\mathcal{N}_\varepsilon(K/\Gamma_\tau)(b)} \, ,
\end{equation*}
where $\mathcal{N}_\varepsilon(K/\Gamma_\tau)(b)$ is defined in the following way: let $\tilde b \subseteq \P(\C^*)$ be the preimage of $b$ under the action of $\Gamma$,
\begin{equation*}
 \mathcal{N}_\varepsilon(K/\Gamma_\tau)(b) \coloneqq \mathcal{N}_\varepsilon(K)(\tilde b)/\Gamma_\tau \subset K/\Gamma_\tau \, . 
\end{equation*}
That fact justifies our choice of notation. It can be seen by introducing a fundamental domain of $\P(\C^*)$ for the action of $\Gamma$. 
\end{remark}

Theorem~\ref{theo Steiner} and Corollary~\ref{cor S0} descend to the quotient and give the following fact. 

\begin{theorem}[Local Steiner Formula for $\tau$-convex domains]
\label{theo Steiner tau-convex}
Let $K$ be a $\tau$-convex domain. Then, there exists a unique family $S_0(K/\Gamma_\tau)$, $\dots$, $S_d(K/\Gamma_\tau)$ of finite Radon measures on $\P(\C^*)/\Gamma$ such that for all $\varepsilon>0$,
\begin{equation*}
V_\varepsilon(K/\Gamma_\tau) = \frac{1}{d+1}\sum_{i=0}^{d}\varepsilon^{d+1-i} \binom{d+1}{i} S_i(K/\Gamma_\tau) \, . 
\end{equation*}
Moreover, $S_0(K/\Gamma_\tau)$ is the volume form induced by $\vol_{\S}$. 
\end{theorem}

\subsection{Hausdorff distance and cosmological time}

\begin{proposition}
Let $K$ and $K'$ be two $\tau$-convex domains of $\A^{d+1}$. Then there exist $\varepsilon >0$ such that
\begin{equation*}
K+ \varepsilon \S \subseteq K' \quad \text{and} \quad K'+ \varepsilon \S \subseteq K \, .
\end{equation*}
\end{proposition}
\begin{proof}
Let us use the parametrisation $(\mathbf{P})$, and let $s$ and $s'$ be the support function of respectively $K$ and $K'$. Then, using the epigraph characterisation of $\C$-convex domain (Proposition~\ref{prop char C-convex}), the two wanted inclusions are equivalent to 
\begin{equation*}
(s+\varepsilon\omega_\C)^* \geq (s')^* \quad \text{and} \quad (s'+\varepsilon\omega_\C)^* \geq s^* \, .
\end{equation*}
As the Legendre--Fenchel transform is order reversing, that is equivalent to
\begin{equation*}
s+\varepsilon\omega_\C^* \leq s' \quad \text{and} \quad s'+\varepsilon\omega_\C \geq s \, ,
\end{equation*}
which is just 
\begin{equation*}
\varepsilon\omega_\C^* \leq s'-s \leq -\varepsilon\omega_\C \, .
\end{equation*}
As $s$ and $s'$ are two $\tau$-support functions, then, because of Proposition~\ref{prop equiv supp function}, the function $(s-s')/\omega_\C$ on $\Omega^*$ is invariant under the dual projective action of $\Gamma$. Hence, it descends to a continuous function on the compact manifold $\Omega^*/\Gamma$. Setting $\varepsilon = \Vert (s-s')/\omega_\C\Vert_\infty$, we thus get the wanted result.
\end{proof}

The last proposition (and its proof) motives to define the following distance on $\tau$-convex domains. 

\begin{definition}[Hausdorff distance on $\tau$-convex domains]
The \emph{Hausdorff distance} between two $\tau$-convex domains $K$ and $K'$, with respective support functions $s$ and $s'$, is defined as 
\begin{equation*}
d_H(K,K') = \inf \left\lbrace \varepsilon >0 \st K+ \varepsilon \S \subseteq K' \ \text{and} \ K'+ \varepsilon \S \subseteq K\right\rbrace. 
\end{equation*}
In the parametrisation $(\mathbf{P})$, using the support function $s$ and $s'$ of respectively $K$ and $K'$, we have
\begin{equation*}
d_H(K,K') = \left\Vert \frac{s-s'}{\omega_\C}\right\Vert_\infty \, . 
\end{equation*}
\end{definition}

\begin{proposition}
\label{prop weak cv and Hausdorff}
Let $(K_n)_{n\in\N}$ be a sequence of $\tau$-convex domains converging to a $\tau$-convex domain $K$ for the Hausdorff distance $d_H$. Then, the sequence of Radon measures $(\Area(K_n/\Gamma_\tau))_{n\in\N}$ on $\P(\C)^*/\Gamma$ converges weakly to $\Area(K/\Gamma_\tau)$. That is, for any $f \in C^0_b(\P(\C)^*/\Gamma)$ a bounded continuous function on $\P(\C)^*/\Gamma$, 
\begin{equation*}
\int_{\P(\C)^*/\Gamma} f \, \dif \Area(K_n/\Gamma_\tau) \xrightarrow[n\to +\infty]{} \int_{\P(\C)^*/\Gamma} f \, \dif \Area(K/\Gamma_\tau) \, . 
\end{equation*}
\end{proposition}
\begin{proof}
Let us work in the parametrisation $(\mathbf{P})$ and introduce the sequence $(s_n)_{n\in \N}$ of support functions of the domains $(K_n)_{n\in\N}$. The Hausdorff convergence of $(K_n)_{n\in\N}$ directly implies the local uniform convergence of $(s_n)_{n\in \N}$ towards $s$, the support function of $K$. Thus, by Proposition~\ref{prop vague convergence of Si}, $(S_d(K_n))_{n\in\N} = (\Area(K_n))_{n\in\N}$ converges vaguely towards $S_d(K)=\Area(K)$. By cocompactness of the action of $\Gamma$ on $\Omega^*$, that implies that the quotient measures $(\Area(K_n/\Gamma_\tau))_{n\in\N}$ on $\Omega^*/\Gamma$ converges weakly to $\Area(K/\Gamma_\tau)$. 
\end{proof}

In \cite[Section 6]{Ablondi_25}, we have shown that the maximal domain $D_\tau$ is naturally endowed with a $C^1$ $\Gamma_\tau$-invariant \emph{cosmological time} $\mathcal{T}_\tau:D_\tau\to\R^*_+$. Its level sets
\begin{equation*}
 S^t_\tau \coloneqq \left\lbrace P \in D_\tau \st \mathcal{T}_\tau(X)=t \right\rbrace
\end{equation*}
are the $C^1$ and $\C$-spacelike embedded $\Gamma_\tau$-invariant hypersurfaces
\begin{equation*}
 S^t_\tau = \partial(D_\tau + t \, \S)
\end{equation*}
foliating $D_\tau$. 

The Hausdorff distance and cosmological time are related in the following way. Given a $\tau$-convex domain $K$, let us introduce 
\begin{equation*}
 \mathcal{T}_\tau^{\max} (K) \coloneqq \sup_{X\in \partial K} \mathcal{T}_\tau(X) \quad \text{and} \quad \mathcal{T}_\tau^{\min} (K) \coloneqq \inf_{X\in \partial K} \mathcal{T}_\tau(X) \, . 
\end{equation*}
Using the parametrisation $(\mathbf{P})$, the definition of cosmological time and the cocompactness of the action of $\Gamma$ on $\Omega^*$, one sees that those supremum and infimum are attained and that, denoting by $s$ the support function of $K$,
\begin{equation*}
 \mathcal{T}_\tau^{\max} (K) = d_H(K,D_\tau) = \max_{y \in \Omega^*} \frac{s(y)-s_\tau(y)}{\omega_\C(y)} \quad \text{and} \quad \mathcal{T}_\tau^{\min} (K) = \min_{y \in \Omega^*} \frac{s(y)-s_\tau(y)}{\omega_\C(y)} \, . 
\end{equation*}

\begin{lemma}[{\cite[Lemma~3.17]{Bonsante_Fillastre_17}}]
\label{lem exctraction of tau-convex domains}
Let $t>0$ and $(K_n)_{n\in\N}$ be a sequence of $\tau$-convex domains all intersecting $S_\tau^t$, the $t$-level hypersurface of the cosmological time $\mathcal{T}_\tau$ on $D_\tau$. Then, there exists a subsequence $(K_{\varphi(n)})_{n\in\N}$ converging for the Hausdorff distance $d_H$. 
\end{lemma}
\begin{proof}
The proof of \cite{Bonsante_Fillastre_17}, relying on the Ascoli--Arzel\`a Theorem, directly translates to our setting once one notices that in our parametrisation $(\mathbf{P})$, the hypersurfaces $\partial K_n$ are now all graphs of $k$-Lipschitz functions where $k = \sup_{y \in \Omega^*} \Vert y \Vert$ \cite[Corollary~13.3.3]{Rockafellar_97}. 
\end{proof}

\begin{corollary}
\label{cor min cosmo time}
Let $(K_n)_{n\in\N}$ be a sequence of $\tau$-convex domains of $\A^{d+1}$. If $d_H(K_n,D_\tau)\xrightarrow[n\to+\infty]{}+\infty$, then $\mathcal{T}_\tau^{\min} (K_n) \xrightarrow[n\to+\infty]{}+\infty$. 
\end{corollary}

\begin{lemma}[{\cite[Lemma~3.20]{Bonsante_Fillastre_17}}]
\label{lem ration minmax cosmo time}
There exists a constant $\delta=\delta(\tau)$ such that for any $\tau$-convex domain $K$, if $\mathcal{T}_\tau^{\min} (K) \geq 1$, then $\mathcal{T}_\tau^{\max} (K)/\mathcal{T}_\tau^{\min} (K) < \delta$
\end{lemma}
\begin{proof}
The proof of \cite{Bonsante_Fillastre_17} again relies on the Ascoli--Arzel\`a Theorem and thus directly translates to our setting. 
\end{proof}

\section{The covolume}
\label{sec covolume}

In this section, we introduce and study what will be the central tool in the proof of Theorem~\ref{theo intro Minkowski solution}: the \emph{covolume functional} on $\tau$-convex domains. It will play the same role for $\tau$-convex domains as volume for bounded convex domain of the affine space. Note that, by the \emph{Brunn--Minkowski Inequality} \cite[Theorem~7.1.1]{Schneider_93}, the volume is a $\log$-concave functional on bounded convex domain of the affine space, while the covolume will be a strictly convex functional on $\tau$-convex domains (see Theorem~\ref{theo convexity covolume}).

Throughout this whole section, we shall again assume that $\C$ is divisible by a subgroup $\Gamma < \SL(\V^{d+1})$ and let $\Gamma_\tau < \SA(\A^{d+1})$ be an affine deformation of $\Gamma$.

Before introducing the covolume functional, let us recall Remark~\ref{rem induced volume}: given any $\tau$-convex domain $K$, the quotient $K/\Gamma_\tau$ is a manifold naturally endowed with a volume form induced by the form $\dvol$ on the affine space $\A^{d+1}$. 

\begin{definition}[Covolume]
The \emph{covolume} of a $\tau$-convex domain $K$, denoted by $\covol(K)$, is the volume of the complementary of $K/\Gamma_\tau$ inside the manifold $D_\tau/\Gamma_\tau$. 
\end{definition}

The proof of the continuity of the covolume by Fillastre \cite[Lemma~5.1]{Fillastre_13} in the case of affine deformation of uniform lattices of $\SO_0(d,1)$, adapted from the Euclidean case of the volume \cite[Theorem~1.8.20]{Schneider_93} also applies in the case of an affine deformation of a subgroup of $\SL(\V^{d+1})$ dividing a cone. 

\begin{proposition}
\label{prop covol is continuous}
The covolume is continuous on the set of $\tau$-convex domains endowed with the Hausdorff distance $d_H$. 
\end{proposition}

\subsection{Strict convexity of the covolume}

In the case of affine deformations of uniform lattices of $\SO_0(d,1)$, the convexity of the covolume on $\tau$-convex domains is due to Bonsante and Fillastre \cite[Theorem~1.2]{Bonsante_Fillastre_17}. In the general case we can reproduce their proof using \emph{Dirichlet--Lee domains}, which are convex fundamental domains of $D_\tau$ for the action of $\Gamma_\tau$ (see Appendix~\ref{app DL domain}). Thanks to the fact that those fundamental domains are slightly different from the one they have used, we get a slightly stronger result: we have the strict convexity of the covolume functional on the set of all $\tau$-convex domains, and not only on those with boundary included in $D_\tau$ \cite[Proposition~3.33]{Bonsante_Fillastre_17}. 

\begin{theorem} 
\label{theo convexity covolume}
The covolume is strictly convex on $\tau$-convex domains, i.e. if $K_0$ and $K_1$ are two distinct $\tau$-convex domains then for all $0 < t < 1$,
\begin{equation*}
 \covol\bp{(1-t)K_0 + t K_1} < (1-t) \, \covol(K_0)+ t \, \covol(K_1) \, . 
\end{equation*}
\end{theorem}
\begin{proof}
For all $0 < t < 1$, set $K_t \coloneqq (1-t)K_0 + t K_1$ (see Remark~\ref{rem convex combination}). Let $[\varphi] \in \P(\C^*)$, $\vec{n} \coloneqq N_{\S}[\varphi] \in \S$, and $\DL_{[\varphi]}$ be the Dirichlet--Lee domain of $D_\tau$ given by $[\varphi]$ (see Definition~\ref{def DL domain}). Let $H$ be a $\C$-spacelike affine hyperplane directed by $[\varphi]$ and such that for all $t \in [0,1]$, $\partial K_t \cap \DL_{[\varphi]}$ is included in the half-space $H - \C$, (see Figure~\ref{fig covol}).

For $i\in\{0,1\}$, let us denote $C_i \coloneqq K_i \cap (H -\C) \cap \DL_{[\varphi]}$. Because of our choice of $H$, we have $\overline{C_i} \cap H = H \cap \DL_{[\varphi]} $. Because of Lemma~\ref{lem for covolume cvx proof}, if $P \in \DL_{[\varphi]}$ and there is an $\varepsilon>0$ such that $ X - \varepsilon \vec{n} \in \DL_{[\varphi]}$, then $P \in \mathrm{int}(\DL_{[\varphi]})$. Thus, for all $P \in H\cup \partial \DL_{[\varphi]}$, there is no $\varepsilon>0$ such that $X - \varepsilon \vec{n} \in \DL_{[\varphi]}$. Hence, as $\DL_{[\varphi]}$ is convex, $C_0$ and $C_1$ are convex caps in the sense of \cite[subsection 50]{Bonnesen_Fenchel_87}. 

For all $0 < t < 1$, set $C_t \coloneqq K_t \cap (H-\C) \cap \DL_{[\varphi]}$. As $K_t = (1-t)K_0 + t K_1$, we have $C_t = (1-t)C_0 + t C_1$. By strict concavity of the volume on convex caps \cite[subsection 50]{Bonnesen_Fenchel_87}, for all $0 < t < 1$,
\begin{equation*}
 \mathcal{L}(C_t) > (1-t)\,\vol(C_0)+t\,\vol(C_1) \, . 
\end{equation*}
As for all $t\in[0,1]$, $\covol(K_t) = \vol(\DL_{[\varphi]} \cap (H-\C)) - \vol(C_t)$, we deduce the strict convexity of the covolume functional on $\tau$-convex domains. 
\end{proof}

\begin{figure}[ht]
 \centering
 \includegraphics{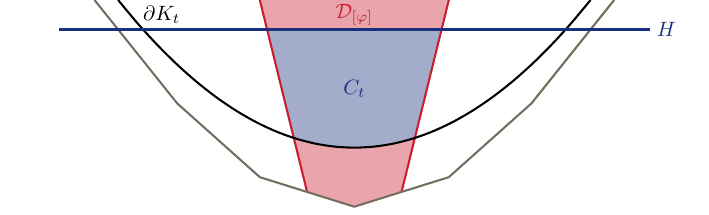}
 \caption{Situation in the proof of Theorem~\ref{theo convexity covolume}. }
 \label{fig covol}
\end{figure}

\begin{remark}
One could consider the covolume on all $\tau$-convex domains with the cocycle $\tau$ varying, but that functional is not convex \cite[Remark 3.35]{Bonsante_Fillastre_17}.
\end{remark}

From the convexity of the covolume, we deduce the following fact. 

\begin{lemma}[{\cite[Lemma~3.36]{Bonsante_Fillastre_17}}]
\label{lem ineq covol}
In the parametrisation $(\mathbf{P})$, let $K_0$ and $K_1$ be two $\tau$-convex domains with respective support functions $s_0$ and $s_1$, and such that $K_1 \subseteq K_0$ (or equivalently such that $s_1 \leq s_0$). Then, 
 \begin{equation*}
 \int_{\Omega^*/\Gamma} \frac{s_1-s_0}{\omega_\C} \dif \Area(K_0/\Gamma_\tau) \leq \covol(K_1)-\covol(K_0) \leq \int_{\Omega^*/\Gamma} \frac{s_1-s_0}{\omega_\C} \dif \Area(K_1/\Gamma_\tau) \, . 
 \end{equation*}
\end{lemma}
\begin{proof}
We shall proceed by approximation. First assume that $K_0$ and $K_1$ are two $C^2_+$ $\tau$-convex domains with support functions satisfying $s_1<s_0$. Consider the function 
 \begin{equation*}
 \function{\Psi}{[0,1]}{\R}{t}{\covol(K_t)-\covol(K_0)},
 \end{equation*}
where for all $0 \leq t \leq 1$,
 \begin{equation*}
 K_t \coloneqq (1-t) K_0 + t K_1. 
 \end{equation*}
Using the foliation given by the graphs of the functions $(s_t^*)_{t\in[0,1]}$, a computation similar to the one in the proof of Lemma~\ref{lem Jacobian computation} gives that for all $t \in [0,1]$
 \begin{equation}
 \label{eq expression f}
 \Psi(t)=\covol(K_t)-\covol(K_0) = \int_0^t \int_{\Omega^*/\Gamma} \frac{s_1-s_0}{\omega_\C} \dif \Area(K_{l}/\Gamma_\tau) \, \dif l \, . 
 \end{equation}
By convexity of the covolume (Theorem~\ref{theo convexity covolume}), $\Psi$ is convex. Thus, we have
 \begin{equation*}
 \Psi'(0) \leq \Psi(1)-\Psi(0)\leq \Psi(1) \, ,
 \end{equation*}
which becomes, using the expression of $\Psi$ given by \eqref{eq expression f},
 \begin{equation*}
 \int_{\Omega^*/\Gamma} \frac{s_1-s_0}{\omega_\C} \dif \Area(K_0/\Gamma_\tau) \leq \covol(K_1)-\covol(K_0) \leq \int_{\Omega^*/\Gamma} \frac{s_1-s_0}{\omega_\C} \dif \Area(K_1/\Gamma_\tau) \, . 
 \end{equation*}
For the general case, using the approximation lemma given in Appendix~\ref{app C2+-approximation}, Proposition~\ref{prop covol is continuous}, and Proposition~\ref{prop weak cv and Hausdorff}, one shows that this last inequality extends to every pair of $\tau$-convex domains $K_0$ and $K_1$ such that $K_1 \subseteq K_0$ (or equivalently such that $s_1 \leq s_0$). 
\end{proof}

\subsection{Area measures as derivatives of the covolume}

In this subsection we shall only work in the parametrisation $(\mathbf{P})$ described in Section~\ref{sec (P)}. Our goal is to compute the ‘‘derivative'' of the covolume. Hence, we first need to consider small deformations of $\C$-convex domains, or equivalently of support functions on $\Omega^*$ in ‘‘every direction''.

At first it can seem difficult. Indeed, let $s$ be a $\tau$-support function, let $f$ be any continuous $\Gamma$-equivariant function on $\Omega^*$, and let $t>0$. Then, the function $s+tf$ is $\tau$-equivariant but has no reason to be convex, even for $t$ small enough (for example when $f\geq 0$ and $s = s_\tau$ is the support function of the maximal $\tau$-convex domain $D_\tau$). 

Nonetheless, there is a way to get rid of that problem. Remember that any continuous $\tau$-equivariant function $g: \Omega^* \to \R$ (not necessarily convex) continuously extends to a function on $\overline{\Omega^*}$ with boundary values given by $g_\tau$ (Proposition~\ref{prop existence of boundary function gtau}). Given such a function $g$, let us introduce
 \begin{equation*}
 K_g \coloneqq \left\lbrace X \in \R^{d+1} \st X\cdot(y,-1)<g(y), \ \forall y \in \overline{\Omega^*}\right\rbrace \, . 
 \end{equation*}
The domains defined in that way satisfy the following key property.

\begin{proposition}
\label{prop Kg and g**}
Let $g: \Omega^* \to \R$ be a continuous $\tau$-equivariant function. Then, $K_g$ is a $\tau$-convex domain with support function $g^{**} = (g^*)^*$, the \emph{Legendre--Fenchel biconjugate} of $g$ which also is its \emph{convex hull}, that is the pointwise largest convex function bounded above by $g$. 
\end{proposition}
\begin{proof}
The fact that the biconjugate of a function is its convex hull is a standard convex analysis result. 

Note that as $g$ continuously extends to $\overline{\Omega^*}$, it is bounded and thus its convex hull $g^{**}$ is finite on the whole $\overline{\Omega^*}$. Moreover, the $\tau$-equivariance of $g$ directly implies that $K_g$ is $\Gamma_\tau$-invariant and that $g^{**}$ is $\tau$-equivariant. Using the definition of $K_g$, we have 
\begin{align*}
 K_g & = \left\lbrace X \in \R^{d+1} \st X\cdot(y,-1)<g(y), \ \forall y \in \overline{\Omega^*}\right\rbrace = \left\lbrace (x,\lambda) \in \R^{d+1} \st x\cdot y - g(y) < \lambda, \ \forall y \in \overline{\Omega^*}\right\rbrace \\
 & = \left\lbrace (x,\lambda) \in \R^{d+1} \st \lambda > \sup_{y \in \Omega^*}\bp{x\cdot y -g(y)}\right\rbrace = \left\lbrace (x,\lambda) \in \R^{d+1} \st \lambda > g^*(x)\right\rbrace \\
 & = \epi (g^*) \, . 
\end{align*}
Its support function $s$ is thus defined by: for all $y \in \Omega^*$,
\begin{equation*}
s(y) = \sup_{(x, \lambda) \in K_g} (x,\lambda) \cdot (y,-1) = \sup_{x \in \R^d } x\cdot y -g^*(x)
 = g^{**}(y) \, . \qedhere
\end{equation*}
\end{proof} 

Proposition~\ref{prop Kg and g**} solves our problem: let $s$ be a $\tau$-support function, $f$ is any continuous $\Gamma$-equivariant function on $\Omega^*$, and $t>0$, then $(s+tf)^{**}$ will be convex, i.e. will be a $\tau$-support function. The following lemma tells us that the path $(K_{s+tf})_{t \in \R}$ is a Hausdorff continuous path of $\tau$-convex domains.

\begin{lemma}
\label{lemma Hausdorff cv of Ks+tf}
Let $K$ be a $\tau$-convex domain with support function $s$, and $f$ be a continuous $\Gamma$-equivariant function on $\Omega^*$. Then, the family of $\tau$-convex domains $(K_{s+tf})_{t\in\R}$ is continuous for the Hausdorff distance. 
\end{lemma}
\begin{proof}
Up to setting $K_0' = K_{s+ tf}$, we just have to prove the continuity at $t = 0$. Moreover, up to replacing $f$ by $-f$ (which is also $\Gamma$-equivariant), we just have to consider the case when $t$ tends to $0^+$. Let $t>0$. As $s$ and $f^{**}$ are convex, so too is $s+tf^{**}$. Moreover, $f^{**}\leq f$, so $s+tf^{**}$ is convex and satisfies $s+tf^{**} \leq s+tf$. Hence
\begin{equation*}
 s+tf^{**} \leq (s+tf)^{**} \leq s+tf \, . 
\end{equation*}
Thus, as $\omega_\C \leq 0$, we get 
\begin{equation*}
 t\frac{f}{\omega_\C} \leq \frac{(s+tf)^{**}-s}{\omega_\C} \leq t\frac{f^{**}}{\omega_\C} \, . 
\end{equation*}
That concludes the proof, as then
\begin{equation*}
 d_H(K,K_{s+tf}) \leq t \max \vp{\left\Vert\frac{f}{\omega_\C} \right\Vert_\infty , \left\Vert\frac{f^{**}}{\omega_\C} \right\Vert_\infty} \, . \qedhere
\end{equation*}
\end{proof}

We shall thus prove that, just like in the Euclidean case \cite[Proposition~2]{Carlier_03}, area measures are derivatives of the covolume in the following sense. 

\begin{theorem}[Theorem~\ref{theo intro derivative covol}]
\label{theo area measure as gateau derivative}
Let $K=K_s$ be a $\tau$-convex domain with support function $s$, and $f$ be a continuous $\Gamma$-equivariant function on $\Omega^*$ (remember that, by Proposition~\ref{prop equiv and inv for function on Omega*}, $\bar f = f/\omega_\C$ is then $\Gamma$-invariant). Then, we have 
\begin{equation*}
\frac{\covol (K_{s+tf})-\covol(K)}{t} \xrightarrow[t \to 0]{} \int_{\Omega^*/\Gamma} \frac{f}{\omega_\C} \, \dif \Area(K/\Gamma_\tau) \, . 
 \end{equation*}
\end{theorem}

First, let us prove a useful convex analysis lemma concerning Monge--Amp\`ere measures of convex hulls. 

\begin{lemma}
\label{prop MA and cvx hull}
 Let $U \subseteq \R^{d+1}$ be a bounded convex domain, $g: \overline{U} \to \R$ be a continuous function and $g^{**}$ be its convex hull. Then, the Borel set
 \begin{equation*}
 b = \left\lbrace x \in U \st g^{**}(x) \neq g(x)\right\rbrace
 \end{equation*}
is negligible for the Monge--Amp\`ere measure of $g^{**}$, i.e. $\MA(g^{**})(b) = 0$. 
\end{lemma}
\begin{proof}
As $g^{**}$, the convex hull of $g$, satisfies $g^{**} \leq g$, $b$ is the set
\begin{equation*}
 b = \left\lbrace x \in U \st g^{**}(x)<g(x)\right\rbrace \, . 
 \end{equation*}
We claim that $\partial g^{**}(b)$ is included in the set of non-differentiable points of $g^*$, which will conclude the proof as, by Rademacher's Theorem, the convex function $g^*$ is differentiable almost everywhere, so that 
 \begin{equation*}
 \MA(g^{**})(b)=\mathcal{L}\bp{\partial g^{**}(b)}= 0\, . 
 \end{equation*}
 
Now, let us prove our claim. Let $x \in b$ and $y \in \partial g^{**}(x)$. By the Fenchel Inequality Theorem \cite[Theorem~23.5]{Rockafellar_97}, $x$ is a subgradient of $g^*$ at $y$. As
 \begin{equation*}
 g^*(y) = \sup_{x \in U} \bp{x \cdot y -g(x)} \, , 
 \end{equation*}
 there exist $x' \in \overline{U}$ such that 
 \begin{equation*}
 g^*(y) = x' \cdot y - g(x') \, . 
 \end{equation*}
 Thus, as $g^{**} \leq g$, we have
 \begin{equation}
 \label{ineq cvx hull}
 g^*(y) = x' \cdot y - g(x') \leq x' \cdot y - g^{**}(x) \, . 
 \end{equation}
 As $g^{**} = (g^*)^*$ is convex, by the Fenchel Inequality Theorem \cite[Theorem~23.5]{Rockafellar_97}, the equality must be achieved in \eqref{ineq cvx hull}, so that $g^{**}(x)=g(x)$, and $x'$ is a subgradient of $g^*$ at $y$. Thus, $x' \notin b$, and $g^*$ has two distinct subgradients $x \neq x'$ at $y$, i.e. $g^*$ is not differentiable at $y$. 
\end{proof}

Proposition~\ref{prop A MA C-convex} and Lemma~\ref{prop MA and cvx hull} directly imply the following. 

\begin{corollary}
\label{cor area measure convex hull}
Let $g: \Omega^* \to \R$ be a continuous $\tau$-equivariant function. Then, the following Borel set (which is well-defined by $\tau$-equivariance of $g$ and $g^{**}$)
 \begin{equation*}
 b = \left\lbrace [y] \in \Omega^*/\Gamma \st g^{**}(y) \neq g(y)\right\rbrace 
 \end{equation*}
is negligible for the area measure of $K_g/\Gamma_\tau$, i.e. $\Area(K_g/\Gamma_\tau)(b) = 0$. 
\end{corollary}

Now we have all the tools needed to prove Theorem~\ref{theo area measure as gateau derivative}.

\begin{proof}[Proof of Theorem~\ref{theo area measure as gateau derivative}]
Up to replacing $f$ by $-f$ (which is also $\Gamma$-equivariant), we just have to consider the case when $t$ tends to $0^+$. 

\noindent \textbullet\ First, let us start with the case where $f \geq 0$. Let $t>0$. As $s$ is convex and satisfies $s \leq s+tf$, we have 
\begin{equation*}
 s \leq (s+tf)^{**} \leq s+tf \, . 
\end{equation*}
Thus, as $\omega_\C \leq 0$, we deduce that 
\begin{equation}
\label{eq + ineq s+tf**}
0 \leq \frac{s - (s+tf)^{**}}{\omega_\C} \leq -t\frac{f}{\omega_\C} \, . 
\end{equation}
Now, let us apply apply Lemma~\ref{lem ineq covol} with $K_0=K_{s+tf}$ and $K_1= K$. We get
\begin{equation}
\label{eq + triple ineq}
 \int_{\Omega^*/\Gamma} \frac{s - (s+tf)^{**}}{\omega_\C} \,\dif \Area(K_{s+tf}/\Gamma_\tau) \leq \covol(K) - \covol(K_{s+tf}) \leq \int_{\Omega^*/\Gamma} \frac{s - (s+tf)^{**}}{\omega_\C} \,\dif \Area(K/\Gamma_\tau) \, . 
 \end{equation}
By Corollary~\ref{cor area measure convex hull}, the first term in \eqref{eq + triple ineq} is 
\begin{equation*}
\int_{\Omega^*/\Gamma} \frac{s - (s+tf)^{**}}{\omega_\C} \,\dif \Area(K_{s+tf}/\Gamma_\tau) = \int_{\Omega^*/\Gamma} \frac{s - (s+tf)}{\omega_\C} \,\dif \Area(K_{s+tf}/\Gamma_\tau) = -t \int_{\Omega^*/\Gamma} \frac{f}{\omega_\C} \,\dif \Area(K_{s+tf}/\Gamma_\tau) \, . 
\end{equation*}
Inequalities \eqref{eq + ineq s+tf**} imply that the last term in \eqref{eq + triple ineq} satisfies
\begin{equation*}
\int_{\Omega^*/\Gamma} \frac{s - (s+tf)^{**}}{\omega_\C} \,\dif \Area(K/\Gamma_\tau) \leq -t \int_{\Omega^*/\Gamma} \frac{f}{\omega_\C} \,\dif \Area(K/\Gamma_\tau) \, . 
\end{equation*}
Thus, we have 
\begin{equation*}
\int_{\Omega^*/\Gamma} \frac{f}{\omega_\C} \, \dif \Area(K/\Gamma_\tau) \leq \frac{\covol (K_{s+tf})-\covol(K)}{t} \leq \int_{\Omega^*/\Gamma} \frac{f}{\omega_\C} \, \dif \Area(K_{s+tf}/\Gamma_\tau) \, . 
\end{equation*}
By Lemma~\ref{lemma Hausdorff cv of Ks+tf}, $(K_{s+tf})_t$ Hausdorff converges to $K$ as $t$ tends to $0$. Hence, letting $t$ tend to $0$ in the previous inequalities, and using the weak convergence of area measures given by Proposition~\ref{prop weak cv and Hausdorff}, we get the wanted result. 

\noindent \textbullet\ Now, let us consider the case where $f \leq 0$. Let $t>0$. We now have 
\begin{equation*}
 (s+tf)^{**} \leq s+tf \leq s \, . 
\end{equation*}
Thus, as $\omega_\C \leq 0$, we deduce that 
\begin{equation}
\label{eq - ineq s+tf**}
 0 \leq t \frac{f}{\omega_\C} \leq \frac{(s+tf)^{**}-s}{\omega_\C} \, . 
\end{equation}
Now, let us apply apply Lemma~\ref{lem ineq covol} with $K_0=K$ and $K_1= K_{s+tf}$. We get
\begin{equation}
\label{eq - first triple ineq}
 \int_{\Omega^*/\Gamma} \frac{(s+tf)^{**} - s}{\omega_\C} \,\dif \Area(K/\Gamma_\tau) \leq \covol(K_{s+tf}) - \covol(K) \leq \int_{\Omega^*/\Gamma} \frac{(s+tf)^{**} - s}{\omega_\C} \,\dif \Area(K_{s+tf}/\Gamma_\tau) \, . 
 \end{equation}
By Corollary~\ref{cor area measure convex hull}, the last term in \eqref{eq - first triple ineq} is 
\begin{equation*}
\int_{\Omega^*/\Gamma} \frac{(s+tf)^{**} -s}{\omega_\C} \,\dif \Area(K_{s+tf}/\Gamma_\tau) =t \int_{\Omega^*/\Gamma} \frac{f}{\omega_\C} \,\dif \Area(K_{s+tf}/\Gamma_\tau) \, . 
\end{equation*}
Inequalities \eqref{eq - ineq s+tf**} imply that the first term in \eqref{eq - first triple ineq} satisfies
\begin{equation*}
\int_{\Omega^*/\Gamma} \frac{(s+tf)^{**}-s}{\omega_\C} \,\dif \Area(K/\Gamma_\tau) \geq t \int_{\Omega^*/\Gamma} \frac{f}{\omega_\C} \,\dif \Area(K/\Gamma_\tau) \, . 
\end{equation*}
Thus, we have 
\begin{equation*}
\int_{\Omega^*/\Gamma} \frac{f}{\omega_\C} \, \dif \Area(K/\Gamma_\tau) \leq \frac{\covol (K_{s+tf})-\covol(K)}{t} \leq \int_{\Omega^*/\Gamma} \frac{f}{\omega_\C} \, \dif \Area(K_{s+tf}/\Gamma_\tau) \, . 
\end{equation*}
By Lemma~\ref{lemma Hausdorff cv of Ks+tf}, $(K_{s+tf})_t$ Hausdorff converges to $K$ when $t$ tends to $0$. Again, letting $t$ tend to $0$ in the previous inequalities, and using the weak convergence of area measures given by Proposition~\ref{prop weak cv and Hausdorff}, we get the wanted result. 

\noindent \textbullet\ Finally, in the general case, decompose $f$ as 
\begin{equation*}
 f= f^++f^- \,
\end{equation*}
where $f^+ = (f + \vert f\vert)/2 \geq 0$ and $f^- = (f - \vert f\vert)/2 \leq 0$. For all $t>0$, we have
\begin{equation*}
 \frac{\covol (K_{s+tf})-\covol(K)}{t} = \frac{\covol (K_{s+tf^{+} +t f^{-}})-\covol(K_{s+tf^{-}})}{t} + \frac{\covol (K_{s+tf^{-}})-\covol(K)}{t} \, . 
\end{equation*}
Using what was done in the first point (the case $f\geq0$), we get that for all $t>0$,
\begin{equation*}
\int_{\Omega^*/\Gamma} \frac{f^+}{\omega_\C} \, \dif \Area(K_{s+tf^-}/\Gamma_\tau) \leq \frac{\covol (K_{s+tf^+ +tf^-})-\covol(K_{s+tf^-})}{t} \leq \int_{\Omega^*/\Gamma} \frac{f^+}{\omega_\C} \, \dif \Area(K_{s + tf^+ + tf^-}/\Gamma_\tau) \, . 
\end{equation*}
Letting $t$ tend to $0$ in the previous inequalities, and using Lemma~\ref{lemma Hausdorff cv of Ks+tf} and Proposition~\ref{prop weak cv and Hausdorff}, we get
\begin{equation*}
\frac{\covol (K_{s+tf^{+} +t f^{-}})-\covol (K_{s+tf^-})}{t}\xrightarrow[t \to 0^+]{} \int_{\Omega^*/\Gamma} \frac{f^+}{\omega_\C} \, \dif \Area(K/\Gamma_\tau) \, . 
\end{equation*}
By the second point (the case $f\leq0$), we have
\begin{equation*}
\frac{\covol (K_{s+tf^{-}})-\covol(K)}{t} \xrightarrow[t \to 0^+]{} \int_{\Omega^*/\Gamma} \frac{f^-}{\omega_\C}\, \dif \Area(K/\Gamma_\tau) \, . 
\end{equation*}
That concludes the proof as we hence have 
\begin{equation*}
\frac{\covol (K_{s+tf})-\covol(K)}{t}\xrightarrow[t \to 0^+]{} \int_{\Omega^*/\Gamma} \frac{f^+ + f^-}{\omega_\C} \, \dif \Area(K/\Gamma_\tau) \, . \qedhere
\end{equation*}
\end{proof}

\begin{remark}[On the ‘‘Gateaux derivative'' in Theorem~\ref{theo intro derivative covol}]
\label{rem Gateaux derivative}
Let $C^0_\tau(\Omega^*)$ denote the space of $\tau$-equivariant continuous function on $\Omega^*$. It has a structure of infinite-dimensional real affine space with associated vector space $C^0(\Omega^*/\Gamma)$, given by the transitive action
\begin{equation*}
 \namelessfunction{C^0(\Omega^*/\Gamma)\times C^0_\tau(\Omega^*)}{C^0_\tau(\Omega^*)}{(\bar{f},g)}{g +\omega_\C \bar{f}} \, . 
\end{equation*}
The metric on $C^0_\tau(\Omega^*)$ induced by the $L^\infty$-norm on the underlying vector space $C^0(\Omega^*/\Gamma)$ is the Hausdorff distance $d_H$. 
The set $\mathcal{K}_\tau(\Omega^*)$ of $\tau$-equivariant convex functions is a closed convex subset of $C^0_\tau(\Omega^*)$. Note that $\mathcal{K}_\tau(\Omega^*)$ can equivalently be seen as the set of all $\tau$-convex domains (Propositions~\ref{prop K is epigraph} and \ref{prop equiv supp function}). 
Lemma~\ref{lemma Hausdorff cv of Ks+tf} states that the map
\begin{equation*}
 \function{\mathrm{hull}}{C^0_\tau(\Omega^*)}{\mathcal{K}_\tau(\Omega^*)}{g}{g^{**}}
\end{equation*}
is continuous on $\mathcal{K}_\tau(\Omega^*)$. 
Theorem~\ref{theo area measure as gateau derivative} states that the extension of the covolume to $C^0_\tau(\Omega^*)$ given by 
\begin{equation*}
 \Psi \coloneqq \covol \circ \mathrm{hull} : C^0_\tau(\Omega^*) \longrightarrow \R 
\end{equation*}
admits a continuous linear \emph{Gateaux derivative} at every point of $\mathcal{K}_\tau(\Omega^*)$, expressed as 
\begin{equation*}
 \function{\dif^{\mathrm{Gat}} \Psi}{\mathcal{K}_\tau(\Omega^*) \times C^0(\Omega^*/\Gamma)}{\R}{(s,\bar f)}{\int_{\Omega^*/\Gamma} \bar f \,\dif \Area(K_s/\Gamma_\tau)} \, . 
\end{equation*}
\end{remark}

\section{\texorpdfstring{The Minkowski problem for $\tau$-convex domains}{The Minkowski problem for tau-convex domains}}
\label{sec inv Mink pb}

Throughout this whole section, we shall again assume that $\C$ is divisible by a subgroup $\Gamma < \SL(\V^{d+1})$ and let $\Gamma_\tau < \SA(\A^{d+1})$ be an affine deformation of $\Gamma$.

\subsection{Existence and uniqueness of a solution}

We shall prove that the Minkowski problem, asking for a $\tau$-convex domain with prescribed area measure, always has a unique solution. 

\begin{theorem}[{Theorem~\ref{theo intro Minkowski solution}}]
\label{theo Minkowski solution}
Let $\mu$ be a Radon measure on $\P(\C^*)/\Gamma$. Then, there exists a unique $\tau$-convex domain with area measure $\Area(K/\Gamma_\tau)=\mu$. 
\end{theorem}

The proof of the existence part of Theorem~\ref{theo Minkowski solution} is variational. To any Radon measure $\mu$ on $\P(\C^*)/\Gamma$, we will associate a functional $L_\mu$, involving the covolume, on the space of $\tau$-convex domains. We shall show that $L_\mu$ attains a minimum at a $\tau$-convex domain $K_\mu$. Then, using that the derivative of $L_\mu$ must vanish at $K_\mu$ and that we have computed the derivative of the covolume in Section~\ref{sec covolume}, we will get that $K_\mu$ satisfies a particular condition: it is a solution of the Minkowski problem $\Area(K/\Gamma_\tau)=\mu$.

The uniqueness part of Theorem~\ref{theo Minkowski solution} is then a direct consequence of the comparison principle \cite[Theorem~2.10]{Figalli_17} for the underlying Monge--Amp\`ere equation. 

\begin{proof}[Proof of Theorem~\ref{theo Minkowski solution}]
Let us work in the parametrisation $(\mathbf{P})$ described in Section~\ref{sec (P)}. Given the finite measure $\mu$ on $\Omega^*/\Gamma$, let us introduce the following continuous functional $L_\mu$ on the space of $\tau$-convex domains: for every $\tau$-convex domain $K$ with support function $s$, 
 \begin{equation*}
 L_\mu(K) \coloneqq \covol(K) - \int_{\Omega^*/\Gamma} \frac{s-s_\tau}{\omega_\C} \dif \mu \, . 
 \end{equation*}

\noindent \textbullet\ First, let us prove that $L_\mu$ is coercive for the Hausdorff distance, i.e. if $d_H(K_n,D_\tau)\xrightarrow[n\to+\infty]{}+\infty$, then $L_\mu(K_n)\xrightarrow[n\to+\infty]{}+\infty$. 

Let $(K_n)_{n\in\N}$ be a sequence of $\tau$-convex domains with respective support functions $(s_n)_{n\in\N}$ such that $d_H(K_n,D_\tau)\xrightarrow[n\to+\infty]{}+\infty$. By Corollary~\ref{cor min cosmo time}, we have that $a_n\coloneqq \mathcal{T}_\tau^{\min}(K_n) \xrightarrow[n\to+\infty]{}+\infty$. On one side, by Lemma~\ref{lem ration minmax cosmo time}, there exist a constant $\delta$ such that for $n$ big enough we have 
\begin{equation*}
 \frac{s_n-s_\tau}{\omega_\C} \leq \delta a_n \, ,
\end{equation*}
so that $\int_{\Omega^*/\Gamma} (s_n-s_\tau)/ \omega_\C \, \dif \mu$ grows at most linearly in $a_n$. 

On the other side, as $a_n = \mathcal{T}_\tau^{\min}(K_n)$, we have $ K_n \subseteq D_\tau + a_n \, \S$, so
\begin{equation*}
 \covol(K_n) \geq \covol(D_\tau + a_n \, \S) \, . 
\end{equation*}
By Theorem~\ref{theo Steiner tau-convex}, $\covol(D_\tau + a_n \, \S) = V_{a_n}(D_\tau/\Gamma_\tau)(\Omega^*/\Gamma)$ is a polynomial of degree $d+1$ in $a_n$. Thus, we have that 
\begin{equation*}
 L_\mu(K_n) = \covol(K_n) - \int_{\Omega^*/\Gamma} \frac{s_n-s_\tau}{\omega_\C} \dif \mu \, \xrightarrow[n\to+\infty]{}+\infty. 
 \end{equation*}

\noindent \textbullet\ By {\cite[Theorem~10.9]{Rockafellar_97}}, we can extract a converging subsequence from a minimising sequence for $L_\mu$. Thus, $L_\mu$ attains its minimum at some $\tau$-convex domain $K_\mu$ with support function $s_\mu$. 

\noindent \textbullet\
Let $\bar{f} : \Omega^* \to \R$ be a continuous $\Gamma$-invariant function. Then, $f\coloneqq \bar{f}\,\omega_\C$ is a continuous $\Gamma$-equivariant function. By Theorem~\ref{theo area measure as gateau derivative}, we have
\begin{equation*} 
\frac{\covol (K_{s_\mu+tf})-\covol(K_\mu)}{t} \xrightarrow[t \to 0]{} \int_{\Omega^*/\Gamma} \frac{f}{\omega_\C} \, \dif \Area(K_\mu/\Gamma_\tau) \, . 
 \end{equation*}
Thus, as $K_\mu$ is a minimum of $L_\mu$, we have 
\begin{equation} 
\label{eq ineq limsup at min}
0 \leq \limsup_{t \to 0^+}\frac{L_\mu (K_{s_\mu+tf})-L_\mu(K_\mu)}{t} \leq \int_{\Omega^*/\Gamma} \frac{f}{\omega_\C} \, \dif \Area(K_\mu/\Gamma_\tau) - \liminf_{t \to 0^+} \frac{1}{t}\int_{\Omega^*/\Gamma} \frac{(s_\mu +tf )^{**}-s_\mu}{\omega_\C} \dif \mu \, . 
\end{equation}
As $(s_\mu +tf )^{**} \leq s_\mu+tf$ and $\omega_\C \leq 0$, for all $t \in \R$ we have that, 
\begin{equation*}
\int_{\Omega^*/\Gamma} \frac{(s_\mu +tf )^{**}-s_\mu}{\omega_\C} \dif \mu \geq t\int_{\Omega^*/\Gamma} \frac{f}{\omega_\C} \, \dif \mu \, . 
 \end{equation*}
and thus
\begin{equation}
\label{eq ineq liminf (s+tf)**-s}
\liminf_{t \to 0^+} \frac{1}{t} \int_{\Omega^*/\Gamma} \frac{(s_\mu +tf )^{**}-s_\mu}{\omega_\C} \dif \mu \geq \int_{\Omega^*/\Gamma} \frac{f}{\omega_\C} \, \dif \mu \, . 
 \end{equation}
As $\bar{f} = f/\omega_\C$, inequalities~\eqref{eq ineq limsup at min} and \eqref{eq ineq liminf (s+tf)**-s} imply
\begin{equation*} 
 \int_{\Omega^*/\Gamma} \bar{f} \, \dif \Area(K_\mu/\Gamma_\tau) \geq \int_{\Omega^*/\Gamma} \bar{f} \, \dif \mu \, . 
\end{equation*}
Applying the same reasoning to $-\bar{f}$, we get the reverse inequality. Hence, for every continuous function $\bar{f} : \Omega^*/\Gamma \to \R$, we have
\begin{equation*} 
 \int_{\Omega^*/\Gamma} \bar{f} \, \dif \Area(K_\mu/\Gamma_\tau) = \int_{\Omega^*/\Gamma} \bar{f} \, \dif \mu \, . 
\end{equation*}
By the uniqueness part of the Riesz Representation Theorem, that implies that $\Area(K_\mu/\Gamma_\tau) = \mu$.

\noindent \textbullet\ For the uniqueness of the solution, just notice that, by Proposition~\ref{prop A MA C-convex}, being a solution of the invariant Minkowski problem is the same as having support function satisfying the following Monge--Amp\`ere equation:
\begin{equation*}
\begin{cases}
\MA(s) = (-\omega_\C)^{-1}\mu & \text{in} \ \Omega^* \, ,\\
s_{\vert \partial U} = g_\tau \, .
\end{cases}
\end{equation*}
Then, the uniqueness of the solution is a consequence of the comparison principle for Monge--Amp\`ere equations \cite[Theorem~2.10 and Corollary~2.11]{Figalli_17}. 
\end{proof}

\subsection{Regularity of the solution}

In order to get regularity results concerning the solution of Theorem~\ref{theo Minkowski solution}, we shall use various facts from the Monge--Amp\`ere equation theory. 

Here is a theorem concerning the regularity of strictly convex solutions of Monge--Amp\`ere equations. 

\begin{theorem}[{\cite[Theorem~3.1]{Trudinger_Wang_08}}]
\label{Strictly convex solution is C2}
Let $U \subset \R^d$ be a bounded convex domain, and $\varphi \in C^{1,1}(U)$ be a positive function. Then, any strictly convex solution of $\MA(s) = \varphi \mathcal{L}$ is $ C^{3,\alpha}$ for any $\alpha \in (0,1)$. 

Furthermore, if $\varphi$ is $C^{k,\alpha}$ for some $k\geq2$ and $\alpha \in (0,1)$, then any strictly convex solution of $\MA(s) = \phi \, \mathcal{L}$ is $ C^{k+2,\alpha}$. 
\end{theorem}

\subsubsection{Dimension 2}

In dimension $d=2$, there is a particular strict convexity result for solutions of a Monge--Amp\`ere equation due to Alexandrov and Heinz. 

\begin{theorem}[Alexandrov--Heinz {\cite[Remark 3.2]{Trudinger_Wang_08}}]
\label{Solution is strictly convex}
Let $U \subset \R^2$ be a bounded convex domain and $s:U \to \R$ a convex function such that there exist $m>0$ such that on $U$:
\begin{equation*}
\MA(s) \geq m\mathcal{L} \, . 
\end{equation*} 
Then, $s$ is strictly convex. 
\end{theorem}

Theorems~\ref{Strictly convex solution is C2} and \ref{Solution is strictly convex} imply the following theorem. 

\begin{theorem}
\label{theo C2+ regularity}
Let $U \subset \R^2$ be a bounded convex domain, and $\varphi \in C^{1,1}(U)$ a positive function. Then, any solution of $\MA(s) = \varphi \mathcal{L}$ is $ C^{3,\alpha}$ for any $\alpha \in (0,1)$.

Furthermore, if $\varphi$ is $C^{k,\alpha}$ for some $k\geq2$ and $\alpha \in (0,1)$, then any solution of $\MA(s) = \varphi \mathcal{L}$ is $ C^{k+2,\alpha}$. 
\end{theorem}

Using the parametrisation $(\mathbf{P})$, a $\tau$-convex domain has $\C$-curvature $\phi \in C^0(\Omega^*/\Gamma)$ if and only if its support function is a solution of the following Monge--Amp\`ere equation
\begin{equation*}
\begin{cases}
\MA(s) = (-\omega_\C) \phi^{-1}\vol_{\S} = \phi^{-1} (-\omega_C)^{-d-2} \mathcal{L} & \text{in} \ \Omega^* \, ,\\
s_{\vert \partial \Omega^*} = g_\tau \, . 
\end{cases}
\end{equation*}

Thus, Theorems~\ref{theo Minkowski solution} and \ref{theo C2+ regularity} give the following statement, implying the $d=2$ part of Theorem~\ref{theo intro Minkowski pb in MGHCC reg}. 

\begin{theorem}
In dimension $d=2$, let $\phi$ be a positive $\Gamma$-invariant $C^{k+1}$ function on $\P(\C^*)/\Gamma$ with $k\geq 2$. Then, there is a unique $\tau$-convex domain $K$ with area measure $A(K) = \phi^{-1} \vol_{\S}$. Moreover, its $\C$-spacelike boundary is a strictly convex $C^{k+2}$ hypersurface with $\C$-curvature $\phi$. 
\end{theorem}

\subsubsection{Higher dimension}

In higher dimension, the following result is due to Caffarelli. 

\begin{theorem}[{Caffarelli \cite[Theorem~4.10]{Figalli_17}}]
\label{theo Caffarelli}
Let $U \subset \R^d$ be a bounded convex domain and $s:U \to \R$ a convex function such that there exist $0<m \leq M$ such that on $U$:
\begin{equation*}
 m\mathcal{L} \leq \MA(s) \leq M\mathcal{L} \, . 
\end{equation*}
Fix $x_0 \in U$ and $y_0\in \partial s(x_0)$, consider the affine function $a(x)= (x-x_0) \cdot y_0 + s(x_0)$, and define the convex set $\Sigma \coloneqq \left\lbrace x \in U \st s(x)=a(x) \right\rbrace$. Then, one of the following two properties holds:
\begin{enumerate}
 \item $\Sigma = \{x\}$ (i.e. $s$ is strictly convex at $x$),
 \item $\Sigma$ has no extremal point in $U$. 
\end{enumerate}
\end{theorem}

Here is the equivalent of the Theorem of Alexandrov and Heinz in any dimension.

\begin{theorem}[Multidimensional Alexandrov--Heinz Theorem {\cite[Theorem~2.34]{Bonsante_Fillastre_17}}]
Let $U \subset \R^d$ be a bounded convex domain and $s:U \to \R$ a convex function such that there exist $m>0$ such that on $U$:
\begin{equation*}
\MA(s) \geq m\mathcal{L} \, . 
\end{equation*} 
Then, for all $i \leq d/2$, $s$ cannot be locally affine on a $(d-j)$-plane around some point of $U$. 
\end{theorem}

Theorems~\ref{Strictly convex solution is C2} and \ref{theo Caffarelli} imply the following theorem (see the proof of \cite[Corollary~2.37]{Bonsante_Fillastre_17}). 

\begin{theorem}
\label{theo regularity higher d}
Let $U \subset \R^d$ be a bounded convex domain, $g:\partial U \to \R$ a continuous function convex on faces of $\partial U$ and $\varphi \in C^{k,\alpha}(U)$, with $k\geq 2$, be a positive function. If $s$ is a convex solution of 
\begin{equation*}
\begin{cases}
\MA(s)= \varphi \mathcal{L} & \text{in} \ \Omega^* \, ,\\
s_{\vert \partial U} = g \, . 
\end {cases}
\end{equation*}
and if one of the following conditions is satisfied 
\begin{enumerate}
\item $s<\mathrm{Env(g)} \coloneqq \sup \left\lbrace a\st a \ \text{is affine and} \ a\leq g \ \text{on} \ \partial U\right\rbrace$,
\item for all $i \leq d/2$, the convex function $\mathrm{Env(g)}$ is not locally affine on any $(d-i)$-plane around any point of $U$,
\end{enumerate}
then $s$ is $C^{k+2,\alpha}$. 
\end{theorem}

Theorem~\ref{theo regularity higher d} motivates to introduce the following notion. 

\begin{definition}[Simple maximal domain]
\label{def simple}
The maximal $\tau$-convex domain $D_\tau$ is called \emph{simple} if for all $i \leq d/2$, its support function $s_\tau$ is not locally affine on any $(d-i)$-plane around any point of $\Omega^*$. 
\end{definition}

The following \emph{total mass condition} on a Radon measure $\mu$ on $\Omega^*/\Gamma$, ensures that the boundary $\partial K$ of the solution of the Minkowski problem $\Area(K) = \mu$ does not meet $\partial D_\tau$. 

\begin{proposition} [{\cite[Corollary~3.25]{Bonsante_Fillastre_17}}]
\label{prop m(tau)}
There exists a non-negative number $m(\tau)$ such that for every $\tau$-convex domain $K$, if its total area $\mathrm{Area}(K/\Gamma_\tau) \coloneqq \Area(K/\Gamma_\tau)(\P(\C^*)/\Gamma)$ satisfies $\mathrm{Area}(K/\Gamma_\tau)>m(\tau)$ then $K$ is a Cauchy $\C$-convex domain, i.e. it is such that $\partial_{sp} K = \partial K \subset D_\tau$. 
\end{proposition}
\begin{proof}
First let us prove that there exists a non-negative number $m(\tau)$ such that every $\tau$-convex domain $K$ such that $\mathrm{Area}(K/\Gamma_\tau) > m(\tau)$ satisfies $\partial K \subset D_\tau$. By contradiction, assume there is a sequence $(K_n)_{n\in\N}$ of $\tau$-convex domains such that for all $n\in \N$, $\partial K_n$ meets $\partial D_\tau$ and $\mathrm{Area}(K_n/\Gamma_\tau) >n$. Then, by Lemma~\ref{lem exctraction of tau-convex domains}, we can extract a subsequence $(K_{\varphi(n)})_{n\in\N}$ converging to a $\tau$-convex domain $K_\infty$ for the Hausdorff distance. Proposition~\ref{prop weak cv and Hausdorff} implies that 
 \begin{equation*}
 \varphi(n) = \mathrm{Area}(K_{\varphi(n)}/\Gamma_\tau) \xrightarrow[n \to +\infty]{} \mathrm{Area}(K_\infty/\Gamma_\tau) \, ,
 \end{equation*}
 which is a contradiction. 

Then, for the $\partial_{sp} K = \partial K $ part, let $K$ be a $\tau$-convex domain having a point $P \in \partial K \setminus \partial_{sp}K $. A supporting hyperplane at $P$ of $K$ is also a supporting hyperplane of $D_\tau$ (that can be seen through support functions in $(\mathbf{P})$: $s_K$ and $s_\tau$ coincide on $\partial \Omega^*$) and thus, as $K\subset D_\tau$, we have $P \in \partial D_\tau$, so that $\partial K \not\subset D_\tau$. 
\end{proof}

Proposition~\ref{prop m(tau)} directly implies Theorem~\ref{theo intro Minkowski pb in MGHCC measure}, while Theorems~\ref{theo Minkowski solution} and \ref{theo regularity higher d}, and Proposition~\ref{prop m(tau)} give the following statement, implying the $d \geq 3$ part of Theorem~\ref{theo intro Minkowski pb in MGHCC reg}. 

\begin{theorem}
Let $\phi$ be a positive $\Gamma$-invariant $C^{k+1}$ function on $\P(\C^*)$ with $k\geq 2$, and such that $\phi<c(\tau)$, where 
\begin{equation*}
 c(\tau) \coloneqq \frac{\vol_{\S}(\Omega^*/\Gamma)}{m(\tau)} \, ,
\end{equation*}
with $m(\tau)$ the constant given by Proposition~\ref{prop m(tau)}. Then, there exist a unique $\tau$-convex domain $K$ with area measure $A(K) = \phi^{-1} \vol_{\S} \,(\P(\C^*)/\Gamma)$. Moreover, its boundary is a strictly convex $\C$-spacelike hypersurface of class $C^{k+2}$ with $\C$-curvature $\phi$. 

When $D_\tau$ is simple, the constant $c(\tau)$ can be taken to be $0$. 
\end{theorem}

\appendix

\section{Dirichlet--Lee fundamental domains}
\label{app DL domain}

In this first appendix, we show that the action of an affine deformation of a group dividing a cone on its maximal domain admits a convex fundamental domains. This fact is used in the proof of the strict convexity of the covolume (Theorem~\ref{theo convexity covolume}). 

Let $\Gamma_\tau < \SA (\A^{d+1})$ be an affine deformation of a subgroup $\Gamma < \SL (\V^{d+1})$ dividing a proper open convex cone $\C$. The construction of a convex fundamental domain for the action of $\Gamma_\tau$ on $D_\tau$, the maximal $\tau$-convex domain is due to Lee \cite{Lee_08} (Lee actually only considered the case where $\tau=0$, but the construction immediately extends to our frame of work). Another study of these domains (again in the case $\tau=0$) using \emph{Vinberg hypersurfaces} instead of affine spheres, can be found in \cite{Marquis_12}. 

\begin{definition}[Dirichlet--Lee domain]
\label{def DL domain}
Let $[\varphi]\in \P(\C^*)$. The \emph{Dirichlet--Lee domain of $D_\tau$ given by $[\varphi]$} is the convex set $\DL_{[\varphi]}$, closed in $D_\tau$, defined as
\begin{equation*}
 \DL_{[\varphi]} = \left\lbrace P \in D_\tau \st \forall \gamma_\tau \in \Gamma_\tau, \ \varphi\vp{\overrightarrow{P(\gamma_\tau P)}}\leq 0\right\rbrace \, ,
\end{equation*}
where $\overrightarrow{P(\gamma_\tau P)} \in \V^{d+1}$ is the vector from $P \in \A^{d+1}$ to $\gamma_\tau \cdot P\in \A^{d+1}$.
\end{definition}

In the parametrisation $(\mathbf{P})$, if $y\in \Omega^*$ and $Y\coloneqq (y,-1) \in \C^*$. The Dirichlet--Lee domain given by $y$ is 
\begin{equation*}
 \DL_{y} = \left\lbrace X \in D_\tau \st \forall (\gamma,\tau) \in \Gamma_\tau, \ X\cdot Y \geq (\gamma X + \tau ) \cdot Y\right\rbrace \, . 
\end{equation*}

\begin{lemma}
\label{lem scalar product to - infty}
Let $[\varphi] \in \P(\C^*)$ and $A$ be a compact subset of $D_\tau$. If $(\gamma^n_\tau)_{n\in \N}$ is a diverging sequence of elements of $\Gamma_\tau$, then for all $P \in \A^{d+1}$,
\begin{equation*}
 \sup_{Q \in A} \varphi\vp{ \overrightarrow{P(\gamma^n_\tau Q)}} \xrightarrow[n \to +\infty]{} -\infty \, . 
\end{equation*}
\end{lemma}
\begin{proof}
Let us work in the parametrisation $(\mathbf{P})$. We now let $Y=(y,-1) \in \C^*$ and $(\gamma_n, \tau_n)_n$ be a diverging sequence of elements of $\Gamma_\tau < \SL(\R^{d+1}) \ltimes \R^{d+1}$. We want to show that
\begin{equation*}
 \sup_{X \in A} (\gamma_n X + \tau_n)\cdot Y \xrightarrow[n \to +\infty]{} -\infty \, . 
\end{equation*}

Using \cite[Proposition~6.7]{Ablondi_25}, let us decompose any element $X \in A $ as $X = \Pi(X) + t N(X)$, where $ \Pi(X) \in \partial D_\tau$, $t>0$, and $N(X) \in \S$. By \cite[Proposition~6.9]{Ablondi_25}, $N(A)$ is compact as the image of the compact set $A$ by a continuous map. Then, for all $X \in A$, we have,
\begin{equation}
\label{eq1 proof inf on compact}
(\gamma_n X + \tau_n)\cdot Y = (\gamma_n \Pi(X) + \tau_n)\cdot Y + \gamma_n N(X)\cdot Y \, . 
\end{equation}
As $P \in \partial D_\tau$, by \cite[Proposition~5.15]{Ablondi_25}, we have
\begin{equation}
\label{eq2 proof inf on compact}
(\gamma_n \Pi(X) + \tau_n)\cdot Y = (\gamma_n \Pi(X) + \tau_n)\cdot (y,-1) \leq g_\tau(y) \, . 
\end{equation}
As the affine sphere $\S$ is $\Gamma$-invariant, for all $n \in \N$, we have $\gamma_n N(A) \subset \S$. As the action of $\Gamma$ on $\S$ is properly discontinuous and the sequence $(\gamma_n,\tau_n)_{n\in \N}$ diverges, the compact sets $(\gamma_n N(A))_{n\in \N}$ escape from every compact subset of $\S$. Then, as $\S$ is asymptotic to $\partial \C$ and $Y \in \C^*$, we have
\begin{equation}
\label{eq3 proof inf on compact}
\sup_{X \in A} \gamma_n N(X) \cdot Y \xrightarrow[n \to +\infty]{} -\infty \, . 
\end{equation}
That concludes the proof, as equations~\eqref{eq1 proof inf on compact}, \eqref{eq2 proof inf on compact} and \eqref{eq3 proof inf on compact} give
\begin{equation*}
 \sup_{X \in A} (\gamma_n X + \tau_n)\cdot Y \leq g_\tau(y)+ \sup_{X \in A} \gamma_n N(X) \cdot Y\xrightarrow[n \to +\infty]{} -\infty \, . \qedhere
\end{equation*}
\end{proof}

\begin{corollary}
\label{cor1 fundamental domain}
Let $y\in \Omega^*$. The Dirichlet--Lee domain $\DL_{[\varphi]}$ of $D_\tau$ satisfies
\begin{equation*}
 D_\tau =\bigcup_{\gamma_\tau \in \Gamma_\tau} (\gamma_\tau \cdot \DL_{[\varphi]}) \, . 
\end{equation*}
\end{corollary}
\begin{proof}
Let $P \in D_\tau$. Lemma~\ref{lem scalar product to - infty} applied with $A=\{P\}$ and the proper discontinuity of the action of $\Gamma$ on $\P(\C^*)$, imply that \begin{equation*}
\sup_{\gamma_\tau \in \Gamma_\tau}\varphi\vp{\overrightarrow{P(\gamma_\tau P)}}
\end{equation*}
is achieved for some $\gamma_\tau^P \in \Gamma_\tau$. Hence, $P \in \gamma_\tau \cdot \DL_{[\varphi]}$, with $\gamma_\tau = (\gamma_\tau^P)^{-1} \in \Gamma_\tau$. 
\end{proof}

\begin{lemma}
\label{lem int of Dirichlet--Lee domain}
Let $[\varphi]\in \P(\C^*)$. The Dirichlet--Lee domain $\DL_{[\varphi]}$ of $D_\tau$ has interior
\begin{equation*}
 \mathrm{int}(\DL_{[\varphi]}) = \left\lbrace P \in D_\tau \st \forall \gamma_\tau \in \Gamma_\tau\setminus \{\Id\}, \ \varphi\vp{\overrightarrow{P(\gamma_\tau P)}} < 0\right\rbrace ,
\end{equation*}
where $\Id$ is the identity in $\SA (\A^{d+1})$. 
\end{lemma}
\begin{proof}
The direct inclusion $\subseteq$ is trivial. Let us then prove the reverse inclusion.

Let us work in the parametrisation $(\mathbf{P})$. We now let $Y=(y,-1) \in \C^*$ and we want to show that 
\begin{equation*}
\left\lbrace X \in D_\tau \st \forall (\gamma,\tau) \in \Gamma_\tau\setminus \{(I,0)\}, \ X\cdot Y > (\gamma X + \tau ) \cdot Y\right\rbrace \subseteq \, \mathrm{int}(\DL_{y}) \,.
\end{equation*}
Let $X \in D_\tau$ such that for all $(\gamma,\tau) \in \Gamma_\tau\setminus \{(I,0)\}$, we have $X\cdot Y > (\gamma X + \tau ) \cdot Y$. Let $A$ be a compact neighbourhood of $X$ in $D_\tau$, such that 
\begin{equation}
\label{eq1 lem int}
\inf_{X' \in A} X' \cdot Y \geq X \cdot Y -1 \, . 
\end{equation}
Lemma~\ref{lem scalar product to - infty} and the proper discontinuity of the action of $\Gamma_\tau$ on $D_\tau$ imply that there is only a finite family of elements $(\gamma_i,\tau_i)_{1 \leq i \leq n}$ of $\Gamma_\tau \setminus\{(I,0)\}$ such that for all $(\gamma,\tau) \in \Gamma_\tau \setminus\{(I,0),(\gamma_1,\tau_1), \dots, (\gamma_n,\tau_n)\}$,
\begin{equation}
\label{eq2 lem int}
X \cdot Y - 1 > \sup_{X' \in A} (\gamma X' + \tau) \cdot Y \, . 
\end{equation}
Thus, using equations~\eqref{eq1 lem int} and \eqref{eq2 lem int}, for all $X' \in A$ and $(\gamma,\tau) \in \Gamma_\tau \setminus\{(I,0),(\gamma_1,\tau_1), \dots, (\gamma_n,\tau_n)\}$,
\begin{equation}
\label{eq3 lem int}
 X' \cdot Y > (\gamma X' + \tau ) \cdot Y \, . 
\end{equation}
Now, for all $1 \leq i \leq n$, we have, by assumption, that $X\cdot Y > (\gamma_i X + \tau_i ) \cdot Y$, so there is a neighbourhood $U_i \subseteq A$ of $X$ such that for all $X' \in U_i$,
\begin{equation}
\label{eq4 lem int}
X' \cdot Y >(\gamma_i X' + \tau_\gamma) \cdot Y \, . 
\end{equation}
Equations~\eqref{eq3 lem int} and \eqref{eq4 lem int} imply the neighbourhood $U \coloneqq \bigcap_{1 \leq i \leq n} U_i \subseteq A$ of $X$ in $D_\tau$ satisfies that for all $X' \in U$ and $(\gamma,\tau) \in \Gamma_\tau$, 
\begin{equation*}
 X' \cdot Y \geq (\gamma X' + \tau ) \cdot Y \, . 
\end{equation*}
That means that $X$ admits a neighbourhood included in $\DL_{y}$, i.e. $X \in \mathrm{int}(\DL_{y})$. 
\end{proof}

\begin{corollary}
\label{cor2 fundamental domain}
Let $[\varphi]\in \P(\C^*)$. The Dirichlet--Lee domain $\DL_{[\varphi]}$ of $D_\tau$ satisfies that for all $(\gamma_\tau) \in \Gamma_\tau \setminus\{\Id\}$, 
\begin{equation*}
\mathrm{int}(\DL_{[\varphi]}) \cap \bp{ \gamma_\tau \cdot \mathrm{int}(\DL_{[\varphi]})} = \emptyset \, . 
\end{equation*}
\end{corollary}

Corollaries~\ref{cor1 fundamental domain} and~\ref{cor2 fundamental domain} give the following Theorem. 

\begin{theorem}
\label{theo Dirichlet--Lee domain}
Let $[\varphi]\in \P(\C^*)$. The Dirichlet--Lee domain $\DL_{[\varphi]}$ is a convex fundamental domain of $D_\tau$ for the action of $\Gamma_\tau$. 
\end{theorem}

Finally, let us end this appendix by noticing a nice property satisfied by Dirichlet--Lee domains, which is useful in the proof of the strict convexity of the covolume. 

\begin{lemma}
\label{lem for covolume cvx proof}
Let $[\varphi]\in \P(\C^*)$ and $\DL_{[\varphi]}$ be the Dirichlet--Lee domain of $D_\tau$ given by it. For all $P \in \DL_{[\varphi]}$ and $t>0$, we have that $ P + tN_{\S}[\varphi] \in \mathrm{int}(\DL_{[\varphi]})$. 
\end{lemma}
\begin{proof}
Let $P \in \DL_{[\varphi]}$, $t >0$ and $\gamma_\tau \in \Gamma_\tau \setminus\{\Id\}$. Let us denote $\vec{n} \coloneqq N_{\S}[\varphi] \in \S$. As $P \in \DL_{[\varphi]}$, we have 
\begin{equation}
\label{ineq vector 1}
 \varphi\vp{\overrightarrow{P(\gamma_\tau P)}}\leq 0 \, . 
\end{equation}
Moreover, $\gamma \cdot \vec{n} \in \S$ and $[\varphi]$ directs the supporting hyperplane to $\S$ at $\vec{n} = N_{\S}[varphi]$, so, as $\gamma N \neq N$ and the affine sphere is a locally uniformly convex hypersurface, for $\gamma \in \Gamma \setminus \{\Id\}$ we have
\begin{equation}
\label{ineq vector 2}
 \varphi(\gamma \cdot \vec{n}) < \varphi(\vec{n}). 
\end{equation}
Let $t >0$ and set $P' = P + t\vec{n} =P + t N_{\S}[\varphi]$. For all $\gamma_\tau \in \Gamma_\tau\setminus\{\Id,\}$, we have
\begin{align*}
\overrightarrow{P'(\gamma_\tau P')}& = \overrightarrow{P'P}+ \overrightarrow{P(\gamma_\tau P)}+\overrightarrow{(\gamma_\tau P)(\gamma_\tau P')} \\
& = \overrightarrow{P(\gamma_\tau P)} + t(\gamma \cdot \vec{n}) -t\vec{n} \, .
\end{align*}
As $t>0$, with \eqref{ineq vector 1} and \eqref{ineq vector 2} that implies that 
\begin{equation*}
 \varphi\vp{\overrightarrow{P(\gamma_\tau P)}} < 0 \, . 
\end{equation*}
By Lemma~\ref{lem int of Dirichlet--Lee domain}, that means that for all $t>0$, $P'= P + t N_{\S}[\varphi] \in \mathrm{int}(\DL_{[\varphi]})$. 
\end{proof}

\section{\texorpdfstring{Smooth approximation of $\tau$-convex domains}{Smooth approximation of tau-convex domains}}
\label{app C2+-approximation}

In this second appendix, we provide the proof of a $C^2_+$-approximation lemma for $\C$-convex domains, which we use in the proof of Lemma~\ref{lem ineq covol}, in order to get Theorem~\ref{theo area measure as gateau derivative} and Theorem~\ref{theo Minkowski solution}. 

The following approximation result is due to Choi \cite[Theorem~4.4.4]{Choi_25}, who managed to adapt the proof of a classical convex geometry approximation result for bounded convex domains in the Euclidean affine space, presented in \cite[Proposition~2.1]{Ghomi_04}. For the sake of completeness, we give a quick proof of it, using the same language and notation as in the rest of this article. 

\begin{lemma}
\label{lem smooth approx}
Let $K$ be a $\tau$-convex. For all $\varepsilon>0$, there exists a $C^2_+$ $\tau$-convex domain $K'$ such that $d_H(K,K') \leq \varepsilon$. 
\end{lemma}

\begin{proof}

Let us work in the parametrisation $(\mathbf{P})$ and let $s$ be the support function of $K$. For any support function $f$, let us denote by $K(f)$ the $\C$-convex domain with support function $f$. 

First, note that $K$ can already be approximated by the $\tau$-convex domains $K(s+t\omega_\C)=K+t\S$, with $t>0$. Thus, we can assume that we are working with $K$ a $\C$-convex domain with a strictly convex support function $s$, and such that $\partial_{sp} K =\partial K \subset D_\tau$ \cite[Lemma~6.4]{Ablondi_25} and $\partial K /\Gamma_\tau$ is compact \cite[Lemma~7.6]{Ablondi_25}. 

Let $\varepsilon>0$. Let us now show that there exists a $\tau$-convex domain $K'$ with smooth and locally uniformly convex boundary, and such that $K(s+\varepsilon \omega_\C) \subseteq K' \subseteq K$. 

Let $g:\R \to \R$ be smooth function, locally uniformly convex on $(0,+\infty)$ and vanishing on $(-\infty,0]$. For instance, consider
\begin{equation*}
 \function{g}{\R}{\R}{t}{\begin{cases}
 t^2 \exp\vp{-\frac{1}{t^2}} & \text{if} \ x>0 \, , \\
 0 & \text{otherwise. }
 \end{cases}}
\end{equation*}
Then, for $Y = -(y,-1)/\omega_\C(y) \in \Sigma_\C^*$, set $\phi_Y$ to be the following smooth convex function:
\begin{equation*}
\function{\phi_Y}{K}{\R_+}{X}{g\bp{X\cdot Y-(\tilde s+\varepsilon\tilde\omega_\C)(Y)}} \, ,
\end{equation*}
where $\tilde s$ and $\tilde \omega_\C$ are the $\C^*$-support function of respectively $K$ and $\S$. Note that the function $\phi_Y$ is smooth and convex as the composition of $g$ and an affine map, and that it vanishes on $K(s+\varepsilon \omega_\C)$ but is positive in the $\C$-past half-space $\left\lbrace X \in \R^{d+1} \st X\cdot(y,-1)>(s+\varepsilon\omega_\C)(y)\right\rbrace$. 

Let $X \in \overline{K} \subset D_\tau$. By \cite[Proposition~4.22]{Ablondi_25}, for all $y \in \partial\Omega^*$, we have $ X\cdot(y,-1)< s_\tau(y)$ . Hence, as $s+\varepsilon \omega_\C$ continuously extend to $s_\tau$ on $\partial \Omega^*$, the subset $\left\lbrace y \in \Omega^* \st X\cdot(y,-1)\geq(s+\varepsilon\omega_\C)(y) \right\rbrace$ is closed and bounded, thus compact.

Thus, we can set $\chi$ to be the function defined by
\begin{equation*}
\function{\chi}{\overline{K}}{\R_+}{X}{\int_{\S^*}\phi_Y(X) \, \dvol_{\S}(Y)} \, ,
\end{equation*}
where $\dvol_{\S}$ is the canonical volume from on $\S^* \simeq\S$ described in subsection~\ref{subsec Cheng-Yau affine sphere}.

Note that $\chi^{-1}(0) = \overline{K(s+\varepsilon \omega_\C)}$ and that $\chi$ is $\Gamma_\tau$-invariant: for all $(\gamma,\tau) \in \Gamma_\tau$ and $X \in K$, using the $\Gamma_\tau$-equivariance satisfied by the total support function $\tilde s+\varepsilon\tilde\omega_\C : \C^* \to \R$, we have
\begin{align*}
 \phi_Y(\gamma X+\tau) & = g\bp{\gamma X\cdot Y + \tau \cdot Y-(\tilde s+\varepsilon\tilde\omega_\C)(Y)} \\
 & = g\bp{ X\cdot(\gamma^{-1} \star Y) + \tau \cdot Y-(\tilde s+\varepsilon\tilde\omega_\C)(\gamma^{-1} \star Y) - \tau \cdot Y} \\
 & = g\bp{ X\cdot(\gamma^{-1} \star Y)-(\tilde s+\varepsilon\tilde\omega_\C)(\gamma^{-1} \star Y)} \\
 & = \phi_{\gamma^{-1} \star Y}(X) \, ,
\end{align*}
where $\star$ denotes the dual action of $\Gamma$. 

By $\Gamma$-invariance of the volume form $\dvol_{\S}$ on the affine sphere $\S^*$, for all $(\gamma,\tau) \in \Gamma_\tau$ and $X \in K$, we have
\begin{align*}
 \chi(\gamma X+\tau) & = \int_{\S^*}\phi_{\gamma^{-1} \star Y}(X) \, \dvol_{\S}(Y) = \int_{\S^*}\phi_{Y}(X) \, \dvol_{\S}(\gamma \star Y) \\
 & =\int_{\S^*}\phi_{Y}(X) \, \dvol_{\S}(Y) \\
 & = \chi(X) \, . 
\end{align*}
As we have assumed that $\partial K / \Gamma_\tau$ is compact, and because of the $\Gamma_\tau$-invariance of $\chi$, there is an $r>0$ such that $\chi(\partial K) \subseteq [r,+\infty)$. 

The smoothness and convexity of the functions $\phi_Y$ clearly imply the smoothness and convexity of $\chi$. We claim that it is actually locally uniformly convex on $\overline{K} \setminus \overline{K(s+\varepsilon \omega_\C)}$, concluding the proof as then, by $\Gamma_\tau$-invariance of $\chi$, the level set $K' \coloneqq (\chi < r/2)$ is a $\tau$-convex domain with smooth and locally uniformly convex boundary, satisfying $K(s+\varepsilon \omega_\C) \subseteq K' \subseteq K$. 

Let us now prove our claim. Let $X \in K \setminus \overline{K(s+\varepsilon \omega_\C)}$ and $V \in \R^{d+1}\setminus\{0\}$. We shall show that $\Hess \chi(X)(V,V)>0$. As $X \notin \overline{K(s+\varepsilon \omega_\C)}$, there exists a $y_X \in \Omega^*$ directing a supporting hyperplane to $K(s+\varepsilon \omega_\C)$ and separating it from $X$, i.e. such that $X\cdot(y_X,-1) > (s+\varepsilon\omega_\C)(y_X)$. Thus, there exist a neighbourhood $U_X \subset \Omega^*$ of $y_X$ and a $\delta_X >0$ such that for all $y\in U_X$ and $t\in (-\delta_X,\delta_X)$,
\begin{equation*}
 (X+tV)\cdot(y,-1) > (s+\varepsilon\omega_\C)(y). 
\end{equation*}
Then, by local uniform convexity of $g$ on $(0,+\infty)$, for all $y\in U_X$,
\begin{equation*}
 \Hess \phi_Y (X) (V,V) = \left. \frac{\dif^2}{\dif t^2}\right\vert_{t=0} \phi_Y(X + tV)>0 \, . 
\end{equation*}
Finally, as $t\mapsto \phi_Y(X+tV)$ is convex we have that for all $y\in \Omega^*$,
\begin{equation*}
 \left. \frac{\dif^2}{\dif t^2}\right\vert_{t=0} \phi_Y(X+tV)\geq 0 \, ,
\end{equation*}
and thus
\begin{align*}
 \Hess \chi (X) (V,V) & = \left. \frac{\dif^2}{\dif t^2}\right\vert_{t=0} \chi(X + tV) = \int_{\Omega^*}\vp{\left. \frac{\dif^2}{\dif t^2}\right\vert_{t=0}\phi_Y(X+tV)}\dvol_{\S}(y) \\
 & \geq \int_{U_X} \vp{ \left. \frac{\dif^2}{\dif t^2}\right\vert_{t=0}\phi_Y(X+tv)} \dvol_{\S}(y) > 0 \, . \qedhere
\end{align*}
\end{proof}

\end{spacing}

\bibliography{AffMin}
\bibliographystyle{alpha}

\end{document}